\documentclass[11pt,letter]{amsart}
\usepackage{amsmath, amssymb,amsthm}
\usepackage{enumerate,hyperref} 
\usepackage{bbm}
\usepackage{mathrsfs}
\usepackage[all]{xy}
\usepackage{color}

\usepackage[inline]{enumitem}

\newtheorem{thm}{Theorem}[section]
\newtheorem{theorem}[thm]{Theorem}
\newtheorem{prop}[thm]{Proposition}
\newtheorem{proposition}[thm]{Proposition}
\newtheorem{lemma}[thm]{Lemma}
\newtheorem{corollary}[thm]{Corollary}

\newtheorem*{theoremletter}{Main Theorem}
\newtheorem*{conjecturetz}{Trivial Zero Conjecture}

\theoremstyle{definition}
\newtheorem{defn}[thm]{Definition}
\newtheorem{definition}[thm]{Definition}
\newtheorem{remark}[thm]{Remark}

\newcommand{\A}{\mathbb{A}}
\newcommand{\C}{\mathbb{C}}
\newcommand{\F}{\mathbb{F}}
\newcommand{\G}{\mathbb{G}}
\newcommand{\Q}{\mathbb{Q}}
\newcommand{\R}{\mathbb{R}}
\newcommand{\T}{\mathbb{T}}
\newcommand{\Z}{\mathbb{Z}}

\newcommand{\cC}{\mathcal{C}}
\newcommand{\cD}{\mathcal{D}}
\newcommand{\cE}{\mathcal{E}}
\newcommand{\cF}{\mathcal{F}}
\newcommand{\cG}{\mathcal{G}}
\newcommand{\cH}{\mathcal{H}}
\newcommand{\cI}{\mathcal{I}}

\newcommand{\cL}{\mathcal{L}}
\newcommand{\cM}{\mathcal{M}}
\newcommand{\cO}{\mathcal{O}}

\newcommand{\cR}{\mathcal{R}}
\newcommand{\cT}{\mathcal{T}}
\newcommand{\cU}{\mathcal{U}}
\newcommand{\cV}{\mathcal{V}}
\newcommand{\cW}{\mathcal{W}}
\newcommand{\cX}{\mathcal{X}}

\newcommand{\gc}{\mathfrak{c}}
\newcommand{\gd}{\mathfrak{d}}
\newcommand{\gf}{\mathfrak{f}}
\renewcommand{\gg}{\mathfrak{g}}
\newcommand{\gh}{\mathfrak{h}}
\newcommand{\gk}{\mathfrak{k}}
\newcommand{\gm}{\mathfrak{m}}
\newcommand{\gn}{\mathfrak{n}}
\newcommand{\gu}{\mathfrak{u}}

\newcommand{\sw}{{\sf{w}}}

\newcommand{\Cl}{\mathscr{C}\!\ell_F^+}

\newcommand{\End}{\operatorname{End}}
\newcommand{\Gal}{\operatorname{Gal}}
\newcommand{\GL}{\operatorname{GL}}
\newcommand{\Res}{\operatorname{Res}}
\newcommand{\pr}{\operatorname{pr}}
\newcommand{\rH}{\operatorname{H}}
\newcommand{\rRG}{\operatorname{R\Gamma}}
\newcommand{\SL}{\operatorname{SL}}
\newcommand{\Sym}{\operatorname{Sym}}

\newcommand{\bs}{\!\setminus}

\newcommand{\an}{\operatorname{an}}
\newcommand{\GB}{\operatorname{GB}}
\newcommand{\FM}{\operatorname{FM}}
\newcommand{\coker}{\operatorname{coker}}
\newcommand{\dlog}{\operatorname{dlog}}

\newcommand{\new}{\operatorname{new}}
\newcommand{\id}{\operatorname{id}}
\newcommand{\BM}{\operatorname{BM}}
\newcommand{\coinv}{\operatorname{coinv}}
\newcommand{\Ad}{\operatorname{Ad}}
\newcommand{\ord}{\operatorname{ord}}

\newcommand{\cyc}{\operatorname{cyc}}

\newcommand{\tw}{\operatorname{tw}}
\newcommand{\triv}{\operatorname{triv}}
\newcommand{\ev}{\operatorname{ev}}
\newcommand{\eval}{\operatorname{eval}}

\newcommand{\Sp}{\operatorname{Sp}}

\newcommand{\cusp}{\operatorname{cusp}}
\newcommand{\Aut}{\operatorname{Aut}}
\newcommand{\sgn}{\operatorname{sign}}
\newcommand{\Det}{\operatorname{Det}}

\newcommand{\Ind}{\operatorname{Ind}}
\newcommand{\rig}{\operatorname{rig}}
\newcommand{\Frob}{\operatorname{Frob}}

\newcommand{\St}{\operatorname{St}}
\newcommand{\Fil}{\operatorname{Fil}}
\newcommand{\st}{\operatorname{st}}
\newcommand{\dR}{\operatorname{dR}}
\newcommand{\Tr}{\operatorname{Tr}}

\newcommand{\Hom}{\operatorname{Hom}}
\renewcommand{\Re}{\operatorname{Re}}

\newcommand{\ontop}[2]{\genfrac{}{}{0pt}{}{#1}{#2}}

\newenvironment{psmallmatrix}{\left(\begin{smallmatrix}}{\end{smallmatrix}\right)}

\title[$p$-adic $L$-functions of Hilbert cusp forms and the trivial zero conjecture]{$p$-adic $L$-functions of Hilbert cusp forms \\ and the trivial zero conjecture}

\author{Daniel Barrera, Mladen Dimitrov \and Andrei Jorza}

\address{Universidad de Santiago de Chile, Alameda 3363, 9160000 Estaci\'on Central, Santiago, Chile}
\email{daniel.barrera.s@usach.cl}

\address{Universit\'e de Lille, CNRS, UMR 8524 -- Laboratoire Paul Painlev\'e, 59000 Lille, France}
\email{mladen.dimitrov@univ-lille.fr}

\address{University of Notre Dame, 275 Hurley Hall, Notre Dame, IN 46556, USA}
\email{ajorza@nd.edu}


\begin{document}

\begin{abstract}
We prove a strong form of the trivial zero conjecture at the central point for the 
$p$-adic $L$-function of a non-critically refined self-dual cohomological cuspidal automorphic
representation of $\GL_2$ over a totally real field, which is Iwahori spherical at places above $p$. 

In the case of a simple zero we adapt the approach of Greenberg and Stevens, based on the functional 
equation for the $p$-adic $L$-function of a nearly finite slope family and on 
improved $p$-adic $L$-functions that we construct using automorphic symbols and overconvergent cohomology. 

For higher order zeros we develop a conceptually new approach
studying the variation of the root number in partial families and establishing the vanishing of many Taylor coefficients of the $p$-adic $L$-function of the family. 
\end{abstract}

\maketitle
\addtocontents{toc}{\setcounter{tocdepth}{0}}

\section*{Introduction}

The complex analytic $L$-function of an algebraic cuspidal automorphic representation $\pi$ on a reductive
group over a number field $F$ lies at the heart of the Langlands program, and the relationship between its analytic properties, namely the order of vanishing at critical points, and the arithmetic of the conjecturally attached $p$-adic representation $V_{\pi}$ of the absolute Galois group $\mathrm{G}_F$ is the content of the famous Bloch-Kato conjectures. 
Iwasawa theory, in turn, seeks to relate the arithmetic of the restriction of $V_{\pi}$ to the $p$-adic cyclotomic 
extension of $F$, and the behavior of the $p$-adic analytic $L$-function $L_p(\pi,s)$ of $\pi$.
The existence of $p$-adic $L$-functions for automorphic representations and families thereof is a challenging problem in itself, but even when they have been constructed, properties such as the location of their zeros and orders of vanishing have remained poorly understood. 

To ensure good analytic properties in the cyclotomic variable $s$,  $L_p(\pi,s)$ contains 
extra interpolation factors which can possibly vanish at a critical integer. Such zeros,
 called trivial, were first considered for 
an elliptic curve $E$ over $\Q$ in the seminal work of Mazur, Tate and Teitelbaum
\cite{mazur-tate-teitelbaum}. If $E$ has split multiplicative reduction at $p$, the
$p$-adic $L$-function $L_p(E,s)$ has a trivial zero at $s=1$ and it was
conjectured, and later proven by Greenberg and Stevens \cite{greenberg-stevens}, that 
\[L_p'(E,1) = \mathscr{L}(E) \cdot \frac{L(E,1)}{\Omega_{E}},\]
where $\Omega_{E}$ is the real period of $E$ and $\mathscr{L}(E)=\tfrac{\log_p q_{E}}{\ord_p q_{E}}$ is the so-called $\mathscr{L}$-invariant,  $q_{E}$ being the Tate period of $E$. While trivial zeros of 
$p$-adic $L$-functions and their $\mathscr{L}$-invariants were considered by Mazur, Tate and Teitelbaum in their quest to formulate a $p$-adic analogue of the Birch and Swinnerton-Dyer conjecture, various recent works on the Bloch-Kato conjecture rely
crucially on $p$-adic $L$-functions and the Iwasawa main conjecture.
In the context of geometric Galois representations the following more general, albeit somewhat vague, trivial zero conjecture
 springs from various places in the literature and is part of the ``folklore''. 

Let $V$ be a $p$-adic representation of $\mathrm{G}_{\Q}$, critical in the sense of Deligne,  such that $V_p=V|_{\mathrm{G}_{\Q_p}}$ is semi-stable. Let $D\subset \cD_{\st}(V_p)$ be a regular submodule in the sense of Perrin-Riou \cite{perrin-riou:Lp}. The works of Coates and Perrin-Riou posit the existence of a $p$-adic $L$-function $L_p(V,D,s)$ satisfying an interpolation formula of the form
$L_p(V,D,0)=\Omega_V^{-1}L(V,0)\cE(V_p,D)$, where $\Omega_V$ is a Deligne period, 
$L(V,s)$ is the complex $L$-function and $\cE(V_p,D)$ is a product of linear
Euler factors. 

\begin{conjecturetz}
Letting $e$ denote the number of
vanishing Euler factors in $\cE(V_p,D)$ and  $\cE^+(V_p,D)$ the product of the remaining non-vanishing ones, the $p$-adic $L$-function $L_p(V,D,s)$
vanishes to order at least $e$ at $s=0$ and
\[L_p^{(e)}(V,D,0)=e! \cdot \mathscr{L}(V,D)\cdot \cE^+(V_p,D) \cdot \frac{L(V,0)}{\Omega_V} 
, \]
where  $ \mathscr{L}(V,D)$ is an arithmetic $\mathscr{L}$-invariant.
More generally, by the Trivial Zero Conjecture for a geometric representation $V$ of $\mathrm{G}_F$ and a collection of regular submodules $D_v\subset \cD_{\st}(V_{|\mathrm{G}_{F_v}})$ for $v\mid p$ we mean the one for 
$\Ind_{\mathrm{G}_F}^{\mathrm{G}_\Q}V$ and the regular submodule $\mathop{\oplus}\limits_{v\mid p}\Ind_{\mathrm{G}_{F_v}}^{G_{\mathbb{Q}_p}}D_v$.
\end{conjecturetz}
Precise formulations of the Trivial Zero Conjecture exist in a number of restricted settings. In the case when $V$ is crystalline at $p$ Greenberg and Benois made explicit the conjectural interpolation factor $\cE(V_p,D)$. Moreover, Greenberg in the case of ordinary representations, and Benois in the semi-stable case, have defined arithmetic $\mathscr{L}$-invariants Galois cohomologically, 
when $V$ satisfies a number of technical hypotheses (S, U, T of \cite{greenberg:trivial-zeros} and (C1)-(C5) of \cite{benois:L-invariant}).

\smallskip

This article is devoted to proving the Trivial Zero Conjecture
at the central point $s=0$, with precise interpolation factors and with the Greenberg-Benois arithmetic $\mathscr{L}$-invariant, for the Galois representation $V_\pi(1)$ attached to a  unitary self-dual cuspidal automorphic representation $\pi$ of $\GL_2$ over a totally real field $F$ having an arbitrary cohomological weight.

 The construction of a $p$-adic $L$-function for $\pi$ requires the choice of a regular $p$-refinement $\widetilde{\pi}$, {\it i.e.}, the choice for each $v$ dividing $p$ of a character $\nu_v$ of $F_v^\times$ which can be realized uniquely as a sub of the Weil-Deligne representation attached to $\pi_v$ via the local Langlands correspondence.
 Assuming that  $\widetilde{\pi}$ is {\it non-critical} (see Def. \ref{d:non-critical}), there exists a $p$-adic $L$-function $L_p(\widetilde{\pi},s)$.

 Let $S_p$ be the set of places of $F$ above $p$
and $\St_p$ the subset of places at which $\pi$ is a twist of the Steinberg representation.  
  The set $E$ of places for which the local interpolation factor of $L_p(\widetilde{\pi},s)$
vanishes at $s=1$ consists of $v\in \St_p$ such that $\pi_v$ is 
the Steinberg representation.

\begin{theoremletter}[Theorem \ref{t:ezc-hmf}] 
Let  $\widetilde{\pi}$ be a non-critically refined cohomological self-dual cuspidal automorphic
representations of $\GL_2$ over $F$, which is Iwahori spherical at  places above $p$.
Then  $L_p(\widetilde{\pi},s)$ has order of vanishing  at
least $e=|E|$ at $s=1$ and
\begin{align}\label{eq:ezc}
L_p^{(e)}(\widetilde{\pi},1) =e! \,\mathscr{L}(\widetilde{\pi}) \cdot \frac{L(\pi,\tfrac{1}{2})}{\Omega_{\widetilde{\pi}}} 
\cdot 2^{|\St_p\!\setminus E|} \!\! \prod_{v\in S_p\!\setminus \St_p} \left(1-\nu_v^{-1}(\varpi_v)\right)^2, 
 \end{align}
 where $\varpi_v$ is a uniformizer at $v$, 
 $\mathscr{L}(\widetilde{\pi})$ is the Fontaine-Mazur
 $\mathscr{L}$-invariant of Definition \ref{d:L-inv}, and 
$\Omega_{\widetilde{\pi}}$ is a  Betti-Whittaker period defined in \S\ref{ss:periods}. 
Moreover, if the Greenberg-Benois arithmetic $\mathscr{L}$-invariant is defined, then the Trivial Zero Conjecture holds for the Galois representation $V_\pi(1)$ with the choice of regular submodule as in \S\ref{ss:greenberg-benois}.
\end{theoremletter}

The conjectural non-vanishing of the $\mathscr{L}$-invariant is currently only known for elliptic curves over $\Q$ (see \cite{saint-etienne}). Previously, the Trivial Zero Conjecture at the central point was proved for modular forms over $\Q$ in \cite{greenberg-stevens,stevens:ezc}, and for parallel weight $2$ ordinary Hilbert cusp forms in \cite{mok} using a Rankin-Selberg construction for a single trivial zero and in general in \cite{spiess}, building on ideas of \cite{Dar01,orton:ezc}.
The non-criticality condition is implied by the assumption of $\widetilde{\pi}$ having non-critical slope (see Cor. \ref{p:non-criticality}) and is expected to be true for most regular $\widetilde\pi$ (see Bella{\"\i}che \cite{bellaiche:critical}, Pollack-Stevens \cite{pollack-stevens:critical}, and Bella{\"\i}che-Dimitrov \cite{bellaiche-dimitrov} when $F=\Q$).  After the completion of this article we were made aware of a preprint of Bergdall and Hansen \cite{bergdall-hansen} who construct $p$-adic $L$-functions under the weaker condition of $\widetilde{\pi}$ being 
{\it decent} (see Rem.\ref{rem:BH}).

In the case of a single trivial zero our approach to the Main Theorem is inspired by the work of Greenberg-Stevens \cite{greenberg-stevens}, and crucially uses the $p$-adic $L$-function of the unique $p$-adic family containing $\widetilde{\pi}$ constructed and studied in the first part of the paper. 

However, in the case of a multiple trivial zero, the computation of higher order derivatives has long been known to lie outside the reach of the Greenberg-Stevens method, thus  requiring some  genuinely new ideas.
 Indeed, the use of the functional equation for the $p$-adic $L$-function of the family of maximal dimension
containing $\widetilde{\pi}$, as suggested by Hida and Mazur in
the nearly ordinary case (see \cite[\S1]{mok}), does not suffice alone to compute higher order derivatives. Our innovation consists in making use of partial $p$-adic families to flip the sign of the root number and deduce the vanishing of many Taylor coefficients of a certain $p$-adic analytic function $\mathbb{L}_p(x_1,\ldots, x_e;u)$. We deduce the Main Theorem from the following properties:
\begin{enumerate}
\item (Specialization) $ \mathbb{L}_p(0,..,0;u)= \langle \gn\rangle_p^{u/4}L_p(\widetilde{\pi}, \tfrac{2-u}{2})$, 
 $\gn$ being the tame conductor of $\pi$. 
 \item (Functional equation) $\mathbb{L}_p(x_1,\dots, x_e;-u) = \widetilde{\varepsilon} \cdot \mathbb{L}_p(x_1,\dots, x_e; u)$, with $\widetilde{\varepsilon}\in\{\pm 1\}$.
\item (Retrieved $L$-value) $(-2)^{e} \frac{d^e}{du^e}\mathbb{L}_p(u,\dots,u;u)|_{u=0}= \text{R.H.S.  of  } \eqref{eq:ezc}$. 
\item (Taylor coefficients) $\mathbb{L}_p(x_1,\dots, x_e;u)$ contains only multinomials of total degrees $\geqslant e$, 
\[ \text{and }  \frac{d^e}{du^e}\mathbb{L}_p(0,\dots,0;u)|_{u=0}=  \frac{d^e}{du^e}\mathbb{L}_p(u,\dots,u;u)|_{u=0}. \]
\end{enumerate}

 Our construction of $p$-adic $L$-functions is geometric, based on the theory of automorphic symbols 
 introduced in \cite{dimitrov} and on the construction of eigenvarieties using overconvergent cohomology
 as in \cite{urban} and \cite{hansen}. For a Hilbert modular variety $Y_K$ this was 
 initiated in \cite{barrera} and is fully developed in the present work. 
 Given an admissible affinoid neigbourhood $\cU$ of the weight of $\pi$ in the weight space,  
and given a non-zero $U_p$-eigenclass $\Phi$ in the  compactly supported overconvergent cohomology $\rH^d_c(Y_K, \cD_\cU)$, 
 we attach in \S\ref{automorphic-symbols} a canonical $\cO(\cU)$-valued distribution $\ev(\Phi)$ having controlled growth on the Galois group of the 
maximal abelian extension of $F$ which is unramified outside $p\infty$. 
   Specializing to the case where 
$\Phi$ corresponds to the $p$-adic family passing through $\widetilde\pi$ (see Theorem \ref{free-etale}) we define $L_p(\lambda,s)$ as the $p$-adic Mellin transform of $\ev(\Phi)$ (see \S\ref{ss:2var} and \eqref{eq:padicL}).
A specific feature of our treatment is that, thanks to a precise choice of a $p$-refined automorphic newform in $\pi$, 
the interpolation formula for $L_p(\lambda,s)$ has no superfluous factors, allowing us to establish the concise functional equation (ii) (see Theorem \ref{t:functional-equation-2}). While the proof of (iii) uses the improved $p$-adic $L$-functions constructed in \S\ref{ss:improved}, reinterpretting (iii) in terms of arithmetic $\mathscr{L}$-invariants requires an extra input from $p$-adic Hodge theory, namely the existence of rigid analytic triangulations in the category of $(\varphi,\Gamma)$-modules.

 In the case of a simple zero ($e=1$), property (iv) is an immediate consequence of the functional equation (i) as 
in the Greenberg-Stevens method. Establishing (iv) in the case of a zero of higher order ($e>1$), which is the keystone in our approach,  demands to go beyond the Greenberg-Stevens method and use partially improved $p$-adic $L$-functions as well as study the behavior of $\pi_{\lambda,v}$ in certain `partial' $p$-adic families defined in \S\ref{subsection-controls}. This allows us to establish a number of relations between the Taylor coefficients of $\mathbb{L}_p(x_1,\dots, x_e;u)$ which are not all predicted by the Trivial Zero Conjecture and which we believe are of independent interest.  
Our results do not rely on the Leopoldt nor the Bloch-Kato Conjectures. 
The formula that we show is true even when the archimedean $L$-function vanishes at the central point, implying then that 
the order of vanishing of the $p$-adic $L$-function is at least $e+1$. 

\bigskip 
{\small
\paragraph{{\bf Acknowledgements}} We are indebted to the referees for their numerous comments and suggestions 
on how to improve the exposition. We would like to also  thank D.~Benois, A.~Betina, J.~Bergdall, D.~Hansen, H.~Hida, A.~Iovita, A.~Raghuram, D.~Ramakrishnan, G.~Rosso, J.~Tilouine, E.~Urban, C.~Williams for helpful discussions. The second author is thankful to the University of Notre Dame for its hospitality during his visit in 2017, while the third author is grateful to the CEMPI Lab at the University of Lille where he was a Visiting Associate in June 2016 and June 2019, and to the Institute of Mathematics at Sorbonne University, where part of this work was completed, for their excellent research environments. 

The first author has received funding from the European Research Council grant $n^\circ$682152 and FONDECYT PAI 77180007, while the
 second author has benefited from the support of the Agence Nationale de la Recherche grants ANR-11-LABX-0007-01 
and ANR-18-CE40-0029, and the 
 third author was partially supported by the NSA Young Investigator grant H98230-16-1-0302.
}

\newpage

\setcounter{section}{-1}

\section{Notations and conventions}\label{s:notations}

Throughout this paper $F$ will be a totally real number field of degree $d$ and ring of integers $\cO_F$. Let $\A_F=\A\otimes_\Q F=\A_{F,f}\times F_\infty$ be the ring of adeles of $F$ and denote $\widehat\cO_F= \cO_{F}\otimes_{\Z} \widehat{\Z}$.

We choose a generator $\varpi_{\gf} \in \A_{F,f}^{\times}$ of each fractional ideal $\gf$ of $F$ 
such that for any finite place $v$ of $F$ one has $\varpi_{v\gf}=\varpi_v\cdot\varpi_{\gf}$, 
where $\varpi_v$ is a uniformizer of the ring of integers $\cO_v$ of $F_v$. 

Moreover, we define the adele 
$1_{\gf}\in \A_F$ by $ (1_{\gf})_{v}= 
\begin{cases} 
1 & \text{, if } n_v \neq 0, \\ 
0 & \text{, if } n_v = 0.
\end{cases}$

When $\gf\subset \cO_F$ we consider the strict idele class group
\[\Cl(\gf)= F^{\times}\!\setminus\A_F^{\times}/U(\gf) F_\infty^{\times+},\]
where $(\cdot)^+$ denotes the connected component of  identity in a real Lie group and
$U(\gf)$ denotes the principal congruence subgroup of level $\gf$ of $\widehat\cO_F^{\times}$. 
Moreover we denote $E(\gf) \subseteq \cO_F^{\times}$ the group of totally positive units which are congruent to $1$ modulo $\gf$.

Let $\Sigma$ be the set of infinite places of $F$ which are all real and can also be seen as 
embeddings of $F$ in the algebraic closure $\overline{\Q}$ of $\Q$ inside $\C$. The choice, for a given prime number $p$, of an  embedding of $\iota_p:\overline{\Q}\hookrightarrow \overline{\Q}_p$ induces a partition 
$\Sigma= \bigsqcup_{v\in S_p} \Sigma_v$, where $\sigma\in \Sigma_v$ 
if and only if $v$ is the kernel of the composed map $ \cO_F\xrightarrow{\iota_p\circ \sigma} 
\overline{\Z}_p\twoheadrightarrow\overline{\F}_p$. 

We let   $(\cdot)^\epsilon$ denote the eigenspace corresponding to a character 
  $\epsilon$ of  $F_\infty^\times/F_\infty^{\times+}=\{\pm 1\}^{\Sigma}$. 

We denote by $\mathrm{G}_E$ the absolute Galois group of a perfect field $E$. 
For $S$ a finite set of places of $F$ we let $\Gal_{S}$ denote the Galois group of the 
maximal abelian extension of $F$ which is unramified outside $S$. We let $\Gal_{S\infty}=\Gal_{S\cup\Sigma}$ and 
 $\Gal_{p\infty}=\Gal_{S_p\cup\Sigma}$, where $S_p$ denotes the set of places above $p$.  
For $S\subset S_p$, we let $\Sigma_S= \bigsqcup_{v\in S} \Sigma_v$ and 
 $\cO_{F,S}=\prod_{v\in S}\cO_v$.

 The cyclotomic character $\chi_{\cyc}:\Gal_{p\infty}
\twoheadrightarrow \Gal(F(\mu_{p^{\infty}})/F) \to \Z_p^\times$ corresponds, via global class field theory, to the idele class character 
$\chi_{\cyc}:F^{\times}_+ \!\setminus \A_{F,f}^{\times} \to \Z_p^\times$ sending $y$ to $ \prod\limits_{v\in S_p}\mathrm{N}_{F_v/\Q_p}(y_v)|y_f|_F$. 
One has $\chi_{\cyc}(\varpi_v)=|\varpi_v|_v=q_v^{-1}$ if $v\notin S_p\cup \Sigma$ and 
$\chi_{\cyc}(\varpi_v)=\mathrm{N}_{F_v/\Q_p}(\varpi_v) q_v^{-1}$ if $v\in S_p$. 
Define 
\[\langle \cdot \rangle_p=\chi_{\cyc}\omega_p^{-1}:\Gal_{p\infty}\to 1+2p\Z_p,\] where 
$\omega_p$ is the Teichm\"uller lift of $\chi_{\cyc}$ mod $p$, if $p$ is odd (resp. of 
$\chi_{\cyc}$ mod $4$, if $p=2$). Note that the character $\langle \cdot \rangle_p$ factors through the Galois group $\Gal_{\cyc}= \Gal(F_{\cyc}/F)$ of the cyclotomic $\Z_p$-extension $F_{\cyc}\subset F(\mu_{p^{\infty}})$ of $F$, hence can be raised to power any $s\in \cO_{\C_p}$. 

We normalize the Artin reciprocity map so that a uniformizer $\varpi_v$ is sent to a geometric Frobenius $\Frob_v$, and $p$-adic Hodge theory so that the 
cyclotomic character $\chi_{\cyc}$ has Hodge-Tate weight $-1$. 
We consider  the following non-trivial additive unitary character of $\A_{F}/F$:
\begin{align*}
\psi: \A_{F}/F\longrightarrow \A/\Q \longrightarrow \C^\times,
\end{align*} 
where the first map is the trace, and the second is the
usual additive character $\psi_0$ on $\A/\Q$ 
characterized by $\ker(\psi_{0|\Q_\ell})=\Z_\ell$ for every prime
number $\ell$ and $\psi_{0|\R}=\exp(2i\pi \cdot)$. 
We remark that the largest fractional ideal contained in $\ker(\psi_v)$ equals 
$(\varpi_{v}^{-\delta_v})$, where $\delta_v $ is the valuation at $v$ of the different $\gd $ of $F$. With this notation the discriminant of $F$ is $\mathrm{N}_{F/\Q}(\gd)$.

Let $dx=\otimes_v dx_v$ be the (self-dual) Haar measure on
$\A_F$ which induces the discrete measure on $F\subset \A_F$ and the Haar measure with volume $1$ on $\A_F/ F$. It has the property that $dx_\sigma$ is the usual Lebesgue measure for $\sigma\in \Sigma$ and 
 when $v$ is a finite place $\displaystyle \int_{\cO_v}dx_v= q_v^{-\delta_v/2}$. 
 We also let $d^\times x=\otimes d^\times x_v$ be the Haar measure on $\A_F^\times/F^\times$ such that $d^\times x_\sigma=|x_\sigma|_\sigma^{-1}dx_\sigma$ for $\sigma\in \Sigma$ and $\displaystyle \int_{\cO_v^\times}d^\times x_v=1$, for $v$ finite.

Given a finite place $v$ and  a character $\chi_v$ of $F_v^\times$ of conductor $c_v$,  one can define the local Gauss sum, which is independent of the choice of uniformizers, by: 
\begin{align}\label{eq:gauss}
\tau(\chi_v,\psi_v,d^\times)=\int_{F_v^\times}\chi_v(y)\psi_v(y)d^\times y=
\int_{\cO_v^\times}\chi_v(u\varpi_v^{-c_v-\delta_v})\psi_v(u\varpi_v^{-c_v-\delta_v})du.
\end{align}
For $\chi:F^{\times} \!\setminus \A_F^{\times} \to \C^\times$ an idele class character of conductor $\gc_\chi$ we define the global Gauss sum
\begin{align}\label{eq:global-gauss-sum}
\tau(\chi)=\prod_{v\nmid \infty}\tau(\chi_v,\psi_v,d_{\chi_v})
=\prod_{v\mid \gc_\chi}\tau(\chi_v,\psi_v,d^\times)
 \prod_{v\nmid \gc_\chi\infty }\chi_v(\varpi_v^{-\delta_v}),
 \end{align}
 where the Haar measure $d_{\chi_v}$ on $F_v^\times$ gives $\cO_v^\times$ volume $1$ (resp. $1-q_v^{-1}$)  when $\chi_v$ is unramified (resp. ramified).

\bigskip

\tableofcontents

\addtocontents{toc}{\setcounter{tocdepth}{1}}

\newpage

\part{$p$-adic $L$-functions for families of nearly finite slope Hilbert cuspforms}
We develop a natural framework yielding simultaneous 
 constructions of $p$-adic $L$-functions and their improved counterparts 
 for nearly finite slope families of Hilbert cusp forms.

\section{Automorphic theory of Hilbert cusp forms}
We recall the representation theory of Hilbert automorphic cusp forms and construct normalized 
$p$-refined nearly finite slope newforms allowing us to define canonical 
periods.

\subsection{Hilbert modular varieties}\label{ss:hilbert modular varieties}
We consider the reductive group scheme $G=\Res^{\cO_F}_\Z\GL_{2}$ over $\Z$. We  let  $C_{\infty}$ be the standard maximal compact subgroup of $G_\infty$ and 
 $K_{\infty}= C_\infty F_\infty^\times$.

The Hilbert modular variety of level $K$, an open compact subgroup of $G(\A_f)$, is defined as the locally symmetric space
\[ Y_K= G(\Q)\!\setminus G(\A)/ K K_{\infty}^+ . \]

By the Strong Approximation Theorem for $\SL_{2}(\A_F)$, the fibers of the map: 
\[{\det}_K:Y_K\to F^{\times}\!\setminus\A_F^{\times}/\det(K) F_\infty^{\times+},\]
are connected.
For each $[\eta]\in \A_F^{\times}/(F^{\times}\det(K) F_\infty^{\times+}) = \pi_0(Y_K)$ the connected component
 $Y_K[\eta]= {\det}_K^{-1}([\eta])$ can be described as a quotient of the unbounded hermitian symmetric domain $G_{\infty}^+/K_{\infty}^+ $ by a congruence subgroup as follows. 
Choosing a representative $\eta\in\A_{F,f}^{\times}$ of $[\eta]$ (one can take it to be a uniformizer at some finite place), there is an isomorphism
\[ \Gamma_\eta\!\setminus G_\infty^+/K_{\infty}^+ \simeq Y_K[\eta]\text{, } 
 g_\infty\mapsto g_\infty \begin{psmallmatrix}\eta & 0 \\ 0 & 1\end{psmallmatrix},
 \text{ where } \Gamma_\eta=G(\Q)\cap \begin{psmallmatrix}\eta & 0 \\ 0 & 1\end{psmallmatrix} K\begin{psmallmatrix}\eta & 0 \\ 0 & 1\end{psmallmatrix}^{-1}G_{\infty}^+. 
 \]

In the sequel we assume that $K$ is sufficiently small in the sense that for all $g\in G(\A)$: 
\begin{align}\label{neat}
G(\Q)\cap gKK_{\infty}^+  g^{-1}= F^{\times} \cap K F_\infty^\times.
\end{align}
It is equivalent to ask that $\Gamma_\eta$ modulo its center $\Gamma_\eta\cap F^{\times}$ is torsion free, this property being independent of the choice of the representative $\eta$. Then $Y_K[\eta]$ is a complex manifold admitting $G_\infty^+/K_{\infty}^+$ as a universal covering space with group $ \Gamma_\eta/(\Gamma_\eta\cap F^{\times})$.

\subsection{Local systems and cohomology} \label{local-sys}
We will now describe two natural constructions of local systems on $Y_K$. In \S \ref{alg-weights} we will apply these constructions to attach local systems to algebraic representations of $G$.
Consider first a left $G(\Q)$-module $V$ such that 
\begin{align}\label{center} 
 F^{\times} \cap K F_\infty^\times \text{ acts trivially on } V.
 \end{align}
 
 By \eqref{neat} and \eqref{center} the group $G(\Q)\cap gKK_{\infty}^+  g^{-1}= 
F^{\times} \cap K F_\infty^\times$ acts trivially on $V$. Therefore 
\[ G(\Q)\!\setminus (G(\A)\times V) /K K_{\infty}^+  \to Y_K,\] 
is a local system with left $G(\Q)$-action and right $K K_{\infty}^+ $-action given by:
$\gamma(g,v)k=(\gamma g k ,\gamma\cdot v)$.

Alternatively, given a left $K$-module $V$ satisfying \eqref{center}, one can consider the local system $\cV_K$:
\[G(\Q)\!\setminus (G(\A)\times V) /K K_{\infty}^+ \to Y_K,\] 
with left $G(\Q)$-action and right $K K_{\infty}^+  $-action given by: 
$\gamma(g,v)k=(\gamma g k , k^{-1}\cdot v)$.

 We will denote by $\cV_K$ (or $\cV$ if $K$ is clear from the context) the corresponding sheaf of locally constant sections on $Y_K$ and will consider the usual (resp. compactly supported) cohomology groups 
$\rH^{i}(Y_K,\cV)$ (resp. $\rH_c^{i}(Y_K,\cV) $). 
 Although we use the same notation for both constructions, it will be generally clear from the context which one applies and otherwise we will name it explicitly. 
When the actions of $G(\Q)$ and $K$ on $V$ extend compatibly into a left action of $G(\A)$, the two resulting local systems are isomorphic by $(g,v)\mapsto (g,g^{-1}\cdot v)$, yielding an isomorphism of sheaves and their cohomology groups, 	thus justifying the abuse of notation.

\subsection{Cohomological weights} \label{alg-weights}

Let $B$ be the standard Borel subgroup of $G$ of upper triangular matrices, whose Levi subgroup $T$ consists of the diagonal matrices. 

The characters of the torus $\Res^F_{\Q}\G_m$ can be identified with $\Z[\Sigma]$ as follows: 
for any $k=\sum\limits_{\sigma\in \Sigma} k_\sigma \sigma\in\Z[\Sigma]$ and any $\Q$-algebra $A$ splitting $F$,
we consider the character
\begin{align}\label{eq:integer-char}
x\in (F\otimes_{\Q}A)^\times\mapsto x^k=\prod_{\sigma\in
 \Sigma} \sigma(x)^{k_\sigma}\in A^\times.
 \end{align}
The norm character $\mathrm{N}_{F/\Q}:\Res^F_{\Q}\G_m \to \G_m$  corresponds then to the element $t=\sum\limits_{\sigma\in \Sigma} \sigma$.

Integral weights of $G$ are given by characters of the form $(a,d)\mapsto 
a^k d^{k'}$ for some $(k,k')\in \Z[\Sigma]^2$. Characters such that $k_\sigma\geqslant k'_\sigma$ for all $\sigma\in \Sigma$ are called 
dominant with respect to $B$ and parametrize the irreducible algebraic representations of $G$, explicitly given by 
\[\bigotimes_{\sigma\in \Sigma} \left(\Sym_\sigma^{k_\sigma-k'_\sigma}\otimes \Det_\sigma^{k'_\sigma}\right).\]

\begin{defn} \label{d:coh-weight}
We say that a dominant weight of $G$ is {\it cohomological} if it is of the form 
\[\left(\frac{(\sw-2) t+k}{2}, \frac{(\sw+2)t-k}{2}\right), \mathrm{ where} \]
 $(k,\sw)\in \Z[\Sigma]\times\Z$ is such that for all 
$\sigma\in \Sigma$ we have $k_\sigma\geqslant 2$ and $k_\sigma\equiv \sw \pmod{2}$. We will identify the cohomological weight with the tuple $(k,\sw)$ defining it, and call $\sw$ the purity weight. 
\end{defn}
A dominant integral weight is cohomological exactly when the central character of the corresponding $G$-representation factors through the norm $\mathrm{N}_{F/\Q}$. 

Given a cohomological weight $(k,\sw)$ and a $\Q$-algebra $A$ splitting $F$, we
consider the $A$-module $L_{k,\sw}(A)$ of polynomials $f$ of degree at most $k-2t= (k_\sigma-2)_{\sigma\in \Sigma}$ in the variables $z=(z_\sigma)_{\sigma\in \Sigma}$ with coefficients in $A$, endowed with the following right action of $G(A)\simeq \GL_2(A)^{\Sigma}$:
\begin{align}\label{right-action}
f_{|\gamma}(z)=\det(\gamma)^{((\sw+2)t-k)/2}(cz+d)^{k-2t} f\left(\frac{az+b}{cz+d}\right), \text{ where } 
 \gamma=\begin{psmallmatrix}a & b \\ c & d\end{psmallmatrix} \in G(A).
\end{align}
Then its dual $L_{k,\sw}^\vee(A)=\Hom_A(L_{k,\sw}(A), A)$ is endowed with a left action of $G(A)$ given by 
\begin{align}\label{left-action}
(\gamma\cdot \mu)(f)= \mu(f_{|\det(\gamma)^{-1}\cdot \gamma}), \text{ where }  \gamma\in G(A), \,\mu\in L_{k,\sw}^\vee(A),\, f\in L_{k,\sw}(A). 
\end{align}
and there is an isomorphism of left $G(A)$-modules 
\begin{align}
L_{k,\sw}^\vee(A)\simeq\bigotimes_{\sigma\in \Sigma} (\Sym_\sigma^{k_\sigma-2}\otimes 
\Det_\sigma^{(2-k_\sigma-\sw)/2})(A^2)
\end{align}

For  $(k,\sw)$ and $A$ as above, the assumption \eqref{center} for the left $G(A)$-module $L_{k,\sw}^\vee(A)$ reads:
\begin{align}\label{center-bis} 
 \mathrm{N}^{\sw}_{F/\Q}(\varepsilon)=1
 \text{ for all }\varepsilon \in F^{\times} \cap K F_\infty^\times.
 \end{align}
Under this condition applying the construction of \S \ref{local-sys} yields a sheaf 
$\cL_{k,\sw}^\vee(A)$ whose cohomology groups $\rH^{i}(Y_K,\cL_{k,\sw}^\vee(A))$ and $\rH_c^{i}(Y_K,\cL_{k,\sw}^\vee(A))$ will play a prominent role in this paper.

\subsection{Cohomological cuspidal automorphic representations}\label{ss:cohomological cuspidal automorphic representations}

The aim of this section is to describe the cuspidal automorphic representations contributing to $\rH^d(Y_K,\cL_{k,\sw}^\vee( \C))$ for $(k,\sw)$ a cohomological weight as in Definition \ref{d:coh-weight} and $K$ satisfying \eqref{center-bis}, and to perform some archimedean computations which will be used to interpret cohomologically the special values of automorphic $L$-functions. While the general theory is well known, the applications we have in mind require an explicit version as in \cite[\S4.4]{raghuram-tanabe}.

Let $\gg_\infty$ (resp. $\gk_\infty$) be the complexified Lie algebra of $G_\infty$ (resp. $K_\infty$).
Using the comparison between Betti cohomology over $\C$ and de Rham cohomology, and further
reinterpreting the de Rham complex in terms of the complex computing relative Lie algebra cohomology, we obtain:
\begin{align*}
\rH^d_{\cusp}(Y_K, \cL_{k,\sw}^\vee(\C))=&
\rH^d(\gg_\infty, K_\infty^+,L_{k,\sw}^\vee(\C)\otimes C^\infty_{\cusp}(G(\Q)\!\setminus G(\A)/K))\\
=&
\bigoplus_{\pi}
\rH^d(\gg_\infty, K_\infty^+,  L_{k,\sw}^\vee(\C)\otimes \pi_\infty)\otimes \pi_f^K,
\end{align*}
where $\pi$  runs over the cuspidal automorphic representation of $G(\A)$. 
By K\"unneth's formula 
\[\rH^d(\gg_\infty, K_\infty^+, L_{k,\sw}^\vee(\C)\otimes\pi_\infty )=
\bigotimes\limits_{\sigma \in \Sigma}\rH^1(\gg_\sigma,
K_\sigma^+,  L_{k_\sigma,\sw}^\vee(\C) \otimes \pi_\sigma),\]
where $\rH^1(\gg_\sigma, K_\sigma^+,  L_{k_\sigma,\sw}^\vee(\C)\otimes\pi_\sigma)\neq 0$ if and only if 
$\pi_\sigma$ is the irreducible infinite dimensional representation $\pi_{k_\sigma,\sw}$ of $\GL_2(\R)$ whose Langlands parameter $\C^\times\rtimes \{1,j\}\to \GL_2(\C)$ is given by 
\[z\in \C^\times \mapsto |z|^{\sw /2}
 \begin{pmatrix} (\overline{z}/z)^{(k_\sigma-1)/2}&\\ &(z/\overline{z})^{(k_\sigma-1)/2}\end{pmatrix}
\text{ and } j\mapsto \begin{pmatrix} &1\\(-1)^{k_\sigma-1}&\end{pmatrix}.\]

One can also describe $\pi_{k_\sigma,\sw}$ as follows. Consider the unitary induction from $B_\sigma$ to $G_\sigma$ of the character which is trivial on the unipotent radical and given on $T_\sigma$ by
\begin{align}\label{induced}
(a,d)\mapsto a^{(k_\sigma+\sw -2)/2} d^{(\sw-k_\sigma+2)/2}\left|\tfrac{a}{d}\right|^{1/2}.
\end{align}
By Frobenius reciprocity it has a unique non-trivial finite dimensional quotient given by $L_{k_\sigma,\sw}(\C)$ and the kernel turns out to be isomorphic to $\pi_{k_\sigma,\sw}$. The fact that the extension is non-split
 implies the non-vanishing of $\rH^1(\gg_\sigma, K_\sigma^+,  L_{k_\sigma,\sw}^\vee(\C)\otimes\pi_\sigma)$.

\begin{definition}\label{aut-def} 
We say that an automorphic representation $\pi$ of $G(\A)$ has weight $(k,\sw)$ if 
$\pi_\infty=\mathop{\otimes}\limits_{\sigma\in \Sigma} \pi_{k_\sigma,\sw}$.
The integer $\sw$ is the purity weight of $\pi$, {\it i.e.}, the weight of its central character.
\end{definition}

In summary, a cuspidal automorphic representation $\pi$ of $G(\A)$ contributes to 
$\rH^d(Y_K,\cL_{k,\sw}^\vee(\C))$ if and only if it has weight
 $(k,\sw)$ and $\pi_f^K\neq 0$. If such
$\pi$ exist then condition \eqref{center-bis} is always satisfied. Indeed, if $\phi\in \pi^K$, then for each $\varepsilon\in F^\times\cap K F_\infty^\times$ and $g\in G(\A)$ we have
\[\phi(g)=\phi(\begin{psmallmatrix}\varepsilon&\\ &\varepsilon
\end{psmallmatrix}g) = \phi(g\begin{psmallmatrix}\varepsilon_\infty&\\ &\varepsilon_\infty
\end{psmallmatrix})=\mathrm{N}_{F/\Q}^\sw(\varepsilon) \phi(g).\]

We will now specify a basis of 
\begin{align}\label{eq:gk-coh}
\rH^1(\gg_\sigma, K_\sigma^+, L_{k_\sigma,\sw}^\vee(\C)\otimes \pi_{k_\sigma,\sw})=
\Hom_{C_\sigma^+}(\gg_\sigma/\gk_\sigma, L_{k_\sigma,\sw}^\vee(\C)\otimes \pi_{k_\sigma,\sw}),
\end{align}
where one considers the adjoint action of $C_\sigma^+=\left\{r(\theta)=\begin{psmallmatrix}\cos(\theta) & -\sin(\theta)\\ \sin(\theta)&
\cos(\theta) \end{psmallmatrix} \Big{|} \theta\in \R\right\}$ on $\gg_\sigma/\gk_\sigma$.

Let $(w_\sigma^*, \bar w_\sigma^*)$ be the dual basis of the basis $w_\sigma=\tfrac{1}{4}\begin{psmallmatrix}1 & i \\ i &-1 \end{psmallmatrix}, \bar w_\sigma=\tfrac{1}{4}\begin{psmallmatrix}1 & -i \\ -i &-1 \end{psmallmatrix}$
of $\gg_\sigma/\gk_\sigma$. 

Consider $\eval_{\pm i}\in L_{k_\sigma,\sw}^\vee(\C)$ defined as $\eval_{\pm i}(P)=P(\pm i)$. 
From \eqref{left-action} one finds that 
\[ \Ad(r(\theta))(w_\sigma^*)= e^{2i\theta} w_\sigma^*, \, \Ad(r(\theta))(\bar w_\sigma^*)= e^{-2i\theta}\bar w_\sigma^*, 
\text{ and } r(\theta)\cdot \eval_{\pm i}=e^{\pm i\theta(k_\sigma-2)}\eval_{\pm i}. 
\] 

Since the $C_\sigma^+$-types $r(\theta)\mapsto e^{\pm i\theta k_\sigma}$ do not occur in 
$L_{k_\sigma,\sw}^\vee(\C)$ but do occur in the induced representation from \eqref{induced}, one deduces that there is a unique function $\phi_\sigma\in \pi_{k_\sigma,\sw}$ such that $\phi_\sigma(\cdot r(\theta))= e^{-i\theta k_\sigma}\phi_\sigma$ for all $\theta\in\R$ and normalized such that $\phi_\sigma(1)=1$ (here we use that the characters $r(\theta)\mapsto e^{\pm i\theta k_\sigma}$ and \eqref{induced} agree on $\begin{psmallmatrix}-1 & \\ &
-1 \end{psmallmatrix}$ since $k_\sigma\equiv \sw\pmod{2}$). Hence: 
\[\Hom_{C_\sigma^+}(\gg_\sigma/\gk_\sigma, L_{k_\sigma,\sw}^\vee(\C)\otimes \pi_{k_\sigma,\sw})= \C (w_\sigma^*\otimes  \eval_{i}\otimes\phi_\sigma ) \oplus \C (\bar w_\sigma^*\otimes  \eval_{-i}\otimes\overline{\phi_\sigma}).\]
Since $\begin{psmallmatrix}-1 & \\ &
1 \end{psmallmatrix}\cdot w_\sigma^* = \bar{w}_\sigma^*$, $\begin{psmallmatrix}-1 & \\ &
1 \end{psmallmatrix}\cdot \eval_i = (-1)^{(\sw+k_\sigma-2)/2} \eval_{-i}$ and 
$\begin{psmallmatrix}-1 & \\ &
1 \end{psmallmatrix}\cdot \phi_\sigma=(-1)^{(\sw+k_\sigma-2)/2}\overline{\phi_\sigma}$, 
\begin{equation*}
\Xi_{\pi_\sigma}^{\epsilon_\sigma}= i^{(2-\sw-k_\sigma)/2} \left(w_\sigma^*\otimes  \eval_{i}\otimes\phi_\sigma + \epsilon_\sigma(-1) \bar w_\sigma^*\otimes  \eval_{-i}\otimes\overline{\phi_\sigma}\right)
\end{equation*}
is an eigenbasis of \eqref{eq:gk-coh} for the action of 
$C_\sigma/C_\sigma^+=K_\sigma/K_\sigma^+\xrightarrow[\sim]{\det}
F_\sigma^\times/F_\sigma^{\times+}=\{\pm 1\}$.

Letting $w_\infty^*=\mathop{\otimes}\limits_{\sigma\in \Sigma} w_\sigma^* $ and $\phi_\infty=\mathop{\otimes}\limits_{\sigma\in \Sigma} \phi_\sigma $, 
the space $\rH^d(\gg_\infty, K_\infty^+, \pi_\infty\otimes L_{k,\sw}^\vee(\C))$ has the following basis indexed by the characters $\epsilon: K_\infty/K_\infty^+= \{\pm 1\}^\Sigma\to \{\pm 1\}$:
\begin{align}\label{eq:XI}
\Xi_{\pi_\infty}^\epsilon=\bigotimes_{\sigma\in \Sigma}
\Xi_{\pi_\sigma}^{\epsilon_\sigma}=i^{((2-\sw)t-k)/2} \sum_{s_\infty\in \{\pm 1\}^\Sigma} \epsilon(s_\infty) 
(s_\infty\cdot (w_\infty^*\otimes \eval_{i} \otimes \phi_\infty )).
\end{align}

\subsection{Automorphic representations of nearly finite slope} \label{ss:automorphic representation of nearly}
In this section, we introduce the notion of {\it a nearly finite slope} cuspidal automorphic representation $\pi$. 
These are stable under twists and encompass both Coleman's finite slope and Hida's nearly ordinary cases.
Furthermore, we define a $p$-refined newline line in such a $\pi$.
\begin{definition}\label{d:hilbert-nearly}
Let $\pi$ be a cohomological cuspidal automorphic representation $\pi$ of $G(\A)$ and let $v\in S_p$ be a place of $F$. We say that
\begin{enumerate}
\item $\pi_v$ has {\it finite slope} if either $\pi_v$ is a principal
series with at least one unramified character or $\pi_v$ is an
unramified twist of the Steinberg representation.
\item $\pi_v$ has {\it nearly finite slope} if $\pi_v$ is not supercuspidal, or equivalently if
it is a twist of a finite slope representation by a finite order character.
\item A representation $\pi_v$ is {\it regular} if either it is a twist of the Steinberg representation or it is a principal series representation with distinct characters $\chi_{1,v}\neq\chi_{2,v}$. 
\item A {\it refinement} of a nearly finite slope representation $\pi_v$ of $\GL_2(F_v)$ is a one-dimensional sub $\nu_v$ of the Weil-Deligne representation attached to $\pi_v$ via the local Langlands correspondence for $\GL_2(F_v)$. 
 \item For $S\subset S_p$, a (regular) $S$-refinement of an automorphic representation $\pi$ is a pair $\widetilde{\pi}_S=(\pi,\{\nu_v\}_{v\in S})$ such that $\nu_v$ is a (regular) refinement of $\pi_v$ for all $v\in S$. When $S=S_p$
 we call $\widetilde{\pi}=\widetilde{\pi}_{S_p}$ a $p$-refinement. 
 \end{enumerate} 
\end{definition}

Suppose $(\pi_v,\nu_v)$ is a refined regular nearly finite slope representation. If $\pi_v$ is twist of the Steinberg representation let $K_v$ be the
Iwahori group
\begin{align}\label{eq:iwahori}
I_v=\{\begin{psmallmatrix}
a&b\\c&d\end{psmallmatrix}\in \GL_2(\cO_v )\mid
c\in \varpi_v\cO_v,\}  \text{ and }  
\end{align}
 if $\pi_v$ is a principal series let $K_v=I_v\cap K_1(v^{m_v})$ where $m_v$ is the conductor of $\chi_{1,v}/\chi_{2,v}$. Let
\begin{align}\label{eq:nu-v}
K'_v=K(\pi_v,\nu_v) = \ker \left(K_v\xrightarrow{\det}
\cO_v ^\times\xrightarrow{\nu_v}
\C^\times\right)
\end{align}

For a uniformizer $\varpi_v\in \cO_v $ and $\delta\in \cO_v ^\times$ we define the
Hecke operators
\begin{align}\label{eq:hecke-delta}
U_{\varpi_v}= \left[ K_v' \begin{psmallmatrix}\varpi_v & \\ &
1 \end{psmallmatrix}K_v'\right]\text{ and } U_{\delta}= \left[ K_v' \begin{psmallmatrix}\delta & \\ &
1 \end{psmallmatrix}K_v'\right]. 
\end{align}

\begin{lemma}\label{l:multiplicity one}
For any refined regular nearly finite slope representation $(\pi_v,\nu_v)$ one has 
\[\dim(\pi_v^{K_v'})_{(U_{\varpi_v}-\nu_v(\varpi_v), U_{\delta}-\nu_v(\delta)\mid
\delta\in \cO_v ^\times)}=1.\]
\end{lemma}
\begin{proof}
Since the $\cO_v^\times$-action is semisimple, it follows from \eqref{eq:nu-v} that there is an isomorphism: 
\begin{align}\label{eq:KK}
(\pi_v^{K_v'})_{(U_{\delta}-\nu_v(\delta)\mid
\delta\in \cO_v ^\times)}= (\pi_v\otimes \nu_v^{-1})^{K_v}.
\end{align}
If $\pi_v$ is either a twist of the Steinberg representation or a principal series with $\chi_{1,v}/\chi_{2,v}$ ramified, then 
\eqref{eq:KK} is one-dimensional on which $U_{\varpi_v}$ acts by $\nu_v(\varpi_v)$.
If $\pi_v$ is a principal series with $\chi_{1,v}/\chi_{2,v}$ unramified
then \eqref{eq:KK} is two-dimensional and we conclude that the
$U_{\varpi_v}$-eigenspace for $\nu_v(\varpi_v)$ is a line by regularity.
\end{proof}

Let $\widetilde{\pi}_S=(\pi, \{\nu_v\}_{v\in S})$ be a regular $S$-refinement of a nearly finite slope cuspidal automorphic representation of
$G(\A)$ of cohomological weight $(k, \sw)$.
\begin{definition}\label{d:u}
Let $\gu$ be a prime of $F$ such that:
 \begin{enumerate}
\item any open compact subgroup $K$ of $\GL_2(\A_{F,f})$ such that  $K_\gu=K_0(\gu)$ satisfies \eqref{neat}, 
\item $\gu$ is unramified and $\pi_{\gu}$ is an unramified principal series with Hecke parameters $\alpha_{\gu}\neq\beta_{\gu}$.
\end{enumerate}
\end{definition}
The existence of $\gu$ satisfying (i) follows from \cite[Lem.2.1]{dimitrov:ihara}, while the fact that $\gu$ can be chosen to satisfy (ii) as well can be shown using the irreducibility of the Galois representation 
$V_{\pi}$. It is also a consequence of the Sato-Tate conjecture which is known for Hilbert cusp forms. In this case $\pi_{\gu}^{K_0(\gu)}$ is
$2$-dimensional on which $U_{\gu}$ acts with eigenvalues
$\alpha_{\gu}$ and $\beta_{\gu}$.

\begin{definition}\label{d:mpi-tilde}
Let $E$ be a number field containing the Galois closure of $F$, the field rationality of $\pi_f$, the Hecke parameters of $\pi_{\gu}$, and the values of the characters $(\nu_v)_{v\in S_p}$. 

Let $\gm_\pi$ be the maximal ideal corresponding to $\pi_f$ of 
the Hecke algebra $\T=E[T_v, S_v\mid v\nmid \gn\gu p]$. 

 For $S\subset S_p$, we consider the maximal ideal
\[\gm_{\widetilde{\pi}_S} = (\gm_\pi, U_{\gu}-\alpha_\gu,U_{\varpi_v}-\nu_v(\varpi_v),
U_{\delta}-\nu_v(\delta)\mid \delta\in \cO_v ^\times, v\in S)\]
of the Hecke algebra $\widetilde\T_S=\T[U_{\gu},U_{\varpi_v},U_{\delta}\mid \delta\in
\cO_v ^\times, v\in S]$, and we let $\gm_{\widetilde{\pi}}=\gm_{\widetilde{\pi}_{S_p}}$. 
\end{definition}

\begin{definition}\label{d:newline} 
Let $K(\widetilde{\pi}_S,\gu)=K_0(\gu)\prod\limits_{v\notin
 S\cup\{\gu\}}K_1(v^{m_v})\prod\limits_{v\in S}K'_v$ 
  where $m_v$ is the conductor of $\pi_v$ and $K_v'$ is as in \eqref{eq:nu-v}.
The $(S,\alpha_{\gu})$-{\it refined} newline of a regular $\widetilde \pi_S$ is given by
 \[N_{\widetilde{\pi}_S,\alpha_\gu}=\left(\pi_f^{K(\widetilde{\pi}_S,\gu)}\right)_{\gm_{\widetilde{\pi}_S}}=\left(\pi_f^{K(\widetilde{\pi}_S,\gu)}\right)[\gm_{\widetilde{\pi}_S}],\]
 where $[\gm_{\widetilde{\pi}}]$ denotes the subspace annihilated by $\gm_{\widetilde{\pi}}$. 
 While for $v\in S_p$ the $U_{\varpi_v}$-eigenvalue $\nu_v(\varpi_v)$ of any $\phi \in N_{\widetilde{\pi},\alpha_\gu}$ depends on the choice of
$\varpi_v$, its $p$-adic valuation is independent of it. 
\end{definition}

\begin{definition} \label{d:non-critical-slope} 
 The {\it slope} $h_{\widetilde{\pi}_v}$ of $\widetilde{\pi}_v=(\pi_v, \nu_v)$ is defined as the 
 $p$-adic valuation of 
\[\nu_v(\varpi_v)\prod_{\sigma\in \Sigma_v}\sigma(\varpi_v)^{(k_\sigma+\sw-2)/2}.\]
 
 We say that the refinement $\widetilde{\pi}_v$ has {\it non-critical slope} if $ e_v h_{\widetilde{\pi}_v} < \min\limits_{\sigma\in \Sigma_v} (k_\sigma-1)$. 
 
 For $S\subset S_p$, we say that
 $\widetilde{\pi}_S$ has {\it non-critical slope} if $\widetilde{\pi}_v$ has non-critical slope for each $v \in S$. 
 
 Finally, we say that  $\widetilde{\pi}=\widetilde{\pi}_{S_p}$ has 
 {\it very  non-critical slope} if the following inequality holds:
\begin{align}\label{eq:very-non-critical-slope}
\sum\limits_{v\in S_p}e_{v}h_{\widetilde{\pi}_v} < \min_{\sigma\in \Sigma} (k_\sigma-1).
\end{align}

 \end{definition}

\subsection{Normalized $(S,\gu)$-refined eigenforms and Whittaker functions}\label{ss:whittaker}
Let $\widetilde{\pi}_S= (\pi, \{\nu_v\}_{v\in S})$ be a regular $S$-refinement of a nearly
finite slope refinement of a cuspidal cohomological automorphic representation $\pi$ of
$G(\A)$. In this section we use Whittaker models to choose a basis $\phi_{\widetilde{\pi}_S,\alpha_\gu}$ of the line 
$N_{\widetilde{\pi}_S,\alpha_\gu}$ from Def. \ref{d:newline} for which a suitable zeta 
integral yields the Jacquet-Langlands $L$-function of
$\pi$. The cusp form $\phi_{\widetilde{\pi}_S,\alpha_\gu}$, divided by a suitable
complex period, will yield an overconvergent cohomology class to
which we will attach a $p$-adic $L$-function in \S\ref{ss:p-adic L-functions attached}.

The global Whittaker model $\cW(\pi,\psi)$ of $\pi$ can be written
as a restricted tensor product of local Whittaker models $\cW(\pi_v,\psi_v)$, with respect to
$W_v^\circ\in \cW(\pi_v,\psi_v)$ for $v$ outside a finite set of bad places, where $W_v^\circ\in \cW(\pi_v,\psi_v)^{\GL_2(\cO_v)}$ is
such that $W_v^\circ(1)=1$. 
To relate values of complex $L$-functions to Whittaker integrals we will use the isomorphism
\begin{align}\label{eq:global-whittaker}
\pi \xrightarrow{\sim} \cW(\pi,\psi), \quad \phi\mapsto W_{\phi}(g)= \int_{\A_{F}/F}\phi\left( \begin{pmatrix} 1 & x
\\ 0 & 1 \end{pmatrix} g \right)\psi(-x) dx,
\end{align}
whose inverse is given by the  Fourier expansion $\displaystyle \phi(g)=\sum_{\xi\in F^\times}
W_{\phi}(\begin{psmallmatrix} \xi &\\& 1\end{psmallmatrix} g)$.

Given any collection $W_v\in \cW(\pi_v,\psi_v)$ such
that $W_v=W_v^\circ$ for almost all $v$, the tensor $\otimes
W_v$ lies in $\cW(\pi,\psi)$ and therefore is of the
form $W_\phi$ for some $\phi\in \pi$. We remark that for any choice of 
isomorphisms $\cW(\pi_v,\psi_v)\simeq \pi_v $ sending $W_v$ to $\phi_v$ 
such that $\phi_v=\phi_v^\circ$ for almost all $v$ ($\phi_v^\circ$ being the vector relative to which the restricted tensor product $\otimes'\pi_v$ is defined), $\phi$ and $\otimes\phi_v$ differ by a scalar, 
and hence $\phi$ is a pure tensor itself.

We now specify explicitly such a collection of Whittaker
functions, beginning with places $v\not\in S\cup \Sigma $ for which $\pi_v^{K_1(v^{m_v})}$ is a line. 
Let $W_v^{\new}$ be the generator of the line $\cW(\pi_v,\psi_v)^{K_1(v^{m_v})}$ 
given by the following formulas (see \cite[\S3.3.1]{raghuram-tanabe} and \cite[Thm.4.6.5]{bump}).
\begin{enumerate}
\item If $v\neq \gu$ and $\pi_v$ is the unramified principal series with
characters $\chi_{1,v}$ and $\chi_{2,v}$ then
\begin{align}\label{eq:whit-spherical}
W_v^{\new}\left(\begin{psmallmatrix} \varpi_v^{m-\delta_v}&\\&
1\end{psmallmatrix}\right)=
\begin{cases}
q_v^{-m/2}\sum\limits_{l=0}^m\chi_{1,v}(\varpi_v)^{l} \chi_{2,v}(\varpi_v)^{m-l}&m\geqslant 0,\\
0 & m<0.
\end{cases}\end{align}
\item If $\pi_v$ is a principal series with unramified $\chi_{1,v}$ and ramified $\chi_{2,v}$ then
\begin{align}\label{eq:whit-principal}
W_v^{\new}\left(\begin{psmallmatrix} \varpi_v^{m-\delta_v}&\\&
1\end{psmallmatrix}\right)=
\begin{cases}
q_v^{-m/2}\chi_{1,v}(\varpi_v)^m&m\geqslant 0,\\
0 & m<0.
\end{cases}\end{align}
\item If $\pi_v$ is a twist of the Steinberg representation by the unramified character $\chi_v$ then
\begin{align}\label{eq:whit-steinberg}
W_v^{\new}\left(\begin{psmallmatrix} \varpi_v^{m-\delta_v}&\\&
1\end{psmallmatrix}\right)=
\begin{cases}
q_v^{-m}\chi_{v}(\varpi_v)^m&m\geqslant 0,\\
0 & m<0.
\end{cases}\end{align}
\item In all other cases $ W_v^{\new}
\left(\begin{psmallmatrix} \varpi_v^{m-\delta_v}&\\&
1\end{psmallmatrix}\right)=
\begin{cases}
1&m\geqslant 0,\\
0 & m<0.
\end{cases}$
\end{enumerate}
Denoting $\{\alpha_{\gu},\beta_{\gu}\}=
\{\chi_{1,\gu}(\varpi_{\gu})\sqrt{q_{\gu}},\chi_{2,\gu}(\varpi_{\gu})\sqrt{q_{\gu}}\}$ the Hecke parameters of $\pi_{\gu}$ we let 
\[W_{\gu}^{\alpha}\left( \begin{psmallmatrix}
y&\\&1\end{psmallmatrix}\right) =W^{\new}_{\gu}\left( \begin{psmallmatrix}
y&\\&1\end{psmallmatrix}\right)-\beta_{\gu} q_{\gu}^{-1}W_{\gu}\left( \begin{psmallmatrix}
y\varpi_{\gu}^{-1}&\\&1\end{psmallmatrix}\right).\]

Then we have $W_{\gu}^{\alpha}\left( \begin{psmallmatrix}
\varpi_{\gu}^m&\\&1\end{psmallmatrix}\right)=
\begin{cases}
q_\gu^{-m} \alpha_\gu^m &m\geqslant 0,\\
0 & m<0.
\end{cases} $

Finally, we specify $v$-refined Whittaker functions at $v\in S$. 
Recall that $\nu_v$ is a one-dimensional sub
of the Weil-Deligne representation attached to  $\pi_v$ (assumed non-supercuspidal). 

Let $W'_v\in \cW(\pi_v\otimes \nu_v^{-1},\psi_v)$ be the new vector chosen as above. 
If $\pi_v\otimes \nu_v^{-1}$ is ramified we let 
\[W_v^{\nu}=\nu_v(\varpi_v)^{\delta_v}(\nu_v\circ\det) \cdot W'_v\in \cW(\pi_v,\psi_v).\]
When $\pi_v\otimes \nu_v^{-1}$ is an unramified principal
series with characters $|\cdot|^{1/2}\neq |\cdot|^{1/2}\chi_{2,v}/\chi_{1,v}$ we let
\[W_v^{\nu} =\nu_v(\varpi_v)^{\delta_v}(\nu_v\circ\det) \cdot\left(W'_v-\tfrac{\chi_{2,v}(\varpi_v)}{\chi_{1,v}(\varpi_v)}
W'_v\left(\cdot \begin{psmallmatrix}
\varpi_v^{-1}&\\&1\end{psmallmatrix}\right)\right)\in \cW(\pi_v,\psi_v).
\]
Formulas \eqref{eq:whit-spherical}, \eqref{eq:whit-principal}, and \eqref{eq:whit-steinberg} then imply:
\begin{align}\label{eq:whit-p-stabilized}
W_v^{\nu}\left(\begin{psmallmatrix} \varpi_v^{m-\delta_v}&\\&
1\end{psmallmatrix}\right)=
\begin{cases} q_v^{-m}\nu_v(\varpi_v)^m&m\geqslant 0,\\ 0 & m<0. \end{cases}\end{align}

\begin{lemma}
The image of $W_{\widetilde{\pi}_S,\alpha_\gu,f}=W_{\gu}^{\alpha}\otimes\bigotimes\limits_{v\notin
 S\cup\{\gu\}}W_v^{\new}\otimes\bigotimes\limits_{v\in S}W_v^\nu$ in $\pi_f$ is a basis of $N_{\widetilde{\pi}_S,\alpha_\gu}$.
\end{lemma}
\begin{proof} The statement is clear when $v\nmid p \gu$ since any isomorphism 
$\cW(\pi_v,\psi_v)\simeq \pi_v $ matches the new lines, as well as in the case $v=\gu$ as by construction $W_{\gu}^{\alpha}$ has $U_{\gu}$-eigenvalue $\alpha_{\gu}$. Suppose that $v\in S_p$. 
The function $W_v^{\nu}$ is $K_v'$-invariant as both $W'_v$
 and $\nu_v\circ \det$ are. The
fact that $U_{\delta} W_v^{\nu}=\nu_v(\delta)W_v^{\nu}$
for $\delta\in \cO_v^\times$ then follows from the $K_v$-invariance of $W'_v$. Finally, 
$U_{\varpi_v}$ fixes the refined new vector in $\cW(\pi_v\otimes \nu_v^{-1},\psi_v)$, hence 
acts as $\nu_v(\varpi_v)$ on $W_v^{\nu}$.
\end{proof}

For $\sigma\in \Sigma$ we choose $W_\sigma\in\cW(\pi_\sigma,\psi_\sigma)$ such that $W_\sigma(\cdot r(\theta))=e^{-i k_\sigma \theta} W_\sigma$ and 
$W_\sigma\left(\begin{psmallmatrix} y&\\&1\end{psmallmatrix}\right)=y^{(k_\sigma+\sw)/2}e^{-2\pi y}$ for all $y>0$ (see \cite{bump,raghuram-tanabe}). Since $\pi_\sigma\simeq \pi_\sigma\otimes \sgn_\sigma$,  we have: 

\begin{lemma}\label{W-infty}
  $W_\sigma$ has  support in $G_\sigma^+$ and its image  in $\pi_\sigma$ belongs to the line generated by 
$\phi_\sigma$. 
\end{lemma}

\begin{definition}\label{d:newform}
We define the normalized $(S,\alpha_\gu)$-refined newform as the cusp form $\phi_{\widetilde{\pi}_S,\alpha_\gu}\in \pi$ 
whose image under the isomorphism \eqref{eq:global-whittaker} corresponds to the pure tensor
\[W_{\widetilde{\pi}_S,\alpha_\gu}=W_{\widetilde{\pi}_S,\alpha_\gu,f}\otimes \bigotimes_{\sigma\in \Sigma}W_\sigma.\]
\end{definition}

We end this section by computing the local Whittaker integrals
that yield the local Euler factors of the complex $L$-function. 
The following proposition shows how the choice of level $\gf$ in the automorphic symbols 
in \S \ref{automorphic-symbols} reflects in the local Whittaker integrals. 

\begin{proposition}\label{p:local-whittaker-integrals}
Suppose $v\in S$ and $\chi_v$ is a finite order character of $F_v^\times$. Then
 \[Z_v=\int_{F_v^\times}\chi_v(y) W_v^{\nu}
\left(\begin{psmallmatrix}
y\varpi_v^{n_v} &y\\&1\end{psmallmatrix}\right)|y|_v^{s-1}d^\times y =
q_v^{\delta_v(s-1)} \chi_v(\varpi_v^{-\delta_v}) (q_v^{-1}\nu_v(\varpi_v))^{n_v} \frac{q_v}{q_v-1} Q(\chi_v\nu_v,s)\]
for $\Re(s)$ sufficiently large, where $c_v$ denotes the conductor of $\chi_v\nu_v$, and 
\[Q(\chi_v\nu_v,s)=\begin{cases}
q_v^{s c_v} (\chi_v\nu_v)(\varpi_v^{\delta_v}) \tau(\chi_v\nu_v,\psi_v,d_{\chi_v\nu_v})&, \text{ if } n_v \geqslant c_v \geqslant 1,\\
\left(1-\frac{(\chi_v\nu_v)(\varpi_v)}{q_v^s}\right)^{-1}\left(1-\frac{q_v^{s-1}}{(\chi_v\nu_v)(\varpi_v)}\right)&, \text{ if } n_v >c_v=0,\\
\left(1-\frac{(\chi_v\nu_v)(\varpi_v)}{q_v^s}\right)^{-1}\left(1-\frac{1}{q_v}\right) &, \text{ if } n_v=c_v=0.
\end{cases}\]
\end{proposition}

\begin{proof} 
 Since $W_v^{\nu}\in \cW(\pi_v,\psi_v)$ we have $W_v^{\nu}\left(\begin{psmallmatrix}
y\varpi_v^{n_v}&y\\&1\end{psmallmatrix}\right)=\psi_v(y)W_v^{\nu}\left(\begin{psmallmatrix}
y\varpi_v^{n_v}&\\&1\end{psmallmatrix}\right)$. Using \eqref{eq:whit-p-stabilized}:
\begin{align*}
Z_v&=\nu_v(\varpi_v)^{n_v+\delta_v}\int_{F_v^\times}(\chi_v\nu_v)(y) \psi_v(y)
(\nu_v(\varpi_v^{-\delta_v}) (\nu_v^{-1}\circ\det)\cdot W_v^{\nu}) \left(\begin{psmallmatrix}y\varpi_v^{n_v}&\\&1\end{psmallmatrix}\right)|y|_{v}^{s-1} d^\times y=\\
&=\nu_v(\varpi_v)^{n_v+\delta_v} \sum_{m\geqslant -n_v-\delta_v}q_v^{-n_v-\delta_v-ms}\int_{\cO_v^\times}(\chi_v\nu_v)(u \varpi_v^m)\psi_v(u\varpi_v^m)d^\times u.
\end{align*}

If $\chi_v\nu_v$ is ramified the above integral vanishes except for
$m=-\delta_v-c_v\geqslant -\delta_v-n_v$ in which case its value is the
Gauss sum $\tau(\chi_v\nu_v,\psi_v,d^{\times}))= \frac{q_v-1}{q_v}
\tau(\chi_v\nu_v,\psi_v,d_{\chi_v\nu_v})$ (see \eqref{eq:gauss}).

If $\chi_v\nu_v$ is unramified ($c_v=0$) then we have: 
\begin{align*}
\frac{((\chi_v\nu_v)(\varpi_v)q_v^{-s})^{\delta_v}}{ (q_v^{-1}\nu_v(\varpi_v))^{n_v+\delta_v}}\cdot Z_v= Q(\chi_v\nu_v,s)=
\sum_{m\geqslant -n_v}( (\chi_v\nu_v)(\varpi_v)q_v^{-s})^{m}
\int_{\cO_v^\times}\psi_v(u\varpi_v^{m-\delta_v})d^\times u,
\end{align*}
which is computed using the formula
$\displaystyle \int_{\cO_v^\times}\psi_v(u \varpi_v^{m-\delta_v})d^\times
u=\begin{cases} 1&, \text{ if } m\geqslant 0,\\
\frac{1}{1-q_v}&, \text{ if } m=-1,\\
0 &, \text{ if } m<-1.\end{cases}$
\end{proof}

\subsection{Periods for $(S,\gu)$-refined newforms}\label{ss:periods}
The normalized $(S,\gu)$-refined newform $\phi_{\widetilde{\pi}_S,\alpha_\gu}$ of the previous section is a Hilbert cusp form in the following sense:
\begin{definition}\label{hmf-def} A {\it holomorphic Hilbert automorphic form} of level $K$ and weight $(k,\sw)$ is a function 
	$\phi:G(\Q)\!\setminus G(\A)/K\to \C$ such that for all $\sigma\in \Sigma$, $z \in F_\sigma^\times $ and $r(\theta) \in K_\sigma^+ $
	\[\phi(\cdot \begin{psmallmatrix}
z&\\&z\end{psmallmatrix} r(\theta))= z^\sw e^{-i\theta k_\sigma}\phi \]
	and, for all $g_f\in G(\A_f)$, the function $\phi\left(g_f \begin{psmallmatrix} y_\infty & x_\infty \\ 0 & 1\end{psmallmatrix} \right)$ is holomorphic in $x_\sigma+ i y_\sigma$ in the upper half plane for every $\sigma\in \Sigma$. It is a cusp form if $\int_{\A_F/F} \phi\left(\left(\begin{smallmatrix} 1 &x \\ 0 & 1 \end{smallmatrix}\right)g\right) dx=0$ for all $g\in G(\A)$. 
\end{definition}

Note that a classical weight $k$ modular form for $F=\Q$ has $\sw=2-k$. 
The restriction to $G(\A_f) G_\infty^+$ of the Fourier expansion of $\phi$ as above is supported on totally positive elements, {\it i.e.},
\begin{align}\label{eq:Fourier}
\phi(g)=\sum_{\xi\in F^\times_+}
W_{\phi}\left(\begin{psmallmatrix} \xi &\\& 1\end{psmallmatrix} g\right), \text{ for all } g\in G(\A_f) G_\infty^+. 
\end{align}
By Lemma \ref{W-infty}
the normalized holomorphic cusp form
$\phi_{\widetilde{\pi}_S,\alpha_\gu}$ can be written as a pure tensor of the form $\phi_{\widetilde{\pi}_S,\alpha_\gu}= \phi_{\widetilde{\pi}_S,\alpha_\gu,f} \otimes \bigotimes_{\sigma\in \Sigma}\phi_\sigma$. Recall
that for a character $\epsilon: \{\pm 1\}^\Sigma\to \{\pm 1\}$ we constructed in \eqref{eq:XI} a cohomology class
$\Xi_{\pi_\infty}^\epsilon \in \rH^d(\gg_\infty, K_\infty^+, L_{k,\sw}^\vee(\mathbb{C})\otimes \pi_\infty)$ yielding a map
\begin{align}\label{Theta}
\Theta^\epsilon_{\pi}: \pi_f^K\xrightarrow{ \Xi_{\pi_\infty}^\epsilon\otimes} \rH^d(\gg_\infty, K_\infty^+, L_{k,\sw}^\vee(\C)\otimes \pi^K)^{\epsilon}\hookrightarrow 
	\rH^d_{\cusp}(Y_K, \cL_{k,\sw}^\vee(\C))^{\epsilon}
\end{align}
where $K=K(\widetilde{\pi}_S,\gu)$ is as in Def. \ref{d:newline}. For $E$ and $\gm_{\widetilde{\pi}_S}$ is as in Def. \ref{d:mpi-tilde} we consider the line:  
\begin{align}\label{eq:E-line}
\rH^d_{\cusp}(Y_K, \cL_{k,\sw}^\vee(E))^\epsilon_{\gm_{\widetilde{\pi}_S}}.
\end{align}

\begin{definition} \label{of-the-period-omega-pi}
Given a basis $b_{\widetilde{\pi}_S,\alpha_\gu}^{\epsilon}$ of the $E$-line \eqref{eq:E-line}, we 
let $\Omega_{\widetilde{\pi}_S}^\epsilon\in\C^\times$ be
such that $\Theta^\epsilon_{\pi}(\phi_{\widetilde{\pi}_S,\alpha_\gu,f})=\Omega_{\widetilde{\pi}_S}^\epsilon b_{\widetilde{\pi}_S,\alpha_\gu}^{\epsilon}$. When $\epsilon$ is the trivial character we denote this period simply by $\Omega_{\widetilde{\pi}_S}$.
\end{definition} 
Since $\alpha_\gu, \beta_{\gu}\in E$ the period can be taken the same for either choice of Hecke parameter at $\gu$.

The precise choice of a $p$-refined automorphic newform in \S\ref{ss:whittaker} allows us to 
 prove the following formula describing the behavior of the periods $\Omega_{\widetilde\pi}^\epsilon$  under  twisting by characters. 
 Here $S=S_p$.

\begin{prop}\label{p:twisting-periods}
One can choose the bases  $b_{\widetilde{\pi},\alpha_\gu}^{\epsilon}$ in Def. \ref{of-the-period-omega-pi} such that
for every algebraic Hecke character $\chi$ of weight $w$ and $p$-power conductor, one has 
$\Omega_{\widetilde{\pi\otimes\chi}}^{\epsilon\chi_\infty}= i^{-d w}\chi_f(\varpi_\gd)\Omega_{\widetilde{\pi}}^{\epsilon\omega_{p,\infty}^w}$
for any character $\epsilon: \{\pm 1\}^\Sigma\to \{\pm 1\}$, 
where  $\widetilde{\pi\otimes\chi}= (\pi\otimes\chi, \{\nu_v\otimes\chi_v\}_{v\in S_p})$. 
\end{prop}
\begin{proof}
We drop $\alpha_\gu$ to avoid cumbersome notation. 
By \eqref{eq:whit-p-stabilized} and Definition \ref{d:newform}  we have 
$W_{\widetilde{\pi\otimes\chi},f}=\chi(\varpi_\gd) W_{\widetilde{\pi},f}\cdot \chi_f\circ\det$ (compare with  \cite[Thm.1.1]{raghuram-shahidi:periods}),  hence $\phi_{\widetilde{\pi\otimes\chi}}=\chi(\varpi_\gd) \phi_{\widetilde{\pi}}\cdot \chi\circ\det$. 

Since $\begin{psmallmatrix}-1 & \\ &1 \end{psmallmatrix}\cdot\eval^{k,\sw+2w}_{i}= (-1)^w 
\begin{psmallmatrix}-1 & \\ &1 \end{psmallmatrix}\cdot\eval^{k,\sw}_{i}$,  definition \eqref{eq:XI} implies that: 
\begin{align}\label{g-K-relation}
\Xi_{\pi_\infty\otimes\chi_\infty}^\epsilon \otimes
\phi_{\widetilde{\pi\otimes\chi},f}
& =i^{\frac{((2-\sw-2w)t-k)}{2}} \sum_{s_\infty\in \{\pm 1\}^\Sigma} \epsilon(s_\infty) 
(s_\infty\cdot (w_\infty^*\otimes \phi_{\widetilde{\pi\otimes\chi}} \otimes \eval^{k,\sw+2w}_{i}))=\\
&=i^{-d w} \chi(\varpi_\gd)  \Xi_{\pi_\infty}^{\epsilon\chi_\infty\omega_{p,\infty}^w}\otimes
(\phi_{\widetilde{\pi},f}\cdot \chi_f\circ\det). \nonumber
\end{align}

For $v\in S_p$ let $K'_v$ be as in \eqref{eq:nu-v} and 
 $K''_v$ be the analogous subgroup for $(\pi_v\otimes\chi_v, \nu_v\otimes\chi_v)$. 
Letting $K''=K^p\prod_{v\in S_p} (K'_v\cap K''_v)$,  we see that $\chi_f$ factors through 
$\pi_0(Y_{K''})\simeq \A_{F,f}^\times/F_+^\times\det(K'')$. Writing 
 $\mathrm{pr}^*b_{\widetilde{\pi}}^{\epsilon}=(c_\eta)_{\eta\in \pi_0(Y_{K''})}\in \bigoplus\limits_{\eta\in \pi_0(Y_{K''})}
\rH^d_{\cusp}(Y_{K''}[\eta], \cL_{k,\sw}^\vee(E)_{|Y_{K''}[\eta]})$, where $\mathrm{pr}: Y_{K''}\to Y_{K}$ is the natural projection, 
one sees that $(\chi_f(\eta) c_\eta)_{\eta\in \pi_0(Y_{K''})}$ is $E$-rational as well, as the rational structure on its Betti cohomology 
is imposed component-wise. Identifying the local systems $\cL_{k,\sw}^\vee(E)$ and $\cL_{k,\sw+2w}^\vee(E)$, and 
choosing the basis $b_{\widetilde{\pi\otimes\chi}}^{\epsilon\chi_\infty\omega_{p,\infty}^w}$ to correspond to $(\chi_f(\eta) c_\eta)_{\eta\in \pi_0(Y_{K''})}$ then yields the desired  relation in view of  \eqref{g-K-relation} and Definition  \ref{of-the-period-omega-pi}.  
\end{proof}

\section{Overconvergent cohomology and partial nearly finite slope families}\label{s:overconvergent-cohomology}

In this section we introduce overconvergent cohomology spaces for individual weights and in families which naturally interpolate the spaces $\rH^{i}_c(Y_K,\cL_{k,\sw}^\vee(L))$ for cohomological weights $(k,\sw)$. Moreover, we establish, in Theorem \ref{t:classicality}, a classicality criterion from which we deduce in Corollary \ref{p:non-criticality} that 
cohomological $\widetilde\pi$ of non-critical slope are non-critical. 
Furthermore, we construct a cohomological cuspidal Hilbert eigenvariety in a neighborhood of such a non-critical $\widetilde\pi$ and show that it is etale over the weight space. 
 We remark that our local families do not {\it a priori } patch in Buzzard's eigenvarieties machine as one cannot guarantee the projectivity of overconvergent cohomology groups beyond $\rH^0$ and $\rH^1$. Instead we adapt Hida's axiomatic construction of nearly ordinary families to the rigid analytic context and prove equidimensionality and etaleness using the fact that cuspidal cohomology is supported in middle degree.

Finally we construct partial nearly finite slope $p$-adic families that impose no restriction on the local representation at a set of places $S\subset S_p$. These results are crucial to carry out the construction of $p$-adic $L$-functions for families, and to control the behavior of the local representation at $v\in S_p\bs S$ in partial families with fixed weights at $\Sigma_v$.

\subsection{Weight spaces} \label{weights}

Let $\cX$ be the $(d+1)$-dimensional rigid analytic space over $\Q_p$ such that:
 \begin{align}\label{eq:coh-weight-space}
 \cX(\C_p)=\left\{ \lambda \in \Hom_{\mathrm{cont}} 
  (T(\Z_p),\C_p^{\times}) \Big{|} \,\exists \,\sw_\lambda \in \Hom_{\mathrm{cont}} 
  (\Z_p^\times ,\C_p^{\times}),\,  \lambda\left(\begin{psmallmatrix} z &\\& z\end{psmallmatrix}\right)=\sw_\lambda (z^t)\right\}.
   \end{align}

  Letting $k_\lambda(z)=\lambda\left(\begin{psmallmatrix} z &\\& z^{-1}\end{psmallmatrix}\right)\cdot z^{2 t} $ for $z\in (\cO_F\otimes\Z_p)^\times$ there is a finite morphism: 
  \begin{align}\label{eq:weight-finite-map}
  \cX(\C_p)\to \Hom_{\mathrm{cont}}\left((\cO_F\otimes\Z_p)^\times\times \Z_p^\times,\C_p^\times\right),
  \lambda\mapsto (k_\lambda,\sw_\lambda).
     \end{align}
 The cohomological weights of $G$ (see Def. \ref{d:coh-weight}) all belong to $\cX$ and 
 are very Zariski dense in it.

 Given an affinoid $\cU \subset \cX$ we let $\cO(\cU)$ denote the ring of its rigid analytic functions and consider the universal locally analytic character (see \cite[Lem.3.4.6]{urban}): 
 \[\left\langle \cdot \right\rangle_{\cU}: T(\Z_p) \to \cO(\cU)^{\times},\,\,\, t\mapsto 
 (\lambda\mapsto \lambda(t)).\]
 
\begin{definition}\label{defn:XKS}
Fix a cohomological weight $(k,\sw)$. 
Given a subset $S$ of $S_p$ we let $\cX_{S}$, resp. $\cX'_S $, denote the 
rigid analytic subspaces of $\cX$ parametrizing weights which agree with $(k,\sw)$ on
\[  \prod\limits_{v\in S_p\!\setminus S} \begin{psmallmatrix} \cO_v^{\times} & 0 \\ 0 & \cO_v^{\times} \end{psmallmatrix}, \, \text{ resp.  } 
 \prod\limits_{v\in S_p\!\setminus S} \begin{psmallmatrix} \cO_v^\times & 0 \\ 0 & 1\end{psmallmatrix}.\]
 \end{definition}

We have $\cX_{S_p}=\cX'_{S_p}=\cX$ and $\dim \cX_{S}=\dim \cX'_{S}-1= |\Sigma_S|$ for any $S\subsetneq S_p$.
The space $\cX'_S$ is the natural place to consider partially improved $p$-adic $L$-functions
(see \S\ref{ss:norms} and \S\ref{ss:improved}), whereas the partially finite slope families of Hilbert modular cusp forms live on its subspace $\cX_S$. We thank the referees for their suggestion to highlight the difference between these spaces which is important for understanding the results of this paper.

\subsection{Overconvergent modules}\label{sect:oc-modules} 
 Following \cite{urban} we introduce certain modules that will be later used to define  overconvergent sheaves on the Hilbert modular variety.

Let $L$ be a finite extension of $\Q_p$ and $X$ be an open compact subset 
of a finite dimensional $\Q_p$-vector space. 
Given $n \in \Z_{\geqslant 0}$, we let $A_n(X, L)$ denote the Banach $L$-vector space of $n$-locally analytic functions on $X$ and $D_n(X, L)$ its Banach dual (see \cite[\S3.2.1]{urban}). More generally, given an admissible affinoid $\cU\subset \cX$
we consider the orthonormalizable Banach $\cO(\cU)$-modules $A_n(X, \cO(\cU)) = A_n(X, L)\widehat{\otimes}_{L}\cO(\cU)$ and 
$D_n(X, \cO(\cU)) = D_n(X, L)\widehat{\otimes}_{L}\cO(\cU)$ (see \cite[\S2.2]{hansen}). 
The space 
$A(X, \cO(\cU))= \bigcup_{n \in \Z_{\geqslant 0}}A_n(X, \cO(\cU))$ of locally analytic $ \cO(\cU)$-valued functions of $X$, endowed with the inductive limit topology, is a Fr\'echet $\cO(\cU)$-module. The natural maps $ D_{n+1}(X, \cO(\cU))\to D_n(X, \cO(\cU))$ are compact, 
and $D(X, \cO(\cU))= \varprojlim D_n(X, \cO(\cU))$
 is a compact Fr\'echet $\cO(\cU)$-module. 
 There is functorial in $\cO(\cU)$ pairing: 
\begin{align}\label{eq:pairing}
\langle\cdot , \cdot \rangle: D(X, \cO(\cU))\times A(X, \cO(\cU))\to \cO(\cU), 
\end{align}
yielding $D(X, \cO(\cU))\hookrightarrow \Hom_{\cO(\cU)}(A(X, \cO(\cU)), \cO(\cU))$ but, 
as observed in \cite[Rem.3.1]{bellaiche:critical}, the natural injective maps $D_n(X, \cO(\cU))\hookrightarrow \Hom_{\cO(\cU)}(A_n(X, \cO(\cU)), \cO(\cU))$ need not be surjective. 
The above constructions applies to $X=\cO_F\otimes\Z_p$
considered as an open compact subset of $\Q_p^d$. 

We fix a cohomological weight $(k, \sw)$ and a subset $S$ of $S_p$. 
We will now introduce certain {\it partial} overconvergent distributions over an admissible affinoid $\cU_S\subset \cX_S$ containing $(k, \sw)$. These distributions will allows us to construct $p$-adic families parametrized by $\cU_S$ containing $\pi$ even when 
its local components $\pi_v$ for $v\in S_p\!\setminus S$ have critical slope,
{\it e.g.} are supercuspidal, and to attach to them what appears to be a genuinely new kind of $p$-adic $L$-function
 (see \S\ref{ss:partial}). 
 
 We consider the semigroup
\begin{align}\label{eq:Lambda-S}
\varLambda_S=\prod_{v\in S_p\!\setminus S} \GL_2(F_v)
 \prod_{v\in S} \GL_2(F_v) \cap 
\left( F_v^\times\cdot \begin{psmallmatrix}\cO_v & \cO_v \\ \varpi_v \cO_v & \cO_v^\times \end{psmallmatrix} \right)
\end{align}
and we define the partial Iwahori subgroup $I_S =
\varLambda_S\cap G(\mathbb{Z}_p)=\prod\limits_{v\in S_p\!\setminus S} \GL_2(\cO_v) \cdot \prod\limits_{v\in S}I_v$.

Let $K\subset G(\A_{f})$ be an open compact subgroup satisfying \eqref{neat} such that  $K_p \subset I_S$. In particular, we allow $K_v$ to be the maximal compact 
subgroup $\GL_2(\cO_v)$ at places $v\in S_p\!\setminus S$. We let 
\begin{align}\label{eq:partial-ana-fct}
A_{S,\cU_S} = A(\cO_{F,S}, \cO(\cU_S))\otimes_L
\bigotimes_{\sigma\in  \Sigma_{S_p\!\setminus S}} L_{k_\sigma,\sw}(L),
\end{align}
be the subspace of $A(\cO_F\otimes\Z_p, \cO(\cU_S))$ consisting of functions which are
polynomial of degree at most $(k_\sigma-2)_{\sigma\in \Sigma_v}$ in the variables 
$(z_\sigma)_{\sigma\in \Sigma_v}$ for all $v\in S_p\!\setminus S$.
For $\lambda\in \cU_S(L)$ we let $A_{S,\lambda}= A_{S,\{\lambda\}}$. 
For $n\in \Z_{\geqslant 0}$ we let $A_{S,\lambda,n}=A_{S,\lambda}\cap A_n(\cO_F\otimes\Z_p, L)$
and we denote by $D_{S,\lambda,n}$ its topological dual. Finally we consider the 
Banach $\cO(\cU_S)$-module $D_{S,\cU_S,n}= D_{S,\lambda,n}\widehat{\otimes}_{L}\cO(\cU_S)$ and 
the compact Fr\'echet $\cO(\cU_S)$-module $D_{S,\cU_S}= \varprojlim D_{S,\cU_S, n}$. 

\begin{defn} \label{d:action-functions}
We consider the following continuous right action of $\gamma\in I_S$ on $f \in A_{S,\cU_S}$ 
 \begin{align}\label{iwahori-action-bis}
 f_{|\gamma}(z)=f\left( \dfrac{az+b}{cz+d} \right) \left\langle \begin{psmallmatrix} (cz+d) & 0 \\ 0 & \det(\gamma)\cdot (cz+d)^{-1} \end{psmallmatrix} \right\rangle_{\cU_S}, 
 \text { where } z \in \cO_F\otimes\Z_p, \gamma=\begin{pmatrix}a & b \\ c & d \end{pmatrix}.
 \end{align}
 Furthermore, if for all $v\in S$ and all integers $r\geqslant s$ we let 
$ f_{\big|\begin{psmallmatrix}\varpi_v^r & 0 \\ 0 & \varpi_v^s\end{psmallmatrix} }(z)= f(\varpi_v^{r-s} z)$.
\end{defn}
A direct computation shows that the above actions uniquely extend to a continuous right action of 
$\varLambda_S$ on $A_{S,\cU_S}$ inducing, via the pairing \eqref{eq:pairing}, a continuous left action of $\varLambda_S$ on $D_{S,\cU_S}$:
 \begin{align}\label{iwahori-action}
 (\gamma\cdot \mu)(f)= \mu(f_{|\det(\gamma)^{-1}\cdot \gamma}) \text{, for all } \mu \in D_{S,\cU_S}, f\in A_{S,\cU_S}. 
 \end{align} 
 
 For $v\in S$ one has $\left(\begin{psmallmatrix} \varpi_v & 0 \\ 0 & 1\end{psmallmatrix}\cdot \mu\right)(f)= \mu(f(\varpi_v \cdot))$, for all  $\mu \in D_{S,\cU_S}, f\in A_{S,\cU_S}$. 
 Thus the element $\prod\limits_{v\in S}\begin{psmallmatrix} \varpi_v & 0 \\ 0 & 1\end{psmallmatrix}
 \in \varLambda_S$ induces a compact endomorphism on $D_{S,\cU_S}$ (see \cite[\S3.4.12]{urban}). 
 
For $\lambda \in \cU_S(L)$ we consider the natural 
$\varLambda_S$-equivariant specialization map
\begin{align}\label{eq:spec-map}
D_{S,\cU_S}\to D_{S,\cU_S}\otimes_{\cO(\cU_S),\lambda} L= D_{S,\lambda}.
\end{align}

Let $(k,\sw)$ be a cohomological weight (see Def. \ref{d:coh-weight}). 
Using \eqref{right-action} and \eqref{left-action} one sees that the 
natural injection $L_{k,\sw}(L)\hookrightarrow A_{S,(k,\sw)}$ is equivariant for the right 
$I_S$-action, yielding a natural homomorphism of left $I_S$-modules 
\begin{align}\label{eq:vartheta}
\vartheta_S :D_{S,(k,\sw)}\to L_{k,\sw}^\vee(L),
\end{align} 
However, $\vartheta_S$ is not $\varLambda_S$-equivariant, since for all $v\in S$
and $\mu\in D_{S,(k,\sw)}$, one has 
\begin{align}\label{action-uniformizer}
\vartheta_S \left(\begin{psmallmatrix}\varpi_v & 0 \\ 0 & 1 \end{psmallmatrix}\cdot \mu\right) \prod_{\sigma\in \Sigma_v}\sigma(\varpi_v)^{(2-\sw- k_{\sigma})/2}=\begin{psmallmatrix}\varpi_{v} & 0 \\ 0 &1 \end{psmallmatrix} \cdot \vartheta_S(\mu). 
\end{align}

 When $S= S_p$ we will drop it from the notations, {\it e.g.},  
 $A_{\cU}= A_{S_p,\cU_{S_p}}$, $D_{\cU}= D_{S_p,\cU_{S_p}}$, $\varLambda= \varLambda_{S_p}$.

\subsection{Slope decomposition for overconvergent cohomology} \label{ss:slope-dec}

Let $\cU$ be an $L$-affinoid. 
We consider a compact  Fr\'echet $\cO(\cU)$-module $M= \varprojlim M_n$ 
such that for all $n \in \Z_{\geqslant 0}$ the Banach $\cO(\cU)$-module $M_n$ is orthornomalizable, endowed with 
a compact endomorphism $U: M\to M$, {\it i.e.}, a system of maps $U_n: M_n\to M_n$ factoring through the natural projections $M_n\to M_{n-1}$ which are compact. 
For $h\in \Q_{\geqslant 0}$, if $M$ admits a slope $\leqslant h$ decomposition with respect to $U$ written as $M^{\leqslant h}\oplus M^{>h}$ (see \cite[Def.2.3.1]{hansen}), then $M^{\leqslant h}$ is a finitely generated Banach $\cO(\cU)$-module.

The following result is a generalization to compact Fr\'echet $\cO(\cU)$-modules
of a well-known proposition about Banach $\cO(\cU)$-modules. If $\cU' \subset \cU$ is a sub-affinoid, we let $U_{\cU'}$ denote the endomorphism of $M_{\cU'}= M\otimes_{\cO(\cU)}\cO(\cU')$ induced by $U$. 

\begin{prop} \label{dec frechet} Given any $\lambda \in \cU(L)$ and any $h \in \Q_{\geqslant 0}$, there exists an admissible $L$-affinoid neighborhood $\cU' \subset \cU$ of $\lambda$ such that 
	$M_{\cU'}$ admits a slope $\leqslant h$ decomposition with respect to $U_{\cU'}$.
\end{prop}
\begin{proof} When $\cU$ is a point, {\it i.e.}, $\cO(\cU)$ is a $p$-adic field, then this is proven in \cite[Lem.2.3.13]{urban}.

By \cite[Prop.2.3.3]{hansen} applied to the orthonormalizable Banach $\cO(\cU)$-module 
$M_n$ endowed with the compact endomorphism $U_n: M_n \rightarrow M_n$, 
there exists an $L$-affinoid neighborhood $\cU' \subset \cU$ of $\lambda$ such that 
 $M_{n,\cU'}$ admits a slope $\leqslant h$ decomposition with respect to $U_{n,\cU'}$ as
 $M_{n,\cU'}^{\leqslant h}\oplus M_{n,\cU'}^{>h}$. Since  the Fredholm determinant  $\det(1- x\cdot U_n\mid M_n)$ is independent of $n$ (see \cite[Lem.2.7]{buzzard}), one can take the same $\cU'$ for all $n$. Passing to the limit we obtain a direct sum decomposition 
  $M_{\cU'}= \varprojlim M_{n,\cU'}^{\leqslant h}\oplus \varprojlim M_{n,\cU'}^{>h}$. 
  It remain to see that this is a slope $\leqslant h$ decomposition.   
  Since the endomorphism $U$ is compact, applying exactly the same steps as in the proof
 of \cite[Cor.2.3.4]{urban} yields that the natural maps $M_{n,\cU'}^{\leqslant h} \to 
	M_{n-1,\cU'}^{\leqslant h}$ are all isomorphisms, hence 
	$M_{\cU'}^{\leqslant h} \simeq M_{n,\cU'}^{\leqslant h} $ 
	is a finitely generated Banach $\cO(\cU')$-module for all $n$.	\end{proof}

Let $K \subset G(\A_f)$ be an open compact subgroup satisfying \eqref{neat} and $\varLambda \subset G(\A_f)$ a semi-group containing $K$. Given a Fr\'echet $\cO(\cU)$-module $M$ as above, we suppose that it is endowed with a continuous left action of $\varLambda$ such that \eqref{center} holds, and we let $\cM$ denote the associated sheaf on $Y_K$ (see \S \ref{local-sys}).

\begin{prop} \label{prop dec coho} Suppose that $x \in \varLambda$ induces a compact endomorphism on $M$. For $\lambda \in \cU(L)$ and $h \in \Q_{\geqslant 0}$ there is an admissible $L$-affinoid neighborhood $\cU' \subset \cU$ of $\lambda$ such that $\rH_{c}^\bullet(Y_K, \cM_{\cU'})$ admits a slope $\leqslant h$ decomposition with respect to the Hecke operator $[KxK]$. \end{prop}
\begin{proof} By Proposition \ref{dec frechet} there exists an admissible $L$-affinoid $\cU' \subset \cU$ containing $\lambda$ such that $M_{\cU'}$ admits a slope $\leqslant h$ decomposition with respect to the endomorphism induced by $x$.
By \cite[Lem.2]{barrera} the cohomology $\rH_{c}^{\bullet}(Y_K, \cM)$ can be 
computed by a bounded complex $\rRG_{c}^{\bullet}(K,M_{\cU'})$ whose terms are compact Fr\'echet $\cO(\cU')$-modules on which the Hecke operator $[KxK]$ acts compactly. 
Thus $\rRG_{c}^{\bullet}(K,M_{\cU'})$ admits a slope $\leqslant h$ decomposition with respect to $[KxK]$ and the proposition follows from \cite[Prop.2.3.2]{hansen}. 
\end{proof}

\begin{definition} \label{def:many-slopes} 
Let $(U_i)_{i \in I}$ be a family of $\cO(\cU)$-linear endomorphisms of an $\cO(\cU)$-module $M$. 
 Given $h_I=(h_i)_{i \in I} \in \Q_{\geqslant 0}^I$ we let 
$M^{\leqslant h_I}$ denote the subspace consisting of elements having slope $\leqslant h_i$ with respect to $U_i$ for all $i\in I$. 
\end{definition}

We fix  $S\subset S_p$ such that $K_p\subset I_S$. 
When condition \eqref{center} is satisfied by $(k,\sw)$, it is also satisfied by all weights in $\cU_S\subset \cX_S$ sufficiently small containing $(k,\sw)$, yielding a sheaf $\cD_{S,\cU_S}$ on $Y_K$.

Considering the family $(\begin{psmallmatrix} \varpi_v & 0 \\ 0 & 1\end{psmallmatrix})_{v\in S}$ of mutually commuting endomorphisms of $D_{S,\cU_S}$ and applying 
 Proposition \ref{prop dec coho} to their product, which is compact, has the following consequence.

\begin{corollary} \label{fg} 
For any $h_S \in \Q_{\geqslant 0}^S$ and any cohomological weight $(k,\sw)$ satisfying \eqref{center}, there exists an admissible affinoid $\cU_S \subset \cX_{S}$ containing $(k,\sw)$ such that $\rH_{c}^\bullet(Y_K, \cD_{S,\cU_S})^{\leqslant h_S}$ is a finitely generated $\cO(\cU_S)$-module, 
 where the slope condition is with respect to the family $(U_{\varpi_v})_{v\in S}$.
\end{corollary}

\subsection{Classicality} For  $S\subset S_p$ and $K \subset G(\A_f)$ as in \S\ref{ss:slope-dec}, 
the map resulting from \eqref{eq:vartheta} 
\begin{align}\label{specialization}
\vartheta_S: \rH^{\bullet}_{c}(Y_K, \cD_{S,(k,\sw)}) \to \rH^{\bullet}_{c}(Y_K,\cL_{k,\sw}^\vee(L)), 
\end{align}
intertwines for $v\in S$ the $U_{\varpi_v}$-action on $\rH^{\bullet}_{c}(Y_K, \cD_{S,(k,\sw)})$ 
with the action of the normalized 
\begin{align}\label{eq:U0}
U_{\varpi_v}^\circ= \left(\prod_{\sigma\in \Sigma_v}\sigma(\varpi_v)^{(k_\sigma+\sw-2)/2} \right)\cdot U_{\varpi_v}\end{align}
 on $\rH^{\bullet}_{c}(Y_K, \cL_{k,\sw}^\vee(L))$. 
We remark that the $U_{\varpi_v}^\circ$-action is independent of $\sw$. 
\begin{thm}\label{t:classicality} Let $h_S=(h_v)_{v \in S}\in \Q_{\geqslant 0}^{S}$ be such that $\displaystyle e_v h_v < \min_{\sigma \in \Sigma_v}(k_\sigma -1)$ for all $v\in S$. 
Then \eqref{specialization}	induces an isomorphism of slope $\leqslant h_S$ subspaces in the sense of 
Def. \ref{def:many-slopes}:	
	\begin{align}\label{e: claissicality}
	\vartheta_S: \rH^{\bullet}_{c}(Y_K, \cD_{S,(k,\sw)})^{\leqslant h_S} \xrightarrow{\sim}\rH^{\bullet}_{c}(Y_K, \cL_{k,\sw}^\vee(L))^{\leqslant h_S},
	\end{align}
	where we consider $\{U_{\varpi_v},v \in S\}$ on the left hand side 
	 and  $\{U_{\varpi_v}^\circ,v \in S\}$ on the right hand side.
\end{thm} 
\begin{proof} If $S=S_p$ the this follows from \cite[Thm.8.7]{barrera-williams}. 
For a general $S$ we will use a partial version of the the locally analytic BGG resolution. 
For $\sigma\in \Sigma_S$  the image of $(k,\sw)$ by a generator of the Weyl group of $G_\sigma$ yields 
a cohomological weight $(k^\sigma, \sw)$, where
$k^\sigma_{\sigma}=2-k_{\sigma}$ and $k^\sigma_{\sigma'}=k_{\sigma'}$ for all $\sigma' \in \Sigma-\{\sigma\}$. The restriction to $A_{S,(k,\sw)}$ of the map introduced in \cite[Prop.3.2.11]{urban} yields an $I_S$-equivariant map $\Theta_{S,\sigma}: A_{S,(k,\sw)}\to A_{S,(k^\sigma, \sw)}$,  whose dual $\Theta_{S, \sigma}^{\vee}: D_{S,(k^\sigma, \sw)} \rightarrow D_{S,(k,\sw)}$ is $I_S$-equivariant as well. From {\it loc. cit.} the cokernel of 
the map 
 \begin{align}\label{e: partial BGG}
 \sum_{\sigma \in \Sigma_S}\Theta_{S, \sigma}^{\vee}
: \bigoplus_{\sigma \in \Sigma_S}D_{S,(k^\sigma, \sw)}
\longrightarrow D_{S,(k,\sw)}, 
\end{align}
is given by the continuous dual of the subspace of locally algebraic functions in $A_{S,(k,\sw)}$. The proof of the Theorem then proceeds exactly as \cite[Thm.8.7]{barrera-williams}
 using the fact that the finite slope parts of the cohomology of algebraic and locally algebraic distributions coincide  (see \cite[Lem.4.3.8]{urban}), and the following computation 
  for $v \in S$ and $\mu \in D_{S,(k^\sigma, \sw)}$:
	\[\Theta_{S, \sigma}^{\vee}((\begin{smallmatrix}\varpi_v & 0 \\ 0 & 1 \end{smallmatrix})\cdot \mu)=\begin{cases} \varpi_v^{-k_{\sigma}+1}(\begin{smallmatrix}\varpi_v & 0 \\ 0 & 1 \end{smallmatrix})\cdot \Theta_{S, \sigma}^{\vee}(\mu) & \text{, if } \sigma \in \Sigma_v, \\
	 (\begin{smallmatrix}\varpi_v & 0 \\ 0 & 1 \end{smallmatrix})\cdot \Theta_{S, \sigma}^{\vee}(\mu)& \text{, if } \sigma \in \Sigma_S\!\setminus \Sigma_v.\end{cases}\qedhere \] 
	\end{proof} 

\begin{remark}The map (\ref{e: partial BGG}) is a locally-analytic analogue of Lepowsky's generalized BGG resolution for $G$ relative to the parabolic subgroup given by $\mathrm{GL}_2$ at the places outside $S$ and the upper triangular matrices at the places in $S$ (see \cite[Thm.4.3]{lepowsky}). On the other hand, in \cite{loeffler2009} the author uses the general Lepowsky's generalized BGG resolution to prove classicality results for his construction of eigenvarieties for reductive algebraic groups whose real points are compact modulo their center. 
\end{remark}

\subsection{Axiomatic control and freeness} \label{ss:axiomatic}
We will generalize a strategy due to Hida, establishing an exact control
theorem and freeness results of the overconvergent cohomology and the Hecke algebra acting on it. The motivating principle is that while, in general, one cannot establish torsion-freeness for cohomology, this can be done when the cohomology is supported in a single degree.
We begin with an adaptation of \cite[Lem.7.1]{hida} to the setting of analytic families. 

Let $A$ be a regular local
ring with maximal ideal $\gm$ and let $\cC$ be
a sub-category of the category of $A$-modules such that if $M$
is in $\cC$ and $I$ is an ideal of $A$ then $M\otimes_A
A/I$ is also in $\cC$. Henceforth $\cH^\bullet$ will denote a cohomology
functor on $\cC$ such that:
\begin{enumerate}
\item $\cH^\bullet$ sends short exact sequences to long
exact ones, and
\item $\cH^\bullet(M)$ is a finitely generated $A$-module for every $M$ in $\cC$.
\end{enumerate}

\begin{lemma}\label{l:hida control}
Suppose $M$ in $\cC$ is $A$-flat and $\cH^\bullet(M\otimes_A A/\gm)$ is supported in
degree $d$. Then 
\begin{enumerate}
\item $\cH^\bullet(M)$ is supported in degree $d$ and
$\cH^d(M)\otimes_{A}A/\gm\simeq \cH^d(M\otimes_A A/\gm)$.
\item $\cH^d(M)$ is $A$-torsion-free. In particular, if $A$ is a discrete valuation ring then $\cH^d(M)$ is a free $A$-module of rank $\dim_{A/\gm}\cH^d(M\otimes_A A/\gm)$. 
\end{enumerate}
\end{lemma}
\begin{proof} Let $T_1, T_2,\ldots, T_k$ be a regular sequence in $A$ and consider the 
filtration
\[I_0=0\subset I_1=(T_1)\subset I_2=(T_1,T_2)\subset\ldots
\subset I_k=(T_1,\ldots, T_k)=\gm.\]

For $i\neq d$ we will prove by descending induction on $r$ that
$\cH^{i}(M\otimes_{A}A/I_r)=0$. The base case $r=k$ follows from the
hypothesis. By flatness, we have a short exact sequence \[0\to M\otimes A/I_{r-1}\xrightarrow{\cdot
 T_r}M\otimes A/I_{r-1}\to M\otimes A/I_{r}\to 0.\]
 The corresponding long exact sequence yields an injection
\[\cH^{i}(M\otimes A/I_{r-1})\otimes_{A/I_{r-1}}A/I_r\hookrightarrow \cH^{i}(M\otimes A/I_{r}).\]
By the inductive hypothesis we have $\cH^{i}(M\otimes A/I_{r-1})\otimes_{A/I_{r-1}}A/I_r=0$. 
Since the $A/I_{r-1}$-module $\cH^{i}(M\otimes A/I_{r-1})$ is finitely generated, Nakayama's lemma yields $\cH^{i}(M\otimes A/I_{r-1})=0$.
Finally, let $i=d$. The long exact sequence and the vanishing result in degree $d+1$ yield
\[\cH^d(M\otimes A/I_{r-1})\otimes_{A/I_{r-1}}A/I_r\simeq \cH^d(M\otimes A/I_r),\]
and concatenating these isomorphisms for $1\leqslant r\leqslant k$ yields
part (i).

For part (ii), note that it suffices to show that $\cH^d(M)$ has no $T$-torsion where
the non-zero divisor $T$ can be assumed to be $T_1$. The arguments of part (i) imply that multiplication by $T_r$ is injective on $\cH^d(M\otimes A/I_{r-1})$ by descending induction on $r$. The case $r=1$ yields that multiplication by $T$ is injective on $\cH^d(M)$ as desired. Finally, when $A$ is a discrete valuation ring the module $\cH^d(M)$ is free of rank $\dim_{A/\gm}\cH^d(M)\otimes A/\gm=\dim_{A/\gm} \cH^d(M\otimes A/\gm)$.
\end{proof}

We will now apply the abstract paradigm of Lemma \ref{l:hida
 control} to the setting
of overconvergent sheaves and Hecke algebras. 
The rigid localization of an admissible $L$-affinoid $\cU\subset \cX$ at a point 
$\lambda\in \cU(L)$ is defined as $\cO(\cU)_\lambda=
\displaystyle \varinjlim_{\lambda\in\cU'\subset
 \cU}\cO(\cU')$, where the limit is taken over all admissible open sub-affinoids $\cU'$
 in $\cU$ containing $\lambda$. It is a local ring which contains the algebraic localization of $\cO(\cU)$ at the maximal ideal $\gm_\lambda$ at $\lambda$. For an $\cO(\cU)$-module $\cF$ we let 
 \begin{align}\label{eq:rigid-loc}
 \cF_\lambda = \cF\otimes_{\cO(\cU)}\cO(\cU)_\lambda=
 \displaystyle \varinjlim_{\lambda\in\cU'\subset \cU}(\cF\otimes_{\cO(\cU)}\cO(\cU')). 
\end{align}
 
\begin{lemma}\label{l:delocalization}
Let $\cF$, $\cG$ be finitely generated $\cO(\cU)$-modules. If there exists an $\cO(\cU)_\lambda$-linear isomorphism $\cF_{\lambda}\xrightarrow{\sim} 
\cG_{\lambda}$, then there exists an admissible open sub-affinoid
 $\cU'\subset \cU$ containing $\lambda$ such that $\cF\otimes_{\cO(\cU)} \cO(\cU')\simeq
\cG\otimes_{\cO(\cU)} \cO(\cU')$. 
In particular, if $\cF_\lambda$ is free over $\cO(\cU)_\lambda$, then there exists an admissible open affinoid $\cU'\subset \cU$ containing $\lambda$ such that $\cF\otimes_{\cO(\cU)}\cO(\cU')$ is $\cO(\cU')$-free.
\end{lemma}
\begin{proof}
Suppose $f_1,\ldots, f_M$ generate $\cF$
as an $\cO(\cU)$-module. Let
$\phi:\cF_{\lambda}\to
\cG_\lambda$ be the given isomorphism. The
map $\phi$ is uniquely determined by the elements
$\phi(f_1),\ldots,\phi(f_M)\in
\cG_\lambda$. Let $\cU'$ be an
admissible neighborhood of $\lambda$ in $\cU$ over which the
elements $\phi(f_i)$ are all defined. By
$\cO(\cU')$ linearity we get a homomorphism
$\phi_{\cU'}:\cF\otimes_{\cO(\cU)}\cO(\cU')\to \cG\otimes_{\cO(\cU)}\cO(\cU')$.
Since $\cO(\cU')$ is noetherian, $\ker(\phi_{\cU'})$ and $\coker(\phi_{\cU'})$
are finitely generated $\cO(\cU')$-modules, whose localizations at $\lambda$ vanish. 
It follows that, by further shrinking $\cU'$, one may ensure that generators for $\ker(\phi_{\cU'})$ and $\coker(\phi_{\cU'})$ vanish, 
proving that $\phi_{\cU'}$ is an isomorphism. The last statement follows by taking $\cG$ to be the free $\cO(\cU)$-module of rank equal to the rank of $\cF_\lambda$.
\end{proof}

Our final abstract lemma is an application of Lemma
 \ref{l:hida control} to rigid spaces. Fix $S \subset S_p$. 

\begin{lemma}\label{l:free rigid}
Let $A=\cO(\cU)_\lambda$, where $\cU\subset \cX_S$ is an admissible neighborhood 
of $\lambda$ and suppose that $\cH^\bullet(\cD_{S,\lambda})$ is supported in degree $d$. 
Then, after possibly shrinking $\cU$, $\cH^d((\cD_{S,\cU})_\lambda)$ is $\cO(\cU)_\lambda$-free.
\end{lemma}
\begin{proof} Since $(\cD_{S,\cU})_\lambda$ is $\cO(\cU)_\lambda$-flat, applying Lemma \ref{l:hida control}(i)
yields an isomorphism of $\cO(\cU)_\lambda$-modules $\cH^d((\cD_{S,\cU})_\lambda)
\otimes_{\cO(\cU)_\lambda} L\xrightarrow{\sim} \cH^d(\cD_{S,\lambda})$. It follows then from Nakayama's lemma that the module $\cH^d((\cD_{S,\cU})_\lambda)$ can be
generated over $\cO(\cU)_\lambda$ by $r=\dim_L \cH^d(\cD_{S,\lambda})$
generators $m_1,\ldots, m_r$. 

Suppose there exists a relation $f_1m_1+\cdots +f_rm_r=0$ with
$f_1,\ldots, f_r\in \cO(\cU)_\lambda$ not all $0$. Then, for example, $f_1\in \cO(\cU')\!\setminus\{0\}$ for some closed polydisc $\cU'\subset\cU$ containing $\lambda$. Since $f_1$ is analytic, there exists a $1$-dimensional disk $\cV\subset \cU'$ such that the image of $f_1$ in $\cO(\cV)$
is non-zero, yielding a dependence relation of $m_1,\ldots, m_r$ over the discrete valuation ring $\cO(\cV)_\lambda$. 
This contradicts Lemma \ref{l:hida control}(ii) since the $\cO(\cV)_\lambda$-rank of
$\cH^d(\cD_{S,\cV}\otimes \cO(\cV)_\lambda)$ would be $\leqslant r-1$. 
\end{proof}

\subsection{Etaleness at non-critical points}\label{subsection-controls}
Our main interest lies with compactly supported cohomology, which is not supported in middle
degree. To account for this, we will localize at the maximal ideal defined by $\widetilde{\pi}$
 and will obtain etaleness at non-critical points in both full and partial $p$-adic families.

Let $S\subset S_p$ and let $\widetilde{\pi}_S=(\pi,\{\nu_v\}_{v\in S})$ be a regular $S$-refinement of $\pi$ 
(see Def. \ref{d:hilbert-nearly}).

Henceforth we let $K=K(\widetilde{\pi}_S,\gu)$ (see Def. \ref{d:newline}).
By Corollary \ref{fg} for any $h_S \in \Q_{\geqslant 0}^S$ there exists an $L$-affinoid neighborhood $\cU_S$ of $(k,\sw)$ in $\cX_{S}$
such that $\rH_c^d(Y_K, \cD_{S,\cU_S})^{\leqslant h_S}$ is a finitely generated $\cO(\cU_S)$-module. 
Since $\cO(\cU_S)$ is noetherian the $\cO(\cU_S)$-algebra $\T_{S,\cU_S}^{\leqslant h_S}$, generated by the image of $\widetilde\T_S$ (see Def. \ref{d:mpi-tilde}) in $\End_{\cO(\cU_S)}(\rH_c^d(Y_K,\cD_{S,\cU_S})^{\leqslant h_S})$, is finite. 
For {\it any} sub-affinoid $\cU'_S\subset \cU_S$ containing $(k,\sw)$ consider the maximal ideal 
of $\widetilde\T_S\otimes_{E,\iota_p} \cO(\cU'_S)$ generated by $\gm_{\widetilde{\pi}_S}$ and by 
$\gm_{(k,\sw)}$ and, by an abuse of notation, let  $\gm_{\widetilde{\pi}_S}$ denote its image in $\T_{S,\cU'_S}^{\leqslant h_S}$, as well as the corresponding maximal ideal of the rigid analytic localization $(\T_{S,\cU'_S}^{\leqslant h_S})_{(k,\sw)}$ (see \eqref{eq:rigid-loc}). 
The rigid localization of $\Sp(\T_{S,\cU_S}^{\leqslant h_S})$ at the point $\widetilde{\pi}_S$
corresponding to the maximal ideal $\gm_{\widetilde{\pi}_S}\subset \T_{S,\cU_S}^{\leqslant h_S}$ is given by 
 the limit 
 \[ (\T_{S,\cU_S}^{\leqslant h_S})_{\widetilde{\pi}_S}\simeq \varinjlim_{\widetilde{\pi}_S\in\cV }\cO(\cV)\] 
 over all admissible neighborhoods $\cV$ of $\widetilde{\pi}_S$ in $\Sp(\T_{S,\cU_S}^{\leqslant h_S})$.
 The weight map $\kappa: \Sp(\T_{S,\cU_S}^{\leqslant h_S})\to \cU_S$ induces a ring homomorphism $\cO(\cU_S)_{(k,\sw)}\to (\T_{S,\cU_S}^{\leqslant h_S})_{\widetilde{\pi}_S}$. 
  More generally, given a $\T_{S,\cU_S}^{\leqslant h_S}$-module $\cF$ we let $\cF_{\widetilde{\pi}_S} = 
  \cF\otimes_{\T_{S,\cU_S}^{\leqslant h_S}}(\T_{S,\cU_S}^{\leqslant h_S})_{\widetilde{\pi}_S}$. 
 The natural map $\cF_{(k,\sw)} \to  \cF_{\widetilde{\pi}_S}$ induces an isomorphism  
   \begin{align}\label{eq:rigid-bis-loc}
  (\cF_{(k,\sw)})_{\gm_{\widetilde{\pi}_S}}  \xrightarrow{\sim}  \cF_{\widetilde{\pi}_S} . 
\end{align}
 
\begin{definition} \label{d:non-critical}
We say that $\widetilde{\pi}_S$ is non-critical if the $\gm_{\widetilde{\pi}_S}$-localization 
\[\vartheta_S: \rH_c^\bullet(Y_K,\cD_{S,(k,\sw)})_{\gm_{\widetilde{\pi}_S}}
\to \rH^{\bullet}_{c}(Y_K,\cL_{k,\sw}^\vee(L))_{\gm_{\widetilde{\pi}_S}}\]
of \eqref{specialization} is an 
isomorphism. When $S=S_p$ we let $\T_{\cU}=\T_{S_p,\cU_{S_p}}$ and will say that $\widetilde{\pi}$ is {\it non-critical}.
\end{definition}

Theorem \ref{t:classicality} applied to $h_S=(h_{\widetilde{\pi}_v})_{v\in S}$ (see Def. \ref{d:non-critical-slope}) has the following direct consequence. 

\begin{corollary}\label{p:non-criticality}
If $\widetilde{\pi}_S$ has non-critical slope, then $\widetilde{\pi}_S$ is non-critical. 
\end{corollary}

 The main result of this section is the following control and freeness theorem for compactly supported cohomology. 
Fix a character  $\epsilon$ of $\{\pm 1\}^\Sigma$. 

 \begin{thm} \label{free-etale}
Suppose that $\widetilde{\pi}_S$ is non-critical. Then, 
after possibly shrinking $\cU_S\subset \cX_S$:
\begin{enumerate}
\item $\rH^\bullet_c(Y_K, \cD_{S,\cU_S})^{\epsilon}_{\widetilde{\pi}_S}$ is a free $\cO(\cU_S)_{(k,\sw)}$-module of rank 1 and is
supported in degree $d$. 
\item The weight map $\kappa: \Sp(\T_{S,\cU_S}^{\leqslant h_S})\to \cU_S$ is etale at $\widetilde{\pi}_S$, {\it i.e.}, 
there exists an irreducible component $\cV_S$ of $\Sp(\T_{S,\cU_S}^{\leqslant h_S})$ containing $\widetilde{\pi}_S$ such that $\kappa: \cV_S\to \cU_S$ is an isomorphism of affinoids. Moreover 
$\rH^\bullet_c(Y_K, \cD_{S,\cU_S})^{\epsilon,\leqslant h_S}\otimes_{\T_{S,\cU_S}^{\leqslant h_S}}\cO(\cV_S)$ is supported in degree $d$ and is free of rank 1 over $\cO(\cU_S)$.
 \item For any cohomological weight $\lambda\in \cU_S$, $\kappa^{-1}(\lambda)\in \cV_S$ corresponds to a non-critically $S$-refined weight $\lambda$ cuspidal automorphic representation $\widetilde{\pi}_{\lambda,S}$ of $G(\A)$. 
\end{enumerate}
\end{thm}

 \begin{proof} 
 (i) By non-criticality and cuspidality $\rH^\bullet_c(Y_K,\cD_{S,(k,\sw)})^{\epsilon}_{\gm_{\widetilde{\pi}_S}} \xrightarrow{\sim}
 \rH_c^\bullet(Y_K,\cL_{k,\sw}^\vee(L))^\epsilon_{\gm_{\widetilde{\pi}_S}}$
 is supported in degree $d$ and has dimension equal to 
 $\dim (\pi_f^K)_{\gm_{\widetilde{\pi}_S}}=1$  (see Lemma \ref{l:multiplicity one}). 
Applying Lemma \ref{l:hida control}(i) and Lemma \ref{l:free rigid} to the cohomology functor $\cH^\bullet(-)=\rH^\bullet_c(Y_K, -)^{\epsilon}_{\gm_{\widetilde{\pi}_S}}$ yields that 
 \[ \rH^\bullet_c(Y_K,(\cD_{S,\cU_S})_{(k,\sw)})^{\epsilon}_{\gm_{\widetilde{\pi}_S}}\xrightarrow{\sim}
\rH^\bullet_c(Y_K, \cD_{S,\cU_S})^{\epsilon}_{\widetilde{\pi}_S}
 \]
 is also supported in degree $d$ and is $\cO(\cU_S)_{(k,\sw)}$-free of rank $1$ (see \eqref{eq:rigid-bis-loc})

(ii) It follows from (i) that the natural map $\cO(\cU_S)_{(k,\sw)}\to (\T_{S,\cU_S})_{\widetilde{\pi}_S}$ is an isomorphism
and that $\rH^\bullet_c(Y_K, \cD_{S,\cU_S})^{\epsilon}_{\widetilde{\pi}_S}$ has rank $1$ and is supported in degree $d$.  Both claims then follow straightforwardly using Lemma \ref{l:delocalization}.

(iii) After shrinking $\cU_S$, we may assume that any cohomological $\lambda\in \cU_S\!\setminus \{(k,\sw)\}$ is non-constant, {\it i.e.} 
$k_\lambda\ne 2t$, and such  that $e_v h_{\widetilde{\pi}_v}<\min\limits_{\sigma\in \Sigma_v}(k_{\lambda,\sigma}-1)$ for all  $v\in S$. 
Let $\gm\subset \widetilde\T_S$ be the maximal ideal corresponding to the map 
$\widetilde\T_S\to \mathbb{T}_{S,\cU_S}^{\leqslant h_S}\to \cO(\cV_S) \to L$ induced by $\kappa^{-1}(\lambda)$. 
Using the same abuse of notation for $\gm$ as we did for $\gm_{\widetilde{\pi}_S}$ (see the paragraph above 
 \eqref{eq:rigid-bis-loc}), (ii) yields an isomorphism 
\[ \rH_c^\bullet(Y_K, (\cD_{S,\cU_S})_\lambda)_{\gm}^{\epsilon} \xrightarrow{\sim}
\rH_c^\bullet(Y_K, \cD_{S,\cU_S})^{\epsilon,\leqslant h_S}\otimes_{\mathbb{T}_{S,\cU_S}^{\leqslant h_S}}\cO(\cV_S)_{\kappa^{-1}(\lambda)} = \rH_c^\bullet(Y_K, \cD_{S,\cU_S})^{\epsilon}_{\kappa^{-1}(\lambda)}\] 
of free $\cO(\cU_S)_{\lambda}$-modules of rank $1$ supported in degree $d$. 
By using the long exact sequences from the proof of Lemma \ref{l:hida control} for the functor 
$\cH^\bullet(-)=\rH_c^\bullet(Y_K,-)_{\gm}^{\epsilon}$ one can perform this time an {\it ascending}  induction showing that
the $L$-vector space $\rH_c^\bullet(Y_K, \cD_{S,\lambda})^{\epsilon}_{\gm}$ is $1$-dimensional and supported in degree $d$. 
The non-critical slope assumption on $\lambda$ implies, via Theorem \ref{t:classicality}, that
\[\rH_c^\bullet(Y_K,\cD_{S,\lambda})^{\epsilon}_{\gm} \xrightarrow{\sim} 
\rH_c^\bullet(Y_K, \cL_{\lambda}^\vee(L))^{\epsilon}_{\gm} \]
is concentrated is degree $d$ and is one dimensional, and therefore $\kappa^{-1}(\lambda)$ corresponds to a non-critically $S$-refined automorphic representation $\widetilde{\pi}_{\lambda,S}$ of $G(\A)$ of weight $\lambda$ (see  Corollary \ref{p:non-criticality}). 
Finally,   $\pi_\lambda$ is necessarily cuspidal as  otherwise it would necessarily be an Eisenstein series and contribute to the usual cohomology
in all degrees between $d$ and  $2d-1$, hence in 
all degrees between $1$ and  $d$ of the compactly supported one (see, for example, \cite{MSYZ}). 
\end{proof}

 \begin{remark} Our construction of partial eigenvarieties will be crucially used in \S\ref{s:ezc} for the calculations of higher order derivatives of $p$-adic $L$-functions.  In the literature we find other examples of the use of partial eigenvarieties in arithmetic applications. For example in \cite{johansson-newton} the authors use partial eigenvarieties for definite quaternion algebras to study the parity conjecture for Hilbert modular forms. Another example is \cite{chenevier:application} where partial eigenvarieties for definite unitary groups are used to attach Galois representations to conjugate self-dual automorphic representations of $\mathrm{GL}_n$ over CM fields.  
\end{remark}

\section{Automorphic symbols and $p$-adic distributions}\label{automorphic-symbols} 

In this section we use the automorphic cycles introduced in \cite{dimitrov} to construct, for any ideal $\gf\subset \cO_F$
supported in $S\subset S_p$, natural evaluations on the cohomology of the Hilbert modular variety with general coefficients, and show that they satisfy certain distribution relations as $\gf$ varies. When 
applied to a finite slope $U_p$-eigenclass in overconvergent cohomology the construction yields a distribution of controlled growth on $\Gal_{p\infty}$ in a canonical way, {\it i.e.}, independent of choices of uniformizers or representatives in idele class groups. This will allow us in \S \ref{s:p-adic-L} to attach  $p$-adic $L$-functions to nearly finite slope families of Hilbert cusp forms, generalizing \cite{barrera}. 

When the support of the ideal $\gf$ is small, a case not previously considered, the relations established when $S$ varies 
will be used in \S\ref{ss:improved} to construct partially ``improved'' $p$-adic $L$-functions. 
 
 For  $S\subsetneq S_p$ our construction 	allows us to attach to 	a family  $\cV_S\xrightarrow{\sim} \cU_S$ containing a 
 non-critical $S$-refinement $\widetilde{\pi}_S$ (see Thm. \ref{free-etale}), 
 a rather mysterious $\cO(\cU_S)$-valued distribution on $\Gal_{S\infty}$.

\subsection{Fundamental classes and automorphic cycles}\label{ss:fundamental}
For an ideal $\gf$ the $d$-dimensional open manifold
$X^+_{\gf}=E(\gf)\!\setminus F_\infty^{\times+}$ is
orientable with orientation induced by an orientation on
$F_\infty^{\times+}$. The Borel-Moore homology
$\rH_i^{\BM}(X^+_{\gf})$ can be computed as the
relative homology $\rH_i({\overline{X}}^+_{\gf},
{\overline{X}}^+_{\gf}\!\setminus
X^+_{\gf})$ where
${\overline{X}}^+_{\gf}$ is the two-point
compactification of $X^+_{\gf}$ \cite[Def.1.5]{dimitrov}. When $i=d$,
$\rH_d^{\BM}(X^+_{\gf})\simeq \Z$ and one can
choose a fundamental class $\theta_{\gf}=[X^+_{\gf}]\in
\rH_d^{\BM}(X^+_{\gf})$ in a compatible way as $\gf$ varies, such that for all $\gf\mid
\gf'$ the finite map $\pi:X_{\gf'}^+\to X^+_{\gf}$ induces  a commutative diagram:
\[\xymatrix{
	\rH_d^{\BM}(X^+_{\gf})\ar[rr]^-{\sim}\ar@<-.5ex>[d]_{\pi^*}&&\Z\ar@<-.5ex>[d]_{\id}\\
	\rH_d^{\BM}(X_{\gf'}^+)\ar@<-.5ex>[u]_{\pi_*}\ar[rr]^-{\sim}&&\Z\ar@<-.5ex>[u]_{{[E(\gf):E(\gf')]\cdot }}
}\]
In top degree the map $\rH_d^{\BM}(X^+_{\gf})\to \rH_c^d(X^+_{\gf})^\vee$ sending $\theta$ to $c\mapsto \displaystyle c\cap \theta=\int_{\theta}c$ is an isomorphism. When $d\geqslant 2$ one has $\rH_d^{\BM}(X^+_{\gf})=\rH_d({\overline{X}}^+_{\gf}, {\overline{X}}^+_{\gf}\!\setminus X^+_{\gf})\simeq \rH_d({\overline{X}}^+_{\gf})$. We let 
 \[X_{\gf}=\A_F^{ \times}/F^{\times} U(\gf).\]
 For any $\eta\in \A_F^{\times}$ representing a class $[\eta]\in \Cl(\gf)\simeq \pi_0(X_{\gf})$ we denote by
$X_{\gf}[\eta]$ the connected component of $\eta$ in $X_{\gf}.$
The map $X^+_{\gf} \xrightarrow{\cdot \eta}X_{\gf}[\eta]$
yields an isomorphism $\rH_d^{\BM}(X^+_{\gf})\xrightarrow{\eta_*}
\rH_d^{\BM}(X_{\gf}[\eta])$ independent of the choice of
the representative $\eta$. Hence $\theta_{\gf}$ yields a fundamental class
\begin{align}\label{eq:theta}
\theta_{\gf,[\eta]}=\eta_*(\theta_{\gf})\in
\rH_d^{\BM}(X_{\gf}[\eta]).
\end{align}
In this entire section $S\subset S_p$, $\gf\subset\cO_{F}$ is an ideal supported in $S$ and 
$K\subset G(\A_f)$ is an open compact subgroup containing 
$\begin{psmallmatrix} U(\gf) & \widehat\cO_F \\ 0 & 1\end{psmallmatrix} $
such that the image $K_p$ of $K$ in $G(\Q_p)$ is contained in $I_S$. 

We define the {\it automorphic cycle of level} $\gf$ as 
\[ C_{\varpi_{\gf},K}: X_\gf \longrightarrow Y_K \text{ , } 
y\mapsto 
\left[\begin{psmallmatrix} y & 0 \\ 0 & 1\end{psmallmatrix}
\begin{psmallmatrix} \varpi_{\gf} & 1_{\gf} \\ 0 & 1\end{psmallmatrix}\right]. \]
which is a well-defined, continuous, finite map (see \cite[Cor. 1.23]{dimitrov-pune}).
For any $\eta\in \A_F^{\times}$ we define $C_{\varpi_{\gf},K}^{\eta}:{\overline{X}}^+_{\gf}\to Y_K$ as the composition of the maps $C_{\varpi_{\gf},K}$ and $X^+_{\gf} \xrightarrow{\cdot \eta}X_{\gf}$.

When $K$ is clear from context it will be dropped from the notation of the automorphic cycle.

\subsection{Evaluations}\label{evaluations}
We will axiomatize and generalize some constructions from \cite{dimitrov} and \cite{barrera} on the evaluations of automorphic cycles on the cohomology of the Hilbert modular varieties. 

Let $M$ be a left $\varLambda_S$-module (see \eqref{eq:Lambda-S}) and let $\cM$ be the sheaf on $Y_K$ attached 
as in \S\ref{local-sys} to $M$ on which $K$ acts via $K_p\subset I_S \subset \varLambda_S$. The pullback by $C_{\varpi_{\gf}}$ induces a homomorphism: 
\begin{align}\label{ev1} 
C_{\varpi_{\gf}}^*: \rH_{c}^{d}(Y_K, \cM) \to \rH_{c}^{d}(X_\gf,C_{\varpi_{\gf}}^*\cM).
\end{align}

Since  $ \begin{psmallmatrix} \varpi_{\gf} & 1_{\gf} \\ 0 & 1\end{psmallmatrix}\begin{psmallmatrix} u & (u- 1)1_{\gf}\varpi_{\gf}^{-1} \\ 0 & 1 \end{psmallmatrix}=
\begin{psmallmatrix} u & 0 \\ 0 & 1 \end{psmallmatrix}
\begin{psmallmatrix} \varpi_{\gf} & 1_{\gf} \\ 0 & 1\end{psmallmatrix}$ one sees that 
$C_{\varpi_{\gf}}^*\cM$ is the sheaf of locally constant sections of the local system: 
\begin{align}\label{loc-sys-1} 
C_{\varpi_{\gf}}^*M= F^{\times}\!\setminus (\A_{F}^{\times}\times M) / U(\gf) \to X_\gf,
\end{align}
where $\xi(y, m) u= (\xi y u, \begin{psmallmatrix} u & (u- 1)1_{\gf}\varpi_{\gf}^{-1} \\ 0 & 1 \end{psmallmatrix} ^{-1}\cdot m)$ for $\xi \in F^{\times}$, $y \in \A_{F}^{\times}$, $m \in M$ and $u \in U(\gf)$.

Let $\cM_\gf$ be the sheaf of locally constant sections of the local system: 
\begin{align}\label{loc-sys-2} 
F^{\times}\!\setminus (\A_{F}^{\times}\times M) / U(\gf) \to X_\gf,
\end{align}
where $\xi(y, m) u= (\xi y u, \begin{psmallmatrix} u & 0 \\ 0 & 1 \end{psmallmatrix} ^{-1}\cdot m)$ for $\xi \in F^{\times}$, $y \in \A_{F}^{\times}$, $m \in M$ and $u \in U(\gf)$.

Since the image of 
$\begin{psmallmatrix}\varpi_{\gf} & 1_{\gf} \\ 0 & 1 \end{psmallmatrix}$ in $G(\Q_p)$ belongs to $\varLambda_S$, one can consider the map:
\[\A_{F}^{\times}\times M \to \A_{F}^{\times}\times M, \,
(y, m)\mapsto (y,\begin{psmallmatrix}\varpi_{\gf} & 1_{\gf} \\ 0 & 1 \end{psmallmatrix}\cdot m),\]
which sends the local system \eqref{loc-sys-1} to the one from \eqref{loc-sys-2}. The resulting homomorphism of sheaves 
$\tw_{\varpi_{\gf}}: C_{\varpi_{\gf}}^*\cM \to \cM_\gf$ over $X_\gf$ yields a homomorphism of cohomology groups
\begin{align}\label{ev2} 
\tw_{\varpi_{\gf}}: \rH_{c}^{d}(X_\gf,C_{\varpi_{\gf}}^*\cM) \to \rH_{c}^{d}(X_\gf,\cM_\gf).
\end{align} 

Let $M_{E(\gf)}$ denote the $E(\gf)$-coinvariants of $M$. Consider the sheaf $\cM_{E(\gf)}$ 
attached to the local system
$F^{\times}\!\setminus \A^\times_F\times M_{E(\gf)}/U(\gf)$ with $\xi(y,m)u=(\xi yu, \begin{psmallmatrix}u&\\ &1 \end{psmallmatrix}^{-1}\cdot m)$. There is a natural map
\[\coinv_{\gf}:\rH_c^d(X_{\gf},\cM_{\gf})\to
\rH^d_c(X_{\gf}, \cM_{E(\gf)}).\]
Trivializing the sheaf $ \cM_{E(\gf)}$ requires to choose a representative $\eta\in \A_F^{\times}$ 
of $[\eta]\in \Cl(\gf)$. Then 
\[ \triv_\eta: X_{\gf}[\eta] \times M_{E(\gf)}\to \left(\cM_{E(\gf)}\right)_{|X_{\gf}[\eta]}, \quad 
(a\eta u_\infty u,m)\mapsto (a\eta u_\infty
u, \begin{psmallmatrix}u^{-1}&\\ &1 \end{psmallmatrix}\cdot m)\]
is an isomorphism of local systems yielding the desired trivialization map
\[\triv_\eta^*:\rH^d_c(X_{\gf}[\eta], \cM_{E(\gf)})\to \rH^d_c(X_{\gf}[\eta], \Z)\otimes M_{E(\gf)}\]
Capping with $\theta_{\gf, [\eta]}$ from \eqref{eq:theta} yields an isomorphism $\rH^d_c(X_{\gf}[\eta], \Z)\otimes M_{E(\gf)}\simeq M_{E(\gf)}$. We define 
\begin{align}\label{ev}
\ev_{\varpi_{\gf}}^\eta(M)=(-\cap \theta_{\gf,[\eta]})\circ \triv_\eta^*\circ \coinv_{\gf}\circ\tw_{\varpi_{\gf}}
\circ \, C_{\varpi_{\gf},K}^*:\rH^{d}_{c}(Y_K, \cM) \to M_{E(\gf)}.
\end{align}

Where $M$ is clear from context we will drop it from the notation of the evaluation map.

\begin{lemma} \label{l:ev-functorial}The evaluation maps $\ev_{\varpi_{\gf}}^{\eta}$ are covariant in $M$, in the sense that if $\vartheta: M\to M'$ is a morphism of left $\varLambda_S$-modules then one has $\vartheta\circ \ev_{\varpi_{\gf}}^{\eta}(M)= \ev_{\varpi_{\gf}}^{\eta}(M')\circ \vartheta$. 
\end{lemma}

It follows from \eqref{eq:U0} that there is commutative diagram
	\begin{align}\label{eq:ev-functorial}
	\xymatrix{ \rH^{d}_{c}(Y_K, \cD_{S,(k,\sw)}) \ar[d]^{\vartheta_S} \ar[rrrr]^-{\varpi_{\gf}^{((2-\sw)t-k)/2}\cdot 
\ev_{\varpi_{\gf}}^{\eta}(D_{S,(k,\sw)})} & &&& (D_{S,(k,\sw)})_{E(\gf)} \ar[d]^{\vartheta_S} \\
		\rH^{d}_{c}(Y_K, \cL_{k,\sw}^\vee(L)) \ar[rrrr]^-{\ev_{\varpi_{\gf}}^{\eta}(L_{k,\sw}^\vee(L))} & &&& L_{k,\sw}^\vee(L)_{E(\gf)}.} 
	\end{align}

\begin{lemma}\label{l:ev-eta}
	For $\eta\in \A_F^{\times}$ and $\eta'\in F^{\times} \eta u F_\infty^{\times+}$ with $u\in U(\gf)$ we have 
	$\ev_{\varpi_{\gf}}^{\eta'}= \begin{psmallmatrix}u_p^{-1}& \\ & 1 \end{psmallmatrix}\cdot \ev_{\varpi_{\gf}}^{\eta}$.
\end{lemma}
\begin{proof}
	Since in \eqref{ev} only $\triv_\eta^*$ 	depends on $\eta$ it suffices to check that
	$ \triv_{\eta'}^*	= (\id\otimes\begin{psmallmatrix}u_p^{-1}& \\ & 1 \end{psmallmatrix})\cdot
	\triv_\eta^*$. This follows from the fact that $\triv_{\eta'}(y,m)=\triv_\eta\left(y, \begin{psmallmatrix}u_p^{-1}& \\ & 1 \end{psmallmatrix}\cdot m\right)$.
	\end{proof}
\begin{lemma}\label{l:ev-uniformizer}
If $v\mid \gf$ then for all $\delta\in \cO_v^\times$ we have 
$\ev_{\varpi_{\gf}\delta}^{\eta}=\ev_{\varpi_{\gf}}^{\eta}\circ U_{\delta}$. 
\end{lemma}

\subsection{Relations}\label{relations}
We will prove a fundamental relation between the evaluations
defined in \S \ref{evaluations}, which will later be used to prove an interpolation property and growth condition 
of certain $p$-adic distributions on Galois groups, as well as a relation between the corresponding $p$-adic $L$-functions and their 
improved counterparts.

\begin{prop} \label{prop-relations} For $v \in S$ we let $\pr_{\gf v,\gf}:\Cl(\gf v)\to \Cl(\gf)$ denote the natural projection.
	 Choose a	representative $\eta\in \A_F^{\times}$ for $[\eta] \in	\Cl(\gf)$ and for each $[\delta]\in\pr_{\gf v,\gf}^{-1}([\eta])$ let $\delta\in	\A_F^{\times}$ and $u_{\delta}\in U(\gf)$ be such
	that $\delta \in F^{\times} \eta u_\delta F_\infty^{\times+}$. Then
	\[\sum_{[\delta]\in\pr_{\gf v,\gf}^{-1}([\eta])}\begin{psmallmatrix} u_{\delta}&\\ &1\end{psmallmatrix}\cdot\ev_{\varpi_{\gf v}}^{\delta}=
	\begin{cases} \ev_{\varpi_{\gf}}^{\eta}\circ U_{\varpi_v} & \text{, if } v\mid\gf, \\
	\ev_{\varpi_{\gf}}^{\eta}\circ U_{\varpi_v}- 
	\begin{psmallmatrix} \varpi_{v} & 0 \\ 0 & 1 \end{psmallmatrix} \cdot \ev_{\varpi_{\gf}}^{\eta\varpi_v}& \text{, if } v\nmid\gf. \end{cases} \]
\end{prop}
\begin{proof} We first recall the definition of $U_{\varpi_v}$. Let $\gamma= \begin{psmallmatrix} \varpi_{v} & 0 \\ 0 & 1 \end{psmallmatrix}$ and consider the natural projections
	$\pr_1:Y_{K_{0}(v)} \to Y_K$ and $\pr_2:Y_{K^{0}(v)} \to Y_K$ where $K^{0}(v)= K \cap \gamma K \gamma^{-1}$ and $K_{0}(v)= K \cap \gamma^{-1} K \gamma$. For clarity we will denote by $\cM_K$ the local system on $Y_K$, in which case $\pr_1^* \cM_K = \cM_{K_0(v)}$ and $\pr_2^* \cM_K = \cM_{K^0(v)}$. Define 
	\begin{align}\label{Uv-YK}
	U_{\varpi_v}= \Tr(\pr_2) \circ[\gamma]\circ \pr_1^*:\rH_c^d(Y_K, \cM_K)\to \rH_c^d(Y_K, \cM_K),
	\end{align}
where $[\gamma]:\rH_c^d(Y_{K_0(v)},\cM_{K_0(v)})\to \rH_c^d(Y_{K^0(v)},\cM_{K^0(v)})$ is induced by the morphism of local systems given by $(g,m)\mapsto (g \gamma^{-1},\gamma\cdot m)$ (note that image of $\gamma$ in $G(\Q_p)$ belongs to $\varLambda_S$), $\pr_1^*:\rH_c^d(Y_K, \cM_K)\to \rH_c^d(Y_{K_0(v)},\cM_{K_0(v)})$ is the pullback, and $\Tr(\pr_2):\rH_c^d(Y_{K^0(v)}, \cM_{K^0(v)})= \rH_c^d(Y_K, p_{2*}\pr_2^*\cM_K)\to \rH_c^d(Y_K, \cM_K)$ is the trace attached to the finite map $\pr_2$.

	We will now define analogous maps on $X_{\gf}$. The map 
	\[\widetilde{C}_{\gf,K^{0}(v)}: X_{\gf v}\longrightarrow Y_{K^{0}(v)} \text{ , } [y]\mapsto
	\left[\begin{psmallmatrix} y & 0\\ 0 & 1\end{psmallmatrix}
	\begin{psmallmatrix} \varpi_{\gf} & 1_{\gf v} \\ 0 & 1\end{psmallmatrix}
	\right],\] 
	is well-defined since for all $\xi \in F^{\times}$ and $u \in U(\gf v)$, we have $\begin{psmallmatrix} u & (u-1) 1_{\gf v}\varpi_{\gf}^{-1} \\ 0 & 1\end{psmallmatrix} \in K^{0}(v)$ and 
	\begin{align}\label{eq:tilde}
	\begin{psmallmatrix} \xi yu & 0 \\ 0 & 1\end{psmallmatrix}
	\begin{psmallmatrix} \varpi_{\gf} & 1_{\gf v} \\ 0 & 1\end{psmallmatrix}
	= \begin{psmallmatrix} \xi & 0 \\ 0 & 1\end{psmallmatrix}
	\begin{psmallmatrix} y & 0 \\ 0 & 1\end{psmallmatrix}
	\begin{psmallmatrix} \varpi_{\gf} & 1_{\gf v} \\ 0 & 1\end{psmallmatrix}
	\begin{psmallmatrix} u & (u-1) 1_{\gf v}\varpi_{\gf}^{-1}\\ 0 & 1\end{psmallmatrix}.
	\end{align}

	{\it Case 1: } $v \mid \gf$. In this case $1_{\gf v}=1_{\gf}$. 
	Denoting by $\cdot \gamma$ the right translation, one checks that the following diagram commutes: 
	\begin{align}\label{hilbert vs X}
	\xymatrix{
		Y_{K} && Y_{K^{0}(v)} \ar^{\cdot \gamma}[rr]\ar_{\pr_2}[ll]&& Y_{K_{0}(v)} \ar^{\pr_1}[rr]&& Y_{K}\\
		X_{\gf} \ar^{C_{\varpi_{\gf}}}[u] && X_{\gf v} \ar_{\pr_{\gf v,\gf}}[ll] \ar^{\widetilde{C}_{\gf,K^{0}(v)}}[u]\ar@{=}[rr]&& X_{\gf v} \ar@{=}[rr] \ar^{ C_{\varpi_{\gf v},K_{0}(v)} }[u]&& X_{\gf v} \ar^{C_{\varpi_{\gf v},K}}[u]
	}
	\end{align}
	We consider the local system
	$\widetilde{C}_{\gf,K^{0}(v)}^* \cM = F^{\times}\!\setminus \A_F^{\times}\times M/U(\gf v)$ on $X_{\gf v}$, where $u\in U(\gf v)$ acts on $M$ by $\begin{psmallmatrix} u & (u-1) 1_{\gf}\varpi_{\gf}^{-1} \\ 0 & 1\end{psmallmatrix}$. Since $\begin{psmallmatrix} \varpi_{v} & 0 \\ 0 & 1\end{psmallmatrix} \begin{psmallmatrix} u & (u-1)1_{\gf v}\varpi_{\gf v}^{-1} \\ 0 & 1\end{psmallmatrix}= \begin{psmallmatrix} u & (u-1)1_{\gf}\varpi_{\gf}^{-1} \\ 0 & 1\end{psmallmatrix} \begin{psmallmatrix} \varpi_{v} & 0 \\ 0 & 1\end{psmallmatrix}$
	one has a morphism of local systems 
	\[ C_{\varpi_{\gf v},K}^{\ast}\cM=C_{\varpi_{\gf v},K_{0}(v)}^{\ast}\cM\to \widetilde{C}_{\gf,K^{0}(v)}^{\ast},\, [(y, m)] \mapsto [(y, \begin{psmallmatrix} \varpi_{v} & 0 \\ 0 & 1\end{psmallmatrix}\cdot m )],\]
	inducing a homomorphism on the cohomology: 
	\[[\varpi_v]: \rH^{d}_{c}(X_{\gf v}, C_{\varpi_{\gf v},K_{0}(v)}^{\ast}\cM) \to \rH^{d}_{c}(X_{\gf v}, \widetilde{C}_{\gf,K^{0}(v)}^{\ast}\cM).\]
	Pulling back the $U_{\varpi_v}$ defined in \eqref{Uv-YK} by the vertical maps in \eqref{hilbert vs X}, 
	and noticing that the etale maps $\pr_2$ and $\pr_{\gf v,\gf}$ have the same degree, 
	yields a homomorphism:
	\begin{align}\label{Uv-X}
	U_{\varpi_v}= \Tr(\pr_{\gf v,\gf})\circ [\varpi_v]: \rH^{d}_{c}(X_{\gf v}, C_{\varpi_{\gf v},K}^{\ast}\cM) = \rH^{d}_{c}(X_{\gf v}, C_{\varpi_{\gf v},K_{0}(v)}^{\ast}\cM) \to \rH^{d}_{c}(X_{\gf}, C_{\varpi_{\gf}}^{\ast}\cM).
	\end{align}

	Next, we pull back the $U_{\varpi_v}$ action by the twisting
	operators. By \eqref{eq:tilde} and the fact that
	$\begin{psmallmatrix} \varpi_{\gf}&1_{\gf}\\&1\end{psmallmatrix}\begin{psmallmatrix} 1&-1_{\gf}\varpi_{\gf}^{-1}\\ & 1\end{psmallmatrix}=\begin{psmallmatrix} \varpi_{\gf} & \\ & 1\end{psmallmatrix}$ belongs to the torus, we have 
	morphism of local systems 
	\[\widetilde{\tw}_{\gf}:\widetilde{C}_{K^0(v),\gf}^* \cM \to \cM_{\gf v}, \, (y,m)\mapsto \left(y, \begin{psmallmatrix}
	\varpi_{\gf} & 1_{\gf}\\& 1\end{psmallmatrix}\cdot m\right).\] 
	
	Moreover, as $v\mid \gf$ we have $\begin{psmallmatrix}
	\varpi_{\gf} & 1_{\gf} \\ 0 &
	1\end{psmallmatrix} \begin{psmallmatrix} \varpi_{v} & 0 \\ 0 &
	1\end{psmallmatrix}= \begin{psmallmatrix} \varpi_{\gf v} &
	1_{\gf v} \\ 0 & 1\end{psmallmatrix}$ hence the following diagram:
	\begin{align}\label{comp-hecke-1}
	\xymatrix{
		\rH_{c}^{d}(X_{\gf v}, C_{\varpi_{\gf v},K}^{\ast}\cM) \ar^{[\varpi_v]}[rr]\ar^{\mathrm{tw}_{\gf v}}[d]&& \rH^{d}_{c}(X_{\gf v}, \widetilde{C}_{\gf,K^{0}(v)}^{\ast}\cM)\ar^{\Tr(\pr_{\gf v,\gf})}[rr]\ar^{\widetilde{\mathrm{tw}}_{\gf}}[d] & &\rH_{c}^{d}(X_{\gf}, C_{\varpi_{\gf}}^{\ast}\cM)\ar^{\mathrm{tw}_{\gf}}[d]\\
		\rH_{c}^{d}(X_{\gf v}, \cM_{\gf v})\ar@{=}[rr] && \rH_{c}^{d}(X_{\gf v}, \cM_{\gf v}) \ar^{\Tr(\pr_{\gf v,\gf})}[rr]&& \rH_{c}^{d}(X_{\gf}, \cM_{\gf})}
	\end{align}
	is commutative. 	Taking coinvariants yields:
	\begin{align}\label{comp-hecke-2}
	\xymatrix{
		\bigoplus_{[\delta]\in\pr_{\gf v,\gf}^{-1}([\eta])} \rH_c^d(X_{\gf v}[\delta],
		\cM_{E(\gf v)})\ar[rr]\ar[d]^{\oplus \triv_\delta^*}&& \rH_c^d(X_{\gf}[\eta], \cM_{E(\gf)})\ar[d]^{ \triv_\eta^*}\\
		\bigoplus_{[\delta]} \rH_c^d(X_{\gf v}[\delta])\otimes M_{E(\gf v)}\ar[rr]^{\oplus \begin{psmallmatrix} u_\delta&\\ &1\end{psmallmatrix}\cdot}\ar[d]^{\oplus -\cap \theta_{\gf v, [\delta]}}&& \rH_c^d(X_{\gf}[\eta])\otimes
		M_{E(\gf)}\ar[d]^{ -\cap \theta_{\gf, [\eta]}}\\
		\oplus_{[\delta]\in\pr_{\gf v,\gf}^{-1}([\eta])} M_{E(\gf v)}\ar[rr]_{ (m_{[\delta]})_{[\delta]} \mapsto \sum_{[\delta]} m_{[\delta]}}&& M_{E(\gf)}
	}
	\end{align}
	where the map $M_{E(\gf v)}\to M_{E(\gf)}$ is the canonical projection. 
	The commutativity of the upper square follows from the proof of Lemma \ref{l:ev-eta}, while the commutativity of the bottom square follows from the compatible choice of fundamental classes $\theta_{\gf,[\eta]}$ and $\theta_{\gf v,[\delta]}$ in \S \ref{ss:fundamental}. The proposition then follows from \eqref{ev}, \eqref{comp-hecke-1}, and \eqref{comp-hecke-2}.

	{\it Case 2: } $v \nmid \gf$. The extra term comes from the fact $\pr_{\gf v,\gf}$ has degree one less than the degree of $\pr_2$. Instead of \eqref{hilbert vs X} we consider the following commutative diagram:
	\begin{align}\label{hilbert vs X-bis}
	\xymatrix{
		Y_{K} && Y_{K^{0}(v)} \ar^{\cdot \gamma}[rr]\ar_{\pr_2}[ll]&& Y_{K_{0}(v)} \ar^{\pr_1}[rr]&& Y_{K}\\
		X_{\gf} \ar^{C_{\varpi_{\gf}}}[u] && X_{\gf}\sqcup X_{\gf v} \ar_{\mathrm{id}\sqcup \pr_{\gf v,\gf}}[ll] \ar^{C_{\varpi_{\gf},K^{0}(v)}}_{\widetilde{C}_{\gf,K^{0}(v)}}[u]\ar^{\cdot \varpi_v\sqcup \id}[rr]&& X_{\gf} \sqcup X_{\gf v} \ar@{=}[rr] \ar^{C_{\varpi_{\gf},K_{0}(v)}}_{ C_{\varpi_{\gf v},K_{0}(v)} }[u]&& X_{\gf} \sqcup X_{\gf v} \ar^{C_{\varpi_{\gf}}}_{ C_{\varpi_{\gf v},K}}[u]
	}
	\end{align}
	As in \eqref{Uv-X}, pulling back the $U_{\varpi_v}$ defined in \eqref{Uv-YK} by the vertical maps in 
	\eqref{hilbert vs X-bis} yields: 
	\[U_{\varpi_v}=U_{v, 1}+U_{v, 2}: \rH_{c}^{d}(X_{\gf}, C_{\varpi_{\gf}}^{\ast}\cM)\oplus \rH_{c}^{d}(X_{\gf v}, C_{\varpi_{\gf v},K}^{\ast}\cM) \to \rH_{c}^{d}(X_{\gf}, C_{\varpi_{\gf}}^{\ast}\cM),\]
	where $U_{v, 2}$ is given by the same formulas as $U_{\varpi_v}$ in Case 1, whereas $U_{v, 1}$ comes from the map $(y,m)\mapsto(\varpi_v^{-1} y,\begin{psmallmatrix} \varpi_v&\\ &1\end{psmallmatrix}\cdot m)$. 
	Applying $\coinv_{\gf}\circ\tw_{\varpi_{\gf}}$ to both sides of the map $U_{v, 1}$, one completes the proof by checking the commutativity of the following diagram:
	\[\xymatrix{
		\rH_{c}^{d}(X_{\gf}[\eta\varpi_v], \cM_{E(\gf)}) \ar[r]^{U_{v,1}}
		\ar_{(-\cap \theta_{\gf,[\eta\varpi_v]})\circ \triv_{\eta\varpi_v}^*}[d] &
		\rH_{c}^{d}(X_{\gf}[\eta], \cM_{E(\gf)})\ar^{(-\cap \theta_{\gf,[\eta]})\circ \triv_{\eta}^*}[d] \\
		M_{E(\gf)}\ar[r]^{\begin{psmallmatrix} \varpi_v&\\ &1\end{psmallmatrix}}& M_{E(\gf)}. } \qedhere\]
\end{proof}

\begin{remark} This proposition completes and generalizes \cite[Lem.5.1]{barrera}. The second part of this proposition generalizes \cite[Prop.5.8(i)]{greenberg-stevens} used to obtain a relation between the standard and improved $p$-adic $L$-function in the context of modular curves. Such relations will be vastly generalized in Proposition \ref{p:improved-relation}. 
\end{remark}

\subsection{Distributions on Galois groups}\label{ss:distributions}
The evaluation map $\ev_{\varpi_{\gf}}^{\eta}(M)$ constructed in \S\ref{evaluations}
 depends on a representative $\eta\in \A_F^{\times}$ of the class $[\eta]\in \Cl(\gf)$ and on the choice of uniformizers. In this section we will focus on the case where $M=D_{\cU}$
 (see \S\ref{sect:oc-modules}), with $\cU$ 
 an $L$-affinoid of the weight space $\cX$ (see \S\ref{weights}), and produce distributions on Galois groups which are independent of the above choices. These in turn will be used in \S \ref{s:p-adic-L} to construct $p$-adic $L$-functions. 

By class field theory for any integral ideal $\gf$ supported in $S_p$ there is an exact sequence:
 \begin{align}\label{eq:cft}
 1\to U(\gf)_p/\overline{E(\gf)}\xrightarrow{\iota_{\gf}}\Gal_{p\infty}\to \Cl(\gf)\to 1, 
 \end{align}
 where $U(\gf)_p = (\cO_F\otimes \Z_p)^\times\cap (1+\gf(\cO_F\otimes \Z_p))$ and $\overline{E(\gf)}$ is the $p$-adic closure of $E(\gf)$ in $U(\gf)_p$.
We have $(D_{\cU})_{E(\gf)}\subset \Hom_{\cO(\cU)}(A_{\cU}^{E(\gf)}, \cO(\cU))$, where $A_{\cU}^{E(\gf)}$ consists of $f\in A_{\cU}$
such that $f_{\big|\begin{psmallmatrix} e&\\
	&1\end{psmallmatrix}}=f$ for all $e\in E(\gf)$ (see Def. \ref{d:action-functions}). As in
\cite{barrera} we define an ``extension by zero'' morphism
\begin{align}\label{eq:times}
A(U(\gf)_p/\overline{E(\gf)}, \cO(\cU))\to A_{\cU}^{E(\gf)}, \quad f\mapsto f^\times(z)= 
\begin{cases}
\left\langle \begin{psmallmatrix} z & 0 \\ 0 & 1 \end{psmallmatrix} \right\rangle_{\cU} f(z) & \text{, if } z \in U(\gf)_p, \\ 
0 & \text {, if } z \notin U(\gf)_p. 
\end{cases}
\end{align}
Dualizing, we obtain a map $(\cD_{\cU})_{E(\gf)}\to D(U(\gf)_p/\overline{E(\gf)}, \cO(\cU))$ which we denote by $\mu\mapsto \mu^\times$. 

Let $\Gal_{p\infty}[\eta]$ denote the pre-image of $[\eta]\in \Cl(\gf)$ in $\Gal_{p\infty}$. Multiplication by the image of $\eta\in \A_F^{\times}$ in $\Gal_{p\infty}$
under the Artin recipocity map yields a bijection:
\begin{align}\label{eq:iota}
\iota_\eta: U(\gf)_p/ \overline{E(\gf)}\xrightarrow{\sim}\Gal_{p\infty}[\eta],\,\,\,
u_p\mapsto \eta\iota_{\gf}(u_p).
\end{align}
Dualizing, we obtain a map $\iota_\eta^*:D(U(\gf)_p/\overline{E(\gf)},\cO(\cU))\xrightarrow{\sim}D(\Gal_{p\infty}[\eta],\cO(\cU))$. Explicitly, for all $\mu\in D(U(\gf)_p/\overline{E(\gf)},
\cO(\cU))$ and $f\in A(\Gal_{p\infty}[\eta],\cO(\cU))$ we have
$\langle \iota_\eta^*{\mu},f\rangle = \langle \mu, f\circ \iota_{\eta}\rangle$.

\begin{lemma}\label{p:ev-gal} The following map does not depend on the representative $\eta$ of $[\eta]\in \Cl(\gf)$
\begin{align}\label{eq:ev-gal}
\ev_{\varpi_{\gf}}^{[\eta]}=\iota_\eta^*\circ\ev_{\varpi_{\gf}}^{\eta,\times}:\rH^d_c(Y_K, \cD_{\cU})\to D(\Gal_{p\infty}[\eta],\cO(\cU)). 
\end{align}
\end{lemma}

We omit the proof, which is a consequence of Lemma \ref{l:ev-eta}. 
The following result shows that, when $v$ divides $\gf$ for all $v\in S_p$, then the
passage to $\ev_{\varpi_{\gf}}^{\eta,\times}$ does not make one loose information. 	
\begin{lemma}\label{l:distribution support}
Given $\Phi\in \rH^d_c(Y_K, \cD_{\cU})$, one has $\ev_{\varpi_{\gf}}^\eta(\Phi)\in\begin{psmallmatrix}
	\varpi_{\gf}&1_{\gf}\\&1\end{psmallmatrix}\cdot	D_{\cU}$. 
	In particular, for any $f\in A_{\cU}$ one has  $\langle\ev_{\varpi_{\gf}}^\eta(\Phi),f\rangle=
	\langle\ev_{\varpi_{\gf}}^\eta(\Phi),f_{|1+\gf (\cO_F\otimes \Z_p)}\rangle$. 
\end{lemma}

\begin{definition}\label{d:eta-bracket}
We define 
	\[\ev_{\varpi_{\gf}}=\!\!\!\bigoplus_{[\eta]\in\Cl(\gf)}
	\ev_{\varpi_{\gf}}^{[\eta]}:\rH^d_c(Y_K, \cD_{\cU})\to
	D(\Gal_{p\infty},\cO(\cU)), \text {  } \displaystyle \langle \ev_{\varpi_{\gf}},f\rangle= \!\!\!\sum_{[\eta]\in \Cl(\gf)}\langle \ev_{\varpi_{\gf}}^{[\eta]},f_{|\Gal_{p\infty}[\eta]}\rangle.\]
\end{definition}

\begin{proposition}\label{p:ev-relation-canonical}
	Let $[\eta]\in \Cl(\gf)$. Then for any $v \in S_p$ we have
	\[\ev_{\varpi_{\gf}}^{[\eta]}\circ U_{\varpi_v}=\sum_{[\delta]\in\pr_{\gf v,\gf}^{-1}([\eta])} 
	\ev_{\varpi_{\gf v}}^{[\delta]}\text{ 	and } \displaystyle \ev_{\varpi_{\gf}}\circ U_{\varpi_v} = \ev_{\varpi_{\gf v}}.\]
\end{proposition}
 \begin{proof}
	Let $\Phi\in \rH^d_c(Y_K, \cD_{\cU})$ and $f\in A(\Gal_{p\infty},\cO(\cU))$. Using Proposition \ref{prop-relations} it suffices to show
	that $\displaystyle \left\langle \begin{psmallmatrix} \varpi_v&\\&1\end{psmallmatrix}\cdot\ev_{\varpi_{\gf}}^{\eta\varpi_{v}}(\Phi),(f\circ \iota_\eta)^\times\right\rangle = 0$ when $v\nmid \gf$ and 
	$\displaystyle \left\langle \begin{psmallmatrix} u_\delta&\\
	&1\end{psmallmatrix}\ev_{\varpi_{\gf	 v}}^{\delta},(f\circ \iota_\eta)^\times\right\rangle =
	\left\langle \ev_{\varpi_{\gf v}}^\delta, (f\circ \iota_\delta)^\times\right\rangle$ for all
	$v$, where $\delta\in F^{\times} \eta u_{\delta}
	F_\infty^{\times+}$ and $u_{\delta}\in U(\gf)$ are as in Proposition \ref{prop-relations}.
		The former follows from \eqref{eq:times} since for all $u_p \in U(\gf)_p$ we have 	$(f\circ \iota_\eta)^\times(\varpi_v u_p)=0$. The latter follows from the fact that if $u_p\in U(\gf)_p$ then applying Def. \ref{d:action-functions} and \eqref{eq:times} we have: 
	\begin{align*}
	{(f\circ \iota_\eta)^\times}_{\big|\begin{psmallmatrix}
		1&\\&u_\delta^{-1}\end{psmallmatrix}}(u_p)&=\left\langle \begin{psmallmatrix}
	u_\delta^{-1}&\\&1\end{psmallmatrix}\right\rangle_{\cU} (f\circ \iota_\eta)^\times(u_\delta
	u_p)= \left\langle \begin{psmallmatrix}
	u_p&\\&1\end{psmallmatrix}\right\rangle_{\cU} (f\circ \iota_\eta)(u_\delta u_p)=\\
	= \left\langle \begin{psmallmatrix}
	u_p&\\&1\end{psmallmatrix}\right\rangle_{\cU} f(u_\delta \eta u_p)
	&= \left\langle \begin{psmallmatrix}
	u_p&\\&1\end{psmallmatrix}\right\rangle_{\cU} f(\delta u_p)
	= \left\langle \begin{psmallmatrix}
	u_p&\\&1\end{psmallmatrix}\right\rangle_{\cU} (f\circ \iota_\delta)(u_p)
	=(f\circ \iota_\delta)^\times(u_p). \qedhere
	\end{align*}
\end{proof}

Suppose that $\Phi\in \rH^d_c(Y_K, \cD_{\cU})$ is such that $U_{\varpi_{\gf}} \Phi= \alpha_{\gf}^{\circ} \Phi$ with $\alpha_{\gf}^{\circ} \in \cO(\cU)^{\times}$. By Lemma \ref{l:ev-uniformizer} 
		\begin{align}\label{eq:universal-evaluation}
	\ev(\Phi)=(\alpha_{\gf}^{\circ})^{-1}\ev_{\varpi_{\gf}}(\Phi) \in D(\Gal_{p\infty},\cO(\cU))
		\end{align} 
	is independent of the choice of	uniformizers and by Prop. \ref{p:ev-relation-canonical} it is independent of 	$\gf$ as well. 

Our final result in this subsection concerns the growth of the distributions $\ev(\Phi)$ on $\Gal_{p\infty}$. This will be used in \S \ref{s:p-adic-L} to uniquely characterize by interpolation property the $p$-adic $L$-functions attached to non-critical nearly finite slope Hilbert cusp forms. 

Using the notations from \S \ref{sect:oc-modules}, 
$A(\Gal_{p\infty},\cO(\cU))$ is a union of orthonormalizable Banach $\cO(\cU)$-modules 
$A_n(\Gal_{p\infty},\cO(\cU))$, $n \in \Z_{\geqslant 0}$, and $D(\Gal_{p\infty},\cO(\cU))= \varprojlim D_n(\Gal_{p\infty},\cO(\cU))$. The 
 restriction of $\mu \in D(\Gal_{p\infty},\cO(\cU))$ to $A_n(\Gal_{p\infty},\cO(\cU))$ belongs to the orthonormalizable Banach $\cO(\cU)$-module $D_n(\Gal_{p\infty},\cO(\cU))$, and its norm is denoted by $\|\mu \|_n$. The following definition generalizes the notion of 
growth introduced by Amice-V\'elu and Vishik (see \cite[Def.4.1]{barrera}). 
\begin{definition}We say that a distribution $\mu \in D(\Gal_{p\infty},\cO(\cU))$ has growth at most $h \in \Q_{\geqslant 0}$ if there exists $C\geqslant 0$ such that for each $n \in \Z_{\geqslant 0}$ we have $\| \mu \|_n \leqslant p^{n h}C$.
\end{definition}
\begin{proposition}\label{p:ev-growth}
	Suppose $\Phi\in \rH^d_c(Y_K, \cD_{\cU})$ is such that $U_p \Phi= \alpha_p^{\circ} \Phi$ with $\alpha_p^{\circ} \in \cO(\cU)^{\times}$. Then $\ev(\Phi) \in D(\Gal_{p\infty},\cO(\cU))$ has growth at most $h_p$, where $h_p$ is the $p$-adic valuation of $\alpha_p^{\circ}$.
\end{proposition}
\begin{proof} This is proved in \cite[Prop.5.10]{barrera} when $\cU= \{\lambda\}$. 
Recall that from \cite[Lem.3.4.6]{urban} we know that there exists $m \in \Z_{\geqslant 0}$ such that the universal character $\langle \cdot \rangle_{\cU}$ is $m$-locally $\cO(\cU)$-analytic, and we may further assume that $K \supset \begin{psmallmatrix} U(p^{m}) & \widehat\cO_F \\ 0 & 1\end{psmallmatrix}$. Let
$\cO(\cU)^{\circ} \subset \cO(\cU)$ be the
 ring of rigid functions bounded by $1$ and denote by $D_{\cU,m}^{\circ}$ the $\cO(\cU)^{\circ}$-lattice in the $\cO(\cU)$-Banach space	$D_{\cU,m}$. After rescaling $\Phi$ we may assume that its image $\Phi_m$ under the natural restriction map 
 belongs to $\rH^d_c(Y_K, \cD_{\cU,m}^\circ)$.   
By \eqref{eq:ev-gal} and \eqref{eq:universal-evaluation} for $f\in A_n(\Gal_{p\infty},\cO(\cU))$, 
\[\langle (\alpha_p^{\circ})^{n} \cdot \ev(\Phi),f\rangle=
\langle \ev_{\varpi_p^n }(\Phi),f\rangle=
 \sum_{[\eta] \in	\Cl(p^{n})}
 \langle  \ev_{\varpi_p^n }^{\eta}(\Phi),(f\circ\iota_\eta)^\times \rangle.\]	
In view of \eqref{eq:times} and the fact that $\left|\langle T(\Z_p) \rangle_{\cU}\right|_p=1$, 
to prove the Proposition it suffices to bound the norms $\|\ev_{\varpi_p^n }^{\eta}(\Phi)\|_n$ in the Banach space 
$(D_{\cU,n})_{E(p^n)}$, for all $n\geqslant m$ and $\eta\in \A_F^\times$. 
By Lemma \ref{l:distribution support}, there exist $\mu\in (D_\cU)_{E_n}$ and $\mu'\in (D_{\cU,m}^\circ)_{E_n}$, where 
$E_n= \begin{psmallmatrix}	\varpi_p^n &1_{p}\\ 0&1\end{psmallmatrix}^{-1}	 
\begin{psmallmatrix}E(p^n) & 0\\ 0&1\end{psmallmatrix}
\begin{psmallmatrix}	\varpi_p^n &1_{p}\\ 0&1\end{psmallmatrix}\subset K $, such that $\ev_{\varpi_p^n }^{\eta}(\Phi)= \begin{psmallmatrix}	\varpi_p^n &1_{p}\\ 0&1\end{psmallmatrix}\cdot	\mu$ and 
$\ev_{\varpi_p^n }^{\eta}(\Phi_m)= \begin{psmallmatrix}	\varpi_p^n &1_{p}\\ 0&1\end{psmallmatrix}\cdot	\mu'$. 
By functoriality of the evaluation maps (see Lemma \ref{l:ev-functorial}) $\mu$ and $\mu'$ have the same restrictions to 
$\left(A_{\cU,m}^{E(p^n)}\right)_{\big|\begin{psmallmatrix}	\varpi_p^n &1_{p}\\ 0&1\end{psmallmatrix}}$. 

Since clearly $\left(A_{\cU,m}^{E(p^n)}\right)_{\big|\begin{psmallmatrix}	\varpi_p^n &1_{p}\\ 0&1\end{psmallmatrix}}=
\left(A_{\cU,n}^{E(p^n)}\right)_{\big|\begin{psmallmatrix}	\varpi_p^n &1_{p}\\ 0&1\end{psmallmatrix}} \subset 
A_{\cU,0}^{E_n}$, we deduce that 
	\[\|\ev_{\varpi_p^n }^{\eta}(\Phi)\|_n= 
	\left\| \left(\begin{psmallmatrix}	\varpi_p^n &1_{p}\\ 0&1\end{psmallmatrix}\cdot	\mu\right) \right\|_n= 
	\left\| \left(\begin{psmallmatrix}	\varpi_p^n &1_{p}\\ 0&1\end{psmallmatrix}\cdot	\mu' \right)\right\|_n
	\leqslant \|\mu' \|_0 \leqslant \|\mu' \|_m \leqslant 1. \qedhere\]
\end{proof}

\subsection{Distributions evaluated at norm maps}\label{ss:norms}

To compute higher derivatives of $p$-adic $L$-functions at central trivial zeros we need
to construct partially improved $p$-adic $L$-functions. These will be obtained
by evaluating the distributions $\ev_{\varpi_{\gf}}^\eta(D_{\cU})$ on certain partially polynomial 
functions in $A_{\cU}$ for certain well chosen sub-affinoids $\cU$ of $\cX$. The improvement comes from the fact that when $\gf$ is only divisible by certain primes above $p$, then 
 the support of $\ev_{\varpi_{\gf}}^\eta(D_{\cU})$ need no longer be contained in 
$(\cO_F\otimes\Z_p)^\times$ (see Lemma \ref{l:distribution support}). 	Part of the construction will also be used to attach a new kind of $p$-adic $L$-function to the partial families from Theorem \ref{free-etale}. 

Given an $L$-affinoid $\cU\subset \cX$ containing the cohomological weight $(k,\sw)$ and a subset $S\subset S_p$, 
we let $\cU'_{S}=\cU\cap \cX'_S$ (see Def. \ref{defn:XKS}). Henceforth we fix an integer $r$ such that 
\begin{align}\label{def:semi-crit}
	j_\sigma: = r-1+ \frac{\sw-2+k_\sigma}{2} \geqslant 0, \text{ for all }\sigma\in \Sigma_{S_p\!\setminus S}, 	
		\end{align}
and 	for $z_{S_p\!\setminus S} \in \cO_{F,S_p\!\setminus S} $ we let $z_{S_p\!\setminus S}^j=\prod\limits_{v\in S_p\!\setminus S} 
\prod\limits_{\sigma\in \Sigma_v}\sigma(z_v)^{j_\sigma}$. 
For the remainder of this section we will only consider ideals 	$\gf\subset \cO_F$ whose support in contained in $S$. 
We let $\overline{E(\gf)}$ denote the $p$-adic closure of $E(\gf)$ in $U(\gf)_S= \cO_{F,S}^\times\cap (1+\gf\cO_{F,S})$. Similarly to \eqref{eq:times} one considers the  map  
\begin{align}\label{eq:times-r}
 &A(U(\gf)_S/\overline{E(\gf)}, \cO(\cU'_{S})) \to A_{\cU'_{S}}^{E(\gf)},  f\mapsto f_{S,r}^\times, 
 \text{ where for } z=(z_S,z_{S_p\!\setminus S}) \in \cO_{F,S} \times \cO_{F,S_p\!\setminus S} \\&
 f_{S,r}^\times(z)= 
\begin{cases} f(z_S)\cdot z_{S_p\!\setminus S}^j \cdot \prod\limits_{v\in S}
\left\langle \begin{psmallmatrix}
	z_v &\\&1\end{psmallmatrix}\right\rangle_{\cU'_{S} } \mathrm{N}^{r-1}_{F_v/\Q_p}(z_v) & \text{, if } z_S \in U(\gf)_S, \\ 0 & \text{, if } z_S \notin U(\gf)_S. \nonumber
\end{cases}
\end{align}

Dualizing we obtain a map $(\cD_{\cU'_{S}})_{E(\gf)}\to D(U(\gf)_{S}/E(\gf), \cO(\cU'_{S}))$ denoted $\mu\mapsto \mu_{S,r}^{\times}$.

Note that for all $v\in S_p\!\setminus S$ and $z_v\in \cO_v^\times$ one has 
$\prod\limits_{\sigma\in \Sigma_v}\sigma(z_v)^{j_\sigma}=\left\langle \begin{psmallmatrix}
	z_v &\\&1\end{psmallmatrix}\right\rangle_{\cU'_{S} } \mathrm{N}^{r-1}_{F_v/\Q_p}(z_v)$.

\begin{defn} \label{def:crit}
	We say that $r\in \Z$ is $S$-critical for the cohomological weight $(k,\sw)$ if 
\[	0 \leqslant r-1+ \frac{\sw-2+k_\sigma}{2} \leqslant k_\sigma-2 \text{ for all }\sigma\in \Sigma_{S_p\!\setminus S}.	\]
When 	$S=\varnothing$ we say that $r$ is critical. 
	\end{defn}

\begin{remark}\label{rem:crit-range}
\begin{enumerate}
\item The inequality \eqref{def:semi-crit} holds for any cohomological weight in $\cX'_S$. 
\item If $r$ is $S$-critical for $(k,\sw)$, then it is $S$-critical as well for any cohomological weight in $\cX_S$. Furthermore 
if $\cU_S\subset \cX_S$ is an $L$-affinoid containing $(k,\sw)$, then for all $f\in A(U(\gf)_S/\overline{E(\gf)}, \cO(\cU_{S}))$
 one has $f_{S,r}^\times\in A_{S,\cU_{S}}^{E(\gf)}$. 
\item
Note that $r\in \Z$ is critical for $(k,\sw)$ if and only if $r-\tfrac{1}{2}$ is a critical point for the $L$-function of an automorphic representation $\pi$ of cohomological weight $(k,\sw)$, in the sense of Deligne.
 Moreover, the central point $\frac{1-\sw}{2}$ is critical if and only if $\sw$ is even.
Using \eqref{right-action} and \eqref{left-action}	one checks that $(L_{k,\sw}^\vee(L))_{E(\cO_F)}$ has a basis consisting of linear forms $\mu\mapsto \mu(z^{j})$ with $j=\frac{k+(\sw-2)t}{2}+(r-1)t$, as $r$ ranges across all critical integers for $(k,\sw)$.
\end{enumerate}
\end{remark} 

  Moreover $\vartheta_{S,v}\circ (f\circ (z_S\mapsto z_{S\!\setminus \{v\}}) _{S,r}^{\times}= (\vartheta_{S,v}\circ f)_{S\!\setminus \{v\},r}^\times$
  for $f\in A(U(\gf)_{S\!\setminus\{v\}} /\overline{E(\gf)}, \cO(\cU'_{S}))$ and $v\in S$, where  $\vartheta_{S,v}: \cO(\cU'_{S}) \to \cO(\cU'_{S\!\setminus\{v\}} )$ is the restriction map. 
Applying Lemma \ref{l:ev-functorial}  yields		 
 \begin{align}\label{eq:improved-support}
 \vartheta_{S,v}\circ \ev_{\varpi_{\gf}}^\eta(D_{\cU'_{S}})=\ev_{\varpi_{\gf}}^\eta(D_{\cU'_{S\!\setminus\{v\}}})\circ \vartheta_{S,v}.
 \end{align}
 As in \eqref{eq:iota}, for $\eta\in \A_F^{\times}$, there is an isomorphism $\iota_{\eta}^*:D(U(\gf)_{S}/E(\gf),\cO(\cU))\xrightarrow{\sim}D(\Gal_{S\infty}[\eta],\cO(\cU))$.

\begin{lemma}\label{p:ev-improved}
The following map 	does not depend on the representative $\eta\in \A_F^\times$ of 
$[\eta]\in \Cl(\gf)$ 
\begin{align}\label{eq:ev-improved}
	\ev^{[\eta],r}_{\varpi_{\gf},S} =\chi^{r-1}_{\cyc}(\eta) \left[ \iota_{\eta}^*\circ(\ev_{\varpi_{\gf}}^\eta)_{S,r}^{\times}\right] :
	 \rH^d_c(Y_K,\cD_{\cU'_{S}})	 \to D(\Gal_{S\infty}[\eta], \cO(\cU'_{S})). 
		\end{align}
	\end{lemma}

\begin{proof} Suppose $\eta'\in F^{\times} \eta u F_\infty^{\times+}$ with $u\in U(\gf)$. By Lemma \ref{l:ev-eta} we have 	\begin{align*}
	\langle \iota_{\eta}^*\circ(\ev_{\varpi_{\gf}}^{\eta})_{S,r}^\times, f
	\rangle = \langle \ev_{\varpi_{\gf}}^{\eta},
	(f\circ \iota_{\eta})_{S,r}^\times\rangle=
	\left\langle \begin{psmallmatrix} u_p&\\
	&1\end{psmallmatrix}\ev_{\varpi_{\gf}}^{\eta'},(f\circ
	\iota_{\eta})_{S,r}^\times\right\rangle
	=\left\langle \ev_{\varpi_{\gf}}^{\eta'},{(f\circ
	\iota_{\eta})_{S,r}^\times}_{\big|\begin{psmallmatrix} 1&\\
		&u_p^{-1}\end{psmallmatrix}}\right\rangle
	\end{align*}
	Using Def. \ref{d:action-functions} and the fact that $\iota_{\eta}(u_p \cdot )=\iota_{\eta'}$ we find
\begin{align}\label{eq:useful-relation}
	{(f\circ	\iota_{\eta})_{S,r}}^\times_{\big|\begin{psmallmatrix} 1&\\ &u_p^{-1}\end{psmallmatrix}} =\left\langle 
		\begin{psmallmatrix} u_p^{-1}&\\&1\end{psmallmatrix}\right\rangle_{\cU'_{S}} (f\circ
	\iota_{\eta})_{S,r}^\times(u_p\cdot)
	=\chi^{r-1}_{\cyc}(u_p)(f\circ\iota_{\eta'})_{S,r}^\times
	\end{align}
	hence $\langle \iota_{\eta}^*\circ(\ev_{\varpi_{\gf}}^{\eta})_{S,r}^\times, f \rangle = \chi^{r-1}_{\cyc}(u_p)\left\langle \ev_{\varpi_{\gf}}^{\eta'},(f\circ\iota_{\eta'})_{S,r}^\times\right\rangle =	\chi^{r-1}_{\cyc}(u_p)
	\langle \iota_{\eta'}^*\circ(\ev_{\varpi_{\gf}}^{\eta'})_{S,r}^\times, f \rangle$.
\end{proof}

The above lemma allows one to introduce the following notation analogous to Def. \ref{d:eta-bracket}: 
\[	\ev_{\varpi_{\gf},S} ^r:	 \rH^d_c(Y_K,\cD_{\cU'_{S}})	 \to D(\Gal_{S\infty}, \cO(\cU'_{S})).\]

We first state a distribution relation extending Proposition \ref{p:ev-relation-canonical}, whose proof is very similar and uses the single additional  fact that ${f_{S,r}^{\times} }_{\big|\begin{psmallmatrix}		\varpi_v&\\&1\end{psmallmatrix}}= 0$, for all $v\in S$.

\begin{proposition} \label{p:improved-distribution}
For $v \in S$ we let $\pr_{\gf v,\gf}:\Cl(\gf v)\to \Cl(\gf)$ be the natural projection. Then
	\[\ev^{[\eta],r}_{\varpi_{\gf},S}\circ U_{\varpi_v} =\sum_{[\delta]\in\pr_{\gf v,\gf}^{-1}([\eta])} 
		\ev^{[\delta],r}_{\varpi_{\gf v},S} \text{ and } \ev^{r}_{\varpi_{\gf}, S} \circ U_{\varpi_v} = \ev^{r}_{\varpi_{\gf v},S}. \]
\end{proposition}		
	
The next result will be used in 
\S \ref{s:p-adic-L} to compare $p$-adic $L$-functions and improved ones.

\begin{lemma}\label{l:improved-ev-relation}
We have $\left\langle \ev^r_{\varpi_{\gf},S_p}, \cdot \right\rangle= \left\langle \ev_{\varpi_{\gf}}, \chi^{r-1}_{\cyc} \cdot\right\rangle$. 
\end{lemma}		
		\begin{proof} Note that $\cU'_{S_p}=\cU$. 
		 By definition $f^\times_{S_p,r}(z)=f^\times(z) z^{t(r-1)}$ for $f\in A(U(\gf)_p/\overline{E(\gf)},\cO(\cU))$. 
		Using \eqref{eq:ev-improved} and Def. \ref{d:eta-bracket} we find that for any $[\eta]\in \Cl(\gf)$ and for 	any  $f\in A(\Gal_{p\infty}[\eta],\cO(\cU))$:
		 \[\left\langle \ev^{[\eta],r}_{\varpi_{\gf},S_p}, f\right\rangle=\chi^{r-1}_{\cyc}(\eta)
		 \left\langle \ev_{\varpi_{\gf}}^{\eta}, (f\circ\iota_\eta)^\times_{S_p,r}\right\rangle=
		 \left\langle \ev_{\varpi_{\gf}}^{\eta}, (f\circ\iota_\eta)^\times (\chi^{r-1}_{\cyc}\circ\iota_\eta)
		 \right\rangle=\left\langle \ev_{\varpi_{\gf}}^{[\eta]}, f \chi^{r-1}_{\cyc}\right\rangle. \qedhere
		 \]	 \end{proof}

Finally we relate the improved evaluations when $S$ varies. 
 
 \begin{proposition} \label{p:improved-relation}
 For any $v\in S$, $v\nmid \gf$ we have 
			\[\displaystyle(\ev^{[\eta],r}_{\varpi_{\gf},{S\!\setminus\{v\}}}\circ\vartheta_{S,v} - \vartheta_{S,v}\circ \ev^{[\eta],r}_{\varpi_{\gf},S})\circ U_{\varpi_v}=
			\left(q_v^{r-1}\prod_{\sigma\in \Sigma_v}\sigma(\varpi_v)^{\frac{\sw-2+k_\sigma}{2}}\right)
			 \iota_{\varpi_v^{-1}}^*\circ\ev^{[\eta\varpi_v],r}_{\varpi_{\gf}, S\!\setminus\{v\}}\circ\vartheta_{S,v}, \text{ where }\]
			 $\iota_{\varpi_v^{-1}}: \Gal_{(S\!\setminus\{v\})\infty}[\eta\varpi_v] \xrightarrow{\cdot \varpi_v^{-1}} \Gal_{(S\!\setminus\{v\})\infty}[\eta]$ and 
 $\vartheta_{S,v}: \cO(\cU'_{S}) \to \cO(\cU'_{S\!\setminus\{v\}} )$ is the restriction. 
			\end{proposition}

		\begin{proof} 		
Using ${f_{S\!\setminus\{v\},r}^\times}_{\big|\begin{psmallmatrix}		\varpi_v&\\&1\end{psmallmatrix}}= 
		 \prod\limits_{\sigma\in
		\Sigma_v}\sigma(\varpi_v)^{j_\sigma}f_{S\!\setminus\{v\},r}^\times$ and  $ \chi_{\cyc}(\varpi_v)q_v =\mathrm{N}_{F_v/\Q_p}(\varpi_v)$ we obtain:
	\begin{align*}
	&\chi^{r-1}_{\cyc}(\eta)
	\left\langle \begin{psmallmatrix} \varpi_{v} & 0 \\ 0 & 1 \end{psmallmatrix} \cdot \ev_{\varpi_{\gf}}^{\eta\varpi_v}, (f\circ	\iota_{\eta})_{S\!\setminus\{v\},r}^\times\right \rangle=	\chi^{r-1}_{\cyc}(\eta)
	 \prod_{\sigma\in	\Sigma_v}\sigma(\varpi_v)^{j_\sigma}
	 \left\langle \ev_{\varpi_{\gf}}^{\eta\varpi_v}, (f\circ	\iota_{\eta})_{S\!\setminus\{v\},r}^\times\right \rangle=\\
	& =
	\chi^{r-1}_{\cyc}(\eta\varpi_v) \left\langle \iota_{\eta}^*(\ev_{\varpi_{\gf}}^{\eta\varpi_v})_{S\!\setminus\{v\},r}^\times, f\right \rangle
	=q_v^{r-1}\prod_{\sigma\in	\Sigma_v}\sigma(\varpi_v)^{j_\sigma-r+1} \left\langle \iota_{\varpi_v^{-1}}^*		
			\ev^{[\eta\varpi_v],r}_{\varpi_{\gf}, S\!\setminus\{v\}}, f\right \rangle.
	 \end{align*}
	
Since $j_\sigma-r+1 =\frac{\sw-2+k_\sigma}{2}$, Proposition \ref{prop-relations} applied to the case $v\nmid \gf$ yields
	\begin{align*} &\chi^{1-r}_{\cyc}(\eta)	\left\langle 
\ev^{[\eta],r}_{\varpi_{\gf},{S\!\setminus\{v\}}}\circ U_{\varpi_v}-
			q_v^{r-1}\prod_{\sigma\in \Sigma_v}\sigma(\varpi_v)^{\frac{\sw-2+k_\sigma}{2}}
			 \iota_{\varpi_v^{-1}}^*			\ev^{[\eta\varpi_v],r}_{\varpi_{\gf}, S\!\setminus\{v\}} ,f \right \rangle= \\
			& 
	 =\left\langle \ev_{\varpi_{\gf}}^{\eta}\circ U_{\varpi_v}- 
	\begin{psmallmatrix} \varpi_{v} & 0 \\ 0 & 1 \end{psmallmatrix} \cdot \ev_{\varpi_{\gf}}^{\eta\varpi_v}, 
	(f\circ	\iota_{\eta})_{S\!\setminus\{v\},r}^\times\right \rangle=	
	\sum_{[\delta]\in\pr_{\gf v,\gf}^{-1}([\eta])}
\left\langle \begin{psmallmatrix} u_{\delta}&\\ &1\end{psmallmatrix}\cdot\ev_{\varpi_{\gf v}}^{\delta}, (f\circ	\iota_{\eta})_{S\!\setminus\{v\},r}^\times \right \rangle.
\end{align*}

By \eqref{eq:improved-support} 	and a computation similar to the proof 
 of Proposition \ref{p:ev-relation-canonical} we obtain 	
\[\displaystyle 
 \left\langle \begin{psmallmatrix} u_{\delta}&\\ &1\end{psmallmatrix}\cdot\ev_{\varpi_{\gf v}}^{\delta}
 \circ\vartheta_{S,v}
 , (f\circ	\iota_{\eta})_{S\!\setminus\{v\},r}^\times \right \rangle= \chi^{1-r}_{\cyc}(\eta) \cdot \vartheta_{S,v} \left(\left\langle \ev^{[\delta],r}_{\varpi_{\gf v},S},f\right\rangle\right).\]
 
 Finally, by 		Proposition \ref{p:improved-distribution} 
	we find $ \displaystyle \sum_{[\delta]\in\pr_{\gf v,\gf}^{-1}([\eta])} \vartheta_{S,v}\circ \ev^{[\delta],r}_{\varpi_{\gf v},S}	
		= \vartheta_{S,v}\circ \ev^{[\eta],r}_{\varpi_{\gf},S}\circ U_{\varpi_v}$.
		\end{proof}

		Let $\Phi\in \rH^d_c(Y_K, \cD_{\cU'_{S}})$ be such that for all $v\in S$ we have $U_{\varpi_v} \Phi= \alpha^{\circ}_{v} \Phi$ with $\alpha_v^\circ \in \cO(\cU'_{S})^{\times}$. 
Letting $\alpha_{\gf}^{\circ} =\prod_{v\in S} (\alpha_v^{\circ})^{n_v}$, where 
$n_v$ denotes the valuation of $\gf$ at $v$, the distribution 
\begin{align}\label{eq:universal-improved-evaluation}
\ev_{S}^{r}(\Phi)=(\alpha_{\gf}^{\circ})^{-1}\ev_{\varpi_{\gf}, S}^{r}(\Phi) \in D(\Gal_{S\infty},\cO(\cU'_{S}))
\end{align} 
 is independent of the choice of uniformizers (see Lem. \ref{l:ev-uniformizer}) as well of the
the ideal $\gf$ (see  Prop. \ref{p:improved-distribution}). Lemma \ref{l:improved-ev-relation} 
and \eqref{eq:universal-evaluation} then imply that: 
\begin{align}\label{eq:improved-evK-relation}
\left\langle \ev^r_{S_p}(\Phi), f \right\rangle= \left\langle \ev(\Phi), \chi^{r-1}_{\cyc} f\right\rangle
\text{ for all } f\in A(\Gal_{p\infty},\cO(\cU'_{S_p}))=A(\Gal_{p\infty},\cO(\cU)).
\end{align}

The following important consequence of 
Proposition \ref{p:improved-relation} will be used in \S\ref{ss:improved}. 
\begin{corollary}\label{c:ev-evtilde} 
For $\Phi\in \rH^d_c(Y_K, \cD_{\cU'_{S}})$ as above, letting $\displaystyle \alpha_v = \alpha_v^\circ\prod_{\sigma\in \Sigma_v}\sigma(\varpi_v)^{\frac{2-\sw-k_\sigma}{2}} $, then for any continuous character $\chi: \Gal_{(S\!\setminus\{v\})\infty}\rightarrow \cO(\cU'_{S\!\setminus\{v\}})^{\times}$ we have:
 \[ \left\langle \vartheta_{S,v}(\ev_{S}^{r}(\Phi)), \chi \right\rangle= \left\langle \ev_{ S\!\setminus\{v\}}^{r}(\vartheta_{S,v}(\Phi)), \chi \right\rangle \left(1- \frac{q_v^{r-1}}{\vartheta_{S,v}(\alpha_v)\chi(\varpi_v)} \right). \]
 \end{corollary}

\section{$p$-adic $L$-functions}\label{s:p-adic-L}
In this section we use the distribution valued maps from \S\ref{automorphic-symbols} to attach cyclotomic $p$-adic $L$-functions to 
rigid analytic families of non-critically refined Hilbert cusp forms, which are uniquely determined by an interpolation property
(see Theorem \ref{thm:multi-L_p}). We also construct improved $p$-adic $L$-functions, as well as `partial' $p$-adic $L$-functions for families of $S$-refined cusp forms, which do not appear to have been previously brought into light. 

Let $\pi$ be a cuspidal automorphic representation of $G(\A)$ of central character  $\omega_\pi$ 
and  cohomological weight $(k,\sw)$ (see Def. \ref{d:coh-weight}). 
Throughout this section we assume that $\pi_v$ has nearly finite slope for all $v\in S_p$, except in \S\ref{ss:partial} where we only 
assume this at $S\subsetneq S_p$. 
For $\widetilde{\pi}= (\pi, \{\nu_v\}_{v\in S_p})$ a (regular) non-critical $p$-refinement (see Def. \ref{d:hilbert-nearly} and Def. \ref{d:non-critical}) we consider  the neat open compact subgroup $K=K(\widetilde{\pi},\gu)\subset G(\A_f)$ from Def. \ref{d:newline} and the $(p,\gu)$-refined newforms $\phi_{\widetilde{\pi},\alpha_\gu}$, $\phi_{\widetilde{\pi},\beta_\gu}$ from Def. \ref{d:newform}.

\subsection{$p$-adic $L$-functions for nearly finite slope Hilbert cusp forms}\label{ss:p-adic L-functions attached}
 Let $L/\Q_p$ be a finite extension containing the image by $\iota_p$ of the number field $E$ from Def. \ref{d:mpi-tilde}. 
By cuspidality and non-criticality of $\widetilde{\pi}$, for each character $\epsilon: \{\pm 1\}^\Sigma \to \{\pm 1\}$, 
the basis $\iota_p(b_{\widetilde{\pi},\alpha_\gu}^{\epsilon})$
of $\rH^d_{\cusp}(Y_K, \cL_{k,\sw}^\vee(L))^\epsilon_{\gm_{\widetilde{\pi}}}=
\rH^{\bullet}_{c}(Y_K,\cL_{k,\sw}^\vee(L))^\epsilon_{\gm_{\widetilde{\pi}}}$
(see \eqref{eq:E-line}) lifts canonically to a basis of  $\Phi_{\widetilde{\pi},\alpha_\gu}^{\epsilon}$ of 
$\rH_c^d(Y_K,\cD_{(k,\sw)})^\epsilon_{\gm_{\widetilde{\pi}}}$ having the same $U_{\varpi_v}$-eigenvalue
$\alpha_v^\circ\in L^\times$, $v\in S_p$. For $\gf=\prod_{v\in S_p} v^{n_v}$ we let 
\begin{align}\label{eq:alphas}\alpha_\gf^\circ=\prod_{v\mid
 p} (\alpha_v^\circ)^{n_v} \text{ and } \alpha_\gf=\prod_{v\mid
 p} \alpha_v^{n_v},\textrm{ where }
\alpha_v=\nu_v(\varpi_v)=\alpha_v^\circ\prod_{\sigma\in
 \Sigma_v}\sigma(\varpi_v)^{\frac{2-\sw-k_\sigma}{2}}.
\end{align}

Consider the distribution $\ev(\Phi_{\widetilde{\pi},\alpha_\gu}^{\epsilon})\in D(\Gal_{p\infty}, L)$ defined in \eqref{eq:universal-evaluation}. In order to attach a $p$-adic $L$-function to $\widetilde{\pi}$ without missing Euler factors at $\gu$ we need to also consider the distribution $\ev(\Phi_{\widetilde{\pi},\beta_\gu}^{\epsilon})\in D(\Gal_{p\infty}, L)$
using the other Hecke parameter $\beta_\gu\neq \alpha_\gu$ of $\pi_\gu$ (see Def. \ref{d:u}). We let 
\begin{align}\label{def:padicL}
\cL_p(\widetilde{\pi})= \sum_{\epsilon: \{\pm 1\}^\Sigma\to \{\pm 1\}}
\frac{ \alpha_{\gu} \ev(\Phi_{\widetilde{\pi},\alpha_{\gu }}^{\epsilon})
	-\beta_{\gu}	 \ev(\Phi_{\widetilde{\pi},\beta_{\gu }}^{\epsilon})}{\alpha_{\gu}-\beta_{\gu}}\in D(\Gal_{p\infty}, L). 
\end{align}

For any $f\in A(\Gal_{p\infty},L)$, we let $\cL_p(\widetilde{\pi},f)=\cL_p(\widetilde{\pi})(f)$. 

By Prop. \ref{p:ev-growth} the distribution $\cL_p(\widetilde{\pi})$ has growth at most $\sum\limits_{v\in S_p}e_{v}h_{\widetilde{\pi}_v}$  
(see Def. \ref{d:non-critical-slope}), and we will next show that it interpolates critical values of the archimedean 
$L$-function of $\pi$ and its twists.

We let $r$ be a critical integer for the weight $(k,\sw)$ and  $j = (r-1)t +\frac{(\sw-2)t+k}{2}\geqslant 0$ (see Def. \ref{def:crit}). 
As observed in Rem. \ref{rem:crit-range}(ii) we have $z^j \in L_{k,\sw}(L)^{E(\cO_F)}\subset A_{(k,\sw)}(L)^{E(\cO_F)}$.
 Let  $\Omega_{\widetilde{\pi}}^{\epsilon} \in \C^{\times}$ be the period from Def. \ref{of-the-period-omega-pi}. 
 The following key Proposition allows us to relate the values at $z^j$ of the distributions
constructed in \S \ref{automorphic-symbols} to certain adelic integrals. 

\begin{proposition}\label{p:ev-critical}
	Let $\gf\mid p^\infty$ be an integral ideal, $\eta \in \mathbb{A}_F^\times$, and let 
	$\ev_{\varpi_{\gf}}^{\eta}=\ev_{\varpi_{\gf}}^{\eta}(D_{(k,\sw)})$. Then 
		\[\frac{\alpha_\gf \chi^{r-1}_{\cyc} (\eta)}{\alpha_\gf^\circ} 
	\langle \ev_{\varpi_{\gf}}^{\eta}(\Phi_{\widetilde{\pi},\alpha_\gu}^{\epsilon}), z^j \rangle = \frac{i^{(r-1)d}}{\Omega_{\widetilde{\pi}}^\epsilon}
	\sum_{s_\infty\in \{\pm 1\}^{\Sigma}}s_\infty^{(r-1)t}
	\epsilon(s_\infty)\int_{X_{\gf}[\eta s_\infty ]}\phi_{\widetilde{\pi},\alpha_\gu}\begin{pmatrix}
	y \varpi_{\gf} &y 1_{\gf}\\&1\end{pmatrix} |y|_F^{r-1}
	d^\times y.\]
\end{proposition}

\begin{proof}Since the right hand side in the above formula is in $\C$, we will 
	first prove that the left hand side, which is a priori a $p$-adic number, belongs in fact to $\iota_p(E)$. 
	Denote by $\cL_{k,\sw}^\vee(E)$ the $G(\Q)$-construction of
	a local system attached to $L_{k,\sw}^\vee(E)$. Recall that $\cL_{k,\sw}^\vee(L)$ denotes
	the local system attached to $L_{k,\sw}^\vee(L)$ by the $K_p$-construction. The following diagram commutes: 
		\begin{align}\label{eq:big-diagram}
		\xymatrix{
		\rH^d_c(Y_K,\cL_{k,\sw}^\vee(E))\ar[d]^{\cT_{\gf}}\ar@{^{(}->}[rrrrr]^{(g,v)\mapsto (g, g^{-1}\cdot v)} &&&&& \rH^d_c(Y_K,\cL_{k,\sw}^\vee(L))\ar[d]^{\tw_{\varpi_{\gf}}\circ C_{\varpi_{\gf}}^*}\\
		\rH^d_c(X_{\gf}, \cL_{k,\sw}^\vee(E))\ar@{^{(}->}[rrrrr]^{(y,v)\mapsto
			\left(y, \begin{psmallmatrix} y^{-1}&\\&1\end{psmallmatrix}\cdot
			v\right)} \ar[d]^{\triv_\xi^*\circ\coinv_{\gf}}&&&&& \rH^d(X_{\gf}, \cL_{k,\sw}^\vee(L))\ar[d]^{\triv_\eta^*\circ\coinv_{\gf}}\\
		\rH^d_c(X_{\gf}[\eta])\otimes
		L_{k,\sw}^\vee(E)_{E(\gf)}\ar[d]^{\langle -\cap \theta_{\gf,[\eta]}, z^j\rangle}\ar@{^{(}->}[rrrrr]^{\id \otimes \begin{psmallmatrix} \eta^{-1} &\\&1\end{psmallmatrix}_p}
		&&&&& \rH^d_c(X_{\gf}[\eta])\otimes
		L_{k,\sw}^\vee(L)_{E(\gf)}\ar[d]^{\langle -\cap \theta_{\gf,[\eta]}, z^j\rangle}\\
		E\ar@{^{(}->}[rrrrr]^{ (\eta_p^{-1})^{(r-1)t} }&&&&&L
	}
	\end{align}
	where the horizontal maps are induced from the morphisms of local systems written above them, 
	the map $\cT_{\gf}$ is induced from the morphisms of local systems $(y,v)\mapsto \left(\begin{psmallmatrix} y \varpi_{\gf}&y 1_{\gf}\\&1\end{psmallmatrix}, v\right)$ and, for $\xi\in F^\times$, $\triv_\xi$ 
	is induced from the morphisms of local systems:
	\[ X_{\gf}[\eta] \times L_{k,\sw}^\vee(E)_{E(\gf)} \to
	\left(\cL_{k,\sw}^\vee(E)_{E(\gf)}\right)_{|X_{\gf}[\eta]} \text{ , } (y,v)\mapsto\left(y,\begin{psmallmatrix}\xi &\\&1 \end{psmallmatrix}^{-1}\cdot v\right).\]
	
	By definition of the evaluations in \S \ref{automorphic-symbols} and by the  functoriality relation \eqref{eq:ev-functorial}, the 
composition of the maps in the right column sends $\iota_p(b_{\widetilde{\pi},\alpha_\gu}^{\epsilon})$ to 
	$\langle \ev_{\varpi_{\gf}}^{\eta}(L_{k,\sw}^\vee)(b_{\widetilde{\pi},\alpha_\gu}^{\epsilon}), z^j \rangle=
	\frac{\alpha_\gf}{\alpha_\gf^\circ}  \langle \ev_{\varpi_{\gf}}^{\eta}(\Phi_{\widetilde{\pi},\alpha_\gu}^{\epsilon}), z^j \rangle$.

	The commutativity of the diagram then yields
	\begin{align}\label{eq:ev-rational}
	\langle (\triv_\xi^*\circ\coinv_{\gf}\circ\cT_{\gf} )(b_{\widetilde{\pi},\alpha_\gu}^{\epsilon})\cap \theta_{\gf,[\eta]}, z^j \rangle
	=\frac{\alpha_\gf}{\alpha_\gf^\circ} 
	\eta_p^{(r-1)t} \langle \ev_{\varpi_{\gf}}^{\eta}(\Phi_{\widetilde{\pi},\alpha_\gu}^{\epsilon}), z^j \rangle\in E.
	\end{align}	
	Since the left column in \eqref{eq:big-diagram} can be reproduced with 
	$\cL_{k,\sw}^\vee(\C)$ instead of $\cL_{k,\sw}^\vee(E)$, it follows that the left hand side of \eqref{eq:ev-rational} can be computed via analytic methods, namely the comparison between Betti and de Rham cohomology over $\C$. 
	Since $\chi^{r-1}_{\cyc} (\eta)\eta_p^{(1-r)t}=|\eta|_{F,f}^{r-1}$ one has to show
	\begin{align}\label{eq:ev-critical-equiv}
	&|\eta|_{F,f}^{r-1}
	\langle (\triv_\xi^*\circ\coinv_{\gf}\circ\cT_{\gf} )(\Theta^\epsilon_{\pi}(\phi_{\widetilde{\pi},\alpha_\gu,f}))\cap \theta_{\gf,[\eta]}, z^j \rangle= \\ 
	&=i^{(r-1)d}
	\sum_{s_\infty\in \{\pm 1\}^{\Sigma}} s_\infty^{(r-1)t} \epsilon(s_\infty)\int_{X_{\gf}[\eta]}\phi_{\widetilde{\pi},\alpha_\gu}\left( \begin{psmallmatrix} y s_\infty \varpi_{\gf} &y 1_{\gf}\\&1\end{psmallmatrix}\right) |y|_F^{r-1} d^\times y. \nonumber
	\end{align}

	Explicitly, $\Theta^\epsilon_{\pi}(\phi_{\widetilde{\pi},\alpha_\gu,f})\in \rH^d_{\cusp}(Y_K, \cL_{k,\sw}^\vee(\C))= \rH^d_{\dR,!}(Y_K, \cL_{k,\sw}^\vee(\C))$ is obtained as follows. For each $\eta\in G(\A_f)$ 
	the relative Lie algebra differential 
	\[
	\bigotimes_{\sigma\in \Sigma}w_\sigma^*\otimes  \eval_{i}\otimes\phi_\sigma 
	\in  \rH^d(\gg_\infty, K_\infty^+,  L_{k,\sw}^\vee(\C)\otimes\pi_\infty )\]
	yields a left-invariant $d$-form $\phi_{\widetilde{\pi},\alpha_\gu}\left(\begin{psmallmatrix} \eta & \\&1\end{psmallmatrix} g_\infty\right) (g_\infty\cdot \eval_{i}) (g_\infty^{-1})^*(\wedge_\sigma w_\sigma^*) $
	on $G_\infty^+/K_{\infty}^+ $ which descends to $\Gamma_\eta\!\setminus G_\infty^+/K_{\infty}^+$
	and, once translated by $\begin{psmallmatrix} \eta & \\&1\end{psmallmatrix}$, yields a 
	$d$-form on $Y_K[\eta]$:
	\begin{equation*}
	\displaystyle \phi_{\widetilde{\pi},\alpha_\gu}(g) (g_\infty\cdot \eval_{i}) (g^{-1})^*(\wedge_\sigma w_\sigma^*).
	\end{equation*}
Then $\displaystyle 
	\Theta^\epsilon_{\pi}(\phi_{\widetilde{\pi},\alpha_\gu,f})= i^{\sum_{\sigma\in \Sigma}(2-\sw-k_\sigma)/2}
	\sum_{s_\infty\in 
		\begin{psmallmatrix} \pm 1 & \\&1\end{psmallmatrix}^{\Sigma} } \epsilon(s_\infty)
	\phi_{\widetilde{\pi},\alpha_\gu}(g s_\infty) (g_\infty s_\infty\cdot \eval_{i}) ((g s_\infty)^{-1})^*(\wedge_\sigma w_\sigma^*)$.
	
	By \cite{dimitrov-pune} there exists $v\in F$ and a commutative diagram: 
	\[\xymatrix{
		\Gamma_\eta\!\setminus G_\infty^+/K_{\infty}^+ \ar^{\begin{psmallmatrix} y & \\&1\end{psmallmatrix}}[rr]&& Y_K[\eta] \\
		E(\gf) \!\setminus F_\infty^{\times+}  \ar^{\cdot y }[rr]\ar^{u_\infty \mapsto 
			\begin{psmallmatrix} u_\infty & -v_\infty \\&1\end{psmallmatrix}}[u]&& X_{\gf}[\eta] \ar^{C_{\varpi_{\gf}}}[u]} 
	\]
	inducing for $g=C_{\varpi_{\gf}}(y)$ and $\gh$ the Lie algebra of $\GL_1$ a commutative diagram
	\[\xymatrix{
		(\gg/\gk)^* \ar^{(g^{-1})^*}[rr] \ar^{\iota^*}[d] && (T_g Y_K[\eta])^* \ar^{C_{\varpi_{\gf}}^*}[d]\\
		\gh^* \ar^{(y^{-1})^* }[rr]&& (T_y X_{\gf}[\eta])^*
	} \]
	
	While $\gg/\gk$ denotes the tangent space of $G_\infty^+/K_{\infty}^+$ at $\begin{psmallmatrix} 1 & -v_\infty \\&1\end{psmallmatrix}$ and not at the identity, the map $\iota:\gh \to \gg/\gk$ is still given by $u_\sigma=1\mapsto \begin{psmallmatrix} 1 & 0 \\0 & 0\end{psmallmatrix}=
	(w_\sigma+\bar w_\sigma)$ since horizontal translations in the upper half plane do not change $dy$. Hence 
	\[
	\cT_\gf( \phi_{\widetilde{\pi},\alpha_\gu}(g) (g_\infty\cdot \eval_{i}) (g^{-1})^*(\wedge_\sigma w_\sigma^*))
	= \phi_{\widetilde{\pi},\alpha_\gu} \begin{psmallmatrix} y \varpi_{\gf}&y 1_{\gf}\\&1\end{psmallmatrix}
	(\begin{psmallmatrix} y_\infty & \\&1\end{psmallmatrix}\cdot \eval_{i})
	(y^{-1})^*(\wedge_\sigma u_\sigma^*)
	\]
	Since $y=\xi\eta u u_\infty$ we have $y_\infty=\xi_\infty u_\infty$ and 
	\[\triv_\xi^*\circ\coinv_{\gf} (\cT_\gf( \phi_{\widetilde{\pi},\alpha_\gu}(g) (g_\infty\cdot \eval_{i}) (g^{-1})^*(\wedge_\sigma w_\sigma^*)))=
	(\phi_{\widetilde{\pi},\alpha_\gu} \begin{psmallmatrix} y \varpi_{\gf}&y 1_{\gf}\\&1\end{psmallmatrix}
	(\begin{psmallmatrix} u_\infty & \\&1\end{psmallmatrix}\cdot \eval_{i})
	(y^{-1})^*(\wedge_\sigma u_\sigma^*)).
	\]
	
	Now, a top degree invariant differential is a Haar measure, hence $(y^{-1})^*(\wedge_\sigma u_\sigma^*)=d^\times y$. By \eqref{left-action} we have 
	$(\begin{psmallmatrix} u_\infty & \\&1\end{psmallmatrix}\cdot \eval_{i})(z^j)= u_\infty^{(r-1)t} i^j$ 
	and since $|y|_F=|\eta|_{F,f} u_\infty^t$ we deduce: 
	\[|\eta|_{F,f}^{r-1}
	\langle (\triv_\xi^*\circ\coinv_{\gf}\circ\cT_{\gf} )( \phi_{\widetilde{\pi},\alpha_\gu}(g) (g_\infty\cdot \eval_{i}) (g^{-1})^*(\wedge_\sigma w_\sigma^*))\cap \theta_{\gf,[\eta]}, z^j \rangle=\]
	\[ =i^j
	\int_{X_{\gf}[\eta]}\phi_{\widetilde{\pi},\alpha_\gu}\left( \begin{psmallmatrix} y \varpi_{\gf} &y1_{\gf} \\&1\end{psmallmatrix}\right) |y|_F^{r-1} d^\times y.\]
	
	Since $j_\sigma+\tfrac{2-\sw-k_\sigma}{2}=r-1$ for all $\sigma\in \Sigma$, and since right translating $g$ by elements of $ \begin{psmallmatrix} \pm 1 & \\&1\end{psmallmatrix}^{\Sigma} $ amounts to translating $u_\infty$ by 
	elements of $\{\pm 1\}^{\Sigma}$, one obtains \eqref{eq:ev-critical-equiv} and hence the claim. 
\end{proof}

We now prove the main interpolation relation between $p$-adic and complex $L$-functions. 

\begin{theorem}\label{t:interpolation-one-form}
	Let $\chi$ be a finite order character of $\Gal_{p\infty}$ and for 
	$v$ dividing $p$ denote by $c_v$ the conductor of $\chi_v\nu_v$ . Let $r\in \Z$ be a critical integer for
	$(k,\sw)$. Letting $\mathrm{N}_{F/\Q}(i)=i^{d}$, one has
	\begin{align}\label{eq:interpolation-formula}
	\cL_p(\widetilde{\pi}, \chi\cdot \chi^{r-1}_{\cyc}) = 
	\frac{\mathrm{N}_{F/\Q}^{r-1}(i\gd)\chi(\varpi_\gd^{-1})}{\Omega_{\widetilde{\pi}}^{\chi_\infty\omega_{p,\infty}^{r-1}}} 
	L(\pi\otimes\chi,r-\tfrac{1}{2}) \prod_{v\in S_p} E(\widetilde{\pi}_v,\chi_v,r)\text{, where }
	\end{align}
	\[ E(\widetilde{\pi}_v,\chi_v,s)=\begin{cases}
q_v^{s c_v} (\chi_v\nu_v)(\varpi_v^{\delta_v}) \tau(\chi_v\nu_v,\psi_v,d_{\chi_v\nu_v})&, \text{ if } c_v \geqslant 1
\text{ and }  \chi_v\omega_\pi\nu_v^{-1} \text{ ramified},  \\
\left(1-\frac{(\chi_v\omega_\pi\nu_v^{-1})(\varpi_v)}{q_v^{s-1}}\right)&, \text{ if } c_v \geqslant 1
\text{ and }  \chi_v\omega_\pi\nu_v^{-1} \text{ unramified},  \\
\left(1-\frac{(\chi_v\omega_\pi\nu_v^{-1})(\varpi_v)}{q_v^{s-1}}\right)\left(1-\frac{q_v^{s-1}}{(\chi_v\nu_v)(\varpi_v)}\right)&, 
\text{ if }  \pi_v\otimes \chi_v^{-1} \text{ is unramified}, \\
\left(1-\frac{q_v^{s-1}}{(\chi_v\nu_v)(\varpi_v)}\right)&, \text{ otherwise}.
\end{cases}\]
\end{theorem}

\begin{proof}
	 We will evaluate $\cL_p(\widetilde{\pi}, \chi\cdot \chi^{r-1}_{\cyc})$ using automorphic symbols of 
	level $\gf=\prod_{v\in S_p} v^{n_v}$ such that $n_v\geqslant \max(c_v,1)$ 
	and will see that the result does not depend on $\gf$, as expected. 
	
	Using the notations from \eqref{eq:alphas},  \eqref{eq:universal-evaluation},
	Def. \ref{d:eta-bracket}, \eqref{eq:ev-gal} and \eqref{eq:times}, one gets: 
		\begin{align*}
  &\langle \ev(\Phi_{\widetilde{\pi},\alpha_\gu}^{\epsilon}), \chi\cdot \chi^{r-1}_{\cyc}\rangle
	=(\alpha_{\gf}^\circ)^{-1}\langle
  \ev_{\varpi_{\gf}}(\Phi_{\widetilde{\pi},\alpha_\gu}^{\epsilon}),
  \chi\cdot \chi^{r-1}_{\cyc}\rangle=	\\
&=(\alpha_{\gf}^\circ)^{-1}
\sum_{[\eta]\in\Cl(\gf)}\chi(\eta)\langle \ev_{\varpi_{\gf}}^\eta(\Phi_{\widetilde{\pi},\alpha_\gu}^{\epsilon}),
(\chi^{r-1}_{\cyc}\circ\iota_\eta)^\times\rangle
=\alpha_{\gf}^{-1} \sum_{[\eta]\in\Cl(\gf)}\chi(\eta)
\frac{\alpha_\gf \chi^{r-1}_{\cyc} (\eta)}{\alpha_\gf^\circ} 
\langle	\ev_{\varpi_{\gf}}^\eta(\Phi_{\widetilde{\pi},\alpha_\gu}^{\epsilon}),z^j\rangle.
	\end{align*}
		By Proposition \ref{p:ev-critical} the latter sum vanishes unless 		$\epsilon=\chi_\infty\omega_{p,\infty}^{r-1}$, 
		in which case it equals: 		
		\begin{align*}
	\cI=\frac{2^d i^{(r-1)d}}{\Omega_{\widetilde{\pi}}^{\epsilon}}
	\int_{X_{\gf}}\chi(y)\phi_{\widetilde{\pi},\alpha_\gu}\begin{pmatrix} y \varpi_{\gf} &y 1_{\gf} \\&1\end{pmatrix}|y|_F^{r-1} d^\times y
	\end{align*}
	Since $X_{\gf}=F^\times_+\!\setminus	\A_{F,f}^\times F_\infty^{\times+}/U(\gf)$ and the Haar measure on $\A_F^\times$ gives $U(1)$ volume $1$, we have
	\begin{align*}
\cI=	\frac{2^d i^{(r-1)d}}{ \Omega_{\widetilde{\pi}}^\epsilon}\prod_{v\in S_p} q_v^{n_v} \left(1-\tfrac{1}{q_v}\right) \cdot 
	\int_{F^\times_+\!\setminus \A_{F,f}^\times F_\infty^{\times+}}\chi(y)\phi_{\widetilde{\pi},\alpha_\gu}\begin{pmatrix} y \varpi_{\gf}&y 1_{\gf}\\&1\end{pmatrix}|y|_F^{r-1} d^\times y.
	\end{align*}
	Since $\begin{psmallmatrix} y \varpi_{\gf}&y 1_{\gf}\\&1\end{psmallmatrix}\in G(\A_f) G_\infty^+$, 
	using the Fourier expansion formula \eqref{eq:Fourier} we further compute: 
	\begin{align*}
	\int_{F^\times_+\!\setminus \A_{F,f}^\times F_\infty^{\times+}}\chi(y)\phi_{\widetilde{\pi},\alpha_\gu}\begin{psmallmatrix}
	y \varpi_{\gf}&y 1_{\gf}\\&1\end{psmallmatrix}|y|_F^{r-1} d^\times y =
	\int_{ \A_{F,f}^\times F_\infty^{\times+}}\chi(y)W_{\widetilde{\pi},\alpha_\gu}\begin{psmallmatrix}
	y \varpi_{\gf}&y 1_{\gf}\\&1\end{psmallmatrix}|y|_F^{r-1} d^\times y= 
	\prod_{v}Z_v.
	\end{align*}
A standard calculation (see, e.g., \cite[(16)]{dimitrov}) shows that:
	\[Z_\sigma=\int_{\R^\times_+}W_\sigma \begin{psmallmatrix}
	y&\\&1\end{psmallmatrix}y^{r-2} dy= \int_0^\infty e^{-2\pi y} y^{j_\sigma}dy=
	\frac{j_\sigma !}{(2\pi)^{j_\sigma+1}}=
	\tfrac{1}{2} L(\pi_\sigma,r-\tfrac{1}{2})\text{, if }
	\sigma\mid\infty,\]
A straightforward generalization of \cite[Prop.3.5]{raghuram-tanabe} from $\chi_v$ trivial to $\chi_v$ unramified yields:
	\begin{align*}
	& Z_v=\int_{F_v^\times}\chi_v(y)W_v^{\new} \begin{psmallmatrix}
	y&\\&1\end{psmallmatrix}|y|_{v}^{r-1} d^\times y=
	(q_v^{1-r} \chi_v(\varpi_v))^{-\delta_v}L(\pi_v\otimes
	\chi_v,r-\tfrac{1}{2})\text{, if } v\nmid p
	\gu\infty \text{, and }\\
	& Z_{\alpha_\gu}=\displaystyle \int_{F_{\gu}^\times}\chi_{\gu}(y)W_{\gu}^{\alpha}\begin{psmallmatrix}
	y&\\&1\end{psmallmatrix}|y|_{\gu}^{r-1}d^\times y_{\gu}=L(\pi_\gu \otimes\chi_\gu, r-\tfrac{1}{2})\left(1-\tfrac{\chi_\gu(\varpi_\gu) \beta_\gu}{q_\gu^r}\right), 
	\end{align*}
hence 	$\dfrac{\alpha_\gu Z_{\alpha_\gu} -\beta_\gu Z_{\beta_\gu}}{\alpha_\gu-\beta_\gu}=L(\pi_\gu \otimes \chi_\gu, r-\tfrac{1}{2})$. 
		Finally, for  $v\in S_ p$ the integral  	 $Z_v$ is computed in Proposition \ref{p:local-whittaker-integrals}
		and $E(\widetilde{\pi}_v,\chi_v,s)=Q(\chi_v\nu_v,s)/ L(\pi_v\otimes\chi_v,s)$.
	Putting everything together yields the desired formula. 
\end{proof}  

\begin{remark} The interpolation formula \eqref{eq:interpolation-formula} is independent of the choice of uniformizers $\varpi_v$ at $v\in S_p$, since
$\Omega_{\widetilde{\pi}}^{\epsilon} \prod\limits_{v\in S_p} \nu_v(\varpi_v^{-\delta_v})$ is independent of that choice 
by Proposition \ref{p:twisting-periods}. 
\end{remark}

We end this subsection by classifying the trivial zeros of $\cL_p(\widetilde{\pi})$, {\it i.e.}, by determining 
when $\cL_p(\widetilde{\pi}, \chi\cdot \chi^{r-1}_{\cyc})$ 
 in \eqref{eq:interpolation-formula} 
 vanishes regardless of the value of $L(\pi\otimes\chi,r-\tfrac{1}{2})$.

\begin{proposition}\label{p:trivial-zeros-location} Given a finite order character 
 $\chi_v$ of $F_v^\times$ and $r\in \Z$, one has $E(\widetilde{\pi}_v,\chi_v,r)=0$ if and only if 
\begin{enumerate}
\item either $\pi_v\otimes \nu_v^{-1}$ is the Steinberg representation, $r=\tfrac{2-\sw}{2}$,  
and $\chi_v=\nu_v^{-1}\cdot \mathrm{unr}(q_v^{-\sw/2})$, 
\item or $\pi_v$ is a principal series, $r=\tfrac{3-\sw}{2}$
(resp. $r=\tfrac{1-\sw}{2}$) and $\chi_v=\nu_v^{-1}\cdot \mathrm{unr}(q_v^{(1-\sw)/2})$
(resp. $\chi_v=\nu_v\omega_\pi^{-1}\cdot \mathrm{unr}(q_v^{-(1+\sw)/2})$. 
\end{enumerate}

In the first case $\sw$ is necessarily even and the trivial zero occurs at a central critical point, while 
in the second case $\sw$ is odd and the trivial zero occurs at a nearly-central critical point. 
\end{proposition}
\begin{proof}
If $\pi_v$ is a twist (necessarily by $\nu_v$) of the unitary Steinberg representation, then 
$E(\widetilde{\pi}_v,\chi_v,r)$ is non-zero, unless $c_v=0$ in which case $E(\widetilde{\pi}_v,\chi_v,r)=\left(1-\frac{q_v^{r-1}}{(\chi_v\nu_v)(\varpi_v)}\right)$. 
By local-global compatibility $\nu_v(\varpi_v)$ is a Weil number of weight $-\sw$, {\it i.e.}, an algebraic number whose absolute values are all equal to $q_v^{-\sw/2}$. Hence 
 $E(\widetilde{\pi}_v,\chi_v,r)$ vanishes precisely as stated.

If $\pi_v$ is principal series, then 
\[E(\widetilde{\pi}_v,\chi_v,r)=
(1-(\chi_v\nu_v)^{-1}(\varpi_v)q_v^{r-1} ) (1-(\chi_v\omega_\pi\nu_v^{-1})(\varpi_v)q_v^{1-r}),\]
where the first (resp. second) factor is dropped if $\chi_v\nu_v$ (resp. $\chi_v\omega_\pi\nu_v^{-1}$) is ramified. 
The Ramanujan conjecture for the cuspidal automorphic representation 
$\pi$, proven in \cite[Thm.1]{blasius:ramanujan}, implies that 
$\nu_v(\varpi_v)$ is a Weil number of weight $1-\sw$. Since $\omega_\pi$ has purity weight $\sw$, 
the first (resp. second) factor can only vanish for $r=\tfrac{3-\sw}{2}$ (resp. $r=\tfrac{3-\sw}{2}$) precisely as stated. 
\end{proof}

\subsection{Multi-variable $p$-adic $L$-functions}\label{ss:2var}
Since the construction  of $\cO(\cU)$-valued distributions over $\Gal_{p\infty}$ presented in \S\ref{ss:distributions} 
is functorial in the $L$-affinoid $\cU$, following the same steps as in \S\ref{ss:p-adic L-functions attached} would allow us to attach a $p$-adic $L$-function to a rigid analytic family containing $\widetilde{\pi}$. 
In addition to the cyclotomic variable, this function will have several weight variables. 
In conjunction with its improvements which will be constructed in the next subsection, this multi-variable 
 $p$-adic $L$-functions will allow us to prove the Trivial Zero Conjecture for $\pi$ at the central point in the second part of this paper. 
 
 Let $\widetilde{\pi}$ be a regular non-critical $p$-refinement of a cuspidal automorphic representation $\pi$ of $G(\A)$ of cohomological weight $(k,\sw)$. 
 Letting  $h= (h_{\widetilde{\pi}_v})_{v \in S_p}$ (see Def. \ref{d:non-critical-slope}), by 
 Theorem \ref{free-etale} there exists an $L$-affinoid neighborhood $\cU$ of $(k,\sw)$ in $\cX$ and a connected
component $\cV$ of $\Sp(\T_{\cU}^{\leqslant h})$ containing
$\widetilde{\pi}$ such that the weight map 
 $\kappa:\cV\xrightarrow{\sim}\cU$ is an isomorphism. By shrinking $\cU$ one can assume that $\sw_\lambda\circ\omega_p=\omega_p^{\sw}$ and 
\begin{align}\label{eq:alpha-u}
\alpha_{\gu}(\kappa^{-1}(\lambda))^2\omega_p^{-\sw}(\varpi_{\gu})\sw_\lambda^{-1}(\langle \varpi_{\gu} \rangle)\ne 
q_{\gu}^i,  \text{ for } i\in \{0,1,2\} \text{ and for all } \lambda\in \cU,
\end{align}
since the left hand side is an analytic function on $\lambda\in \cU$ and \eqref{eq:alpha-u} holds at $(k,\sw)$.

\begin{definition}\label{d:Phi-alpha}
Given a character $\epsilon: \{\pm 1\}^\Sigma\to \{\pm 1\}$ we let 
$\Phi_{\cU,\alpha_\gu }^\epsilon\in \rH^d_c(Y_K,\cD_{\cU})^{\leqslant h}$ be a basis of the free 
rank $1$ $\cO(\cU)$-module $\rH^d_c(Y_K, \cD_{\cU})^{\epsilon,\leqslant h}\otimes_{\T_{\cU}^{\leqslant h}}\cO(\cV)$
(see Thm. \ref{free-etale}(ii)) such that $(k,\sw)\circ \Phi^\epsilon_{\cU,\alpha_\gu }= \Phi_{\widetilde{\pi},\alpha_\gu }^\epsilon$. For $v\in S_p$ we let  $\alpha_v^\circ \in \cO(\cU)^{\times}$ denote the $U_{\varpi_v}$-eigenvalue on $\Phi^\epsilon_{\cU,\alpha_\gu}$.
\end{definition}

When $\lambda\in \cU(L)$ is cohomological, by Theorem \ref{free-etale}(iii) there exists a $p$-refined nearly finite slope cuspidal automorphic representation $\widetilde{\pi}_\lambda$ of weight $\lambda$ whose system of Hecke eigenvalues corresponds to $\kappa^{-1}(\lambda)\in \cV$.  Using the specialization map at such a cohomological $\lambda$, we obtain a class $\lambda\circ \Phi^\epsilon_{\cU,\alpha_\gu }\in \rH^d_c(Y_K, \cD_\lambda)^{\leqslant h}$ which generates the same line as the class $\Phi_{\widetilde{\pi}_\lambda,\alpha_\gu(\lambda) }^\epsilon$ from  \S\ref{ss:p-adic L-functions attached}. 
\begin{definition} \label{d:padic-period}
Let $C_{\lambda}^\epsilon \in L^{\times}$ be  such that $\lambda\circ \Phi^\epsilon_{\cU,\alpha_\gu }= C_{\lambda}^\epsilon \cdot \Phi_{\widetilde{\pi}_\lambda,\alpha_\gu(\lambda) }^\epsilon$. 
\end{definition}
Note that whereas $C_{(k,\sw)}^\epsilon=1$ by definition, $C_{\lambda}^\epsilon\in L^{\times} $ is a $p$-adic period analogous to those considered in \cite{greenberg-stevens}, and cannot in general be rescaled to be  $1$, since  the individual periods $\Omega_{\widetilde{\pi}_\lambda}^\epsilon$ for  $\lambda$ cohomological are well-defined up to $\overline{\Q}^\times$. 
Since $\alpha_\gu, \beta_{\gu}\in \cO(\cU)^\times$, one can analogously consider $\Phi^\epsilon_{\cU,\beta_\gu }$ and rescale it 
so that it yields the same $p$-adic periods $C_{\lambda}^\epsilon $. 

 The multi-variable $p$-adic $L$-function is defined as the distribution (see \eqref{eq:universal-evaluation}): 
\begin{align}\label{d:2var-L_p}
\cL_p= \sum_{\epsilon: \{\pm 1\}^\Sigma\to \{\pm 1\}}
\frac{ \alpha_{\gu} \ev(\Phi_{\cU,\alpha_{\gu }}^{\epsilon})
	-\beta_{\gu}	 \ev(\Phi_{\cU,\beta_{\gu }}^{\epsilon})}{\alpha_{\gu}-\beta_{\gu}}\in D(\Gal_{p\infty}, \cO(\cU)). 
\end{align}
For any finite order character $\chi: \Gal_{p\infty} \to L^\times$,  the natural projection $\Gal_{p\infty}\twoheadrightarrow \Gal_{\cyc}$
yields 
\begin{align}\label{d:2var-L_p-chi}
\cL_{p,\chi}= \frac{ \alpha_{\gu} \langle \ev(\Phi_{\cU,\alpha_{\gu }}^{\chi_\infty}), \chi \cdot \rangle
	-\beta_{\gu}	 \langle \ev(\Phi_{\cU,\beta_{\gu }}^{\chi_\infty}), \chi \cdot \rangle}{\alpha_{\gu}-\beta_{\gu}}\in D(\Gal_{\cyc}, \cO(\cU)). 
\end{align}

\begin{theorem} \label{thm:multi-L_p} Let $\widetilde{\pi}$ and $\cU$ be as above. 
 Fix a finite order character $\chi: \Gal_{p\infty} \to L^\times$. 
\begin{enumerate}
 \item  For any $\lambda\in \cU(L)$ cohomological, we have $\cL_{p,\chi}(\lambda) =C_{\lambda}^{\chi_\infty} \cL_{p}(\widetilde{\pi}_{\lambda}, \chi \cdot )$ in  $D(\Gal_{\cyc}, L)$.

\item $\cL_{p,\chi}$ has order of growth at most $\sum\limits_{v\in S_p}e_{v}h_{\widetilde{\pi}_v}$ and
is uniquely determined by its values 
\[\cL_{p,\chi}(\lambda, \chi' \chi_{\cyc}^{r-1})=\cL_p(\lambda,\chi \chi' \chi_{\cyc}^{r-1} ),\] 
at finite order characters $\chi'$ of $\Gal_{\cyc}$ and $r\in\Z$ critical for $\lambda$ cohomological (see \eqref{eq:interpolation-formula}). 
 \end{enumerate}
\end{theorem}

\begin{proof} (i) follows from the definition and the functoriality of $\ev$.

(ii) For  $\lambda\in \cU$ cohomological such that $\widetilde{\pi}_\lambda$ has very non-critical slope (see \eqref{eq:very-non-critical-slope}), the distribution  $\cL_{p,\chi}(\lambda)$ on $\Gal_{\cyc}$ 
has order of growth strictly less than the number $ \min\limits_{\sigma\in \Sigma} (k_{\lambda, \sigma}-1)$ of critical integers for $\lambda$. 
A well-known result of Vishik \cite[Thm.2.3, Lem.2.10]{vishik}, proven independently by Amice-V\'elu, implies that $\cL_{p,\chi}(\lambda)$ 
is uniquely determined by its values at $\chi' \chi_{\cyc}^{r-1}$, where 
 $\chi'$ is a  finite order character of $\Gal_{\cyc}$ and $r\in\Z$ is critical for $\lambda$. 
The claim is deduced by noticing that  such $\lambda$ form a very Zariski dense subset of $\cU$. 
\end{proof}

\begin{remark}
When Leopoldt's conjecture holds for $F$ at $p$, the kernel of the natural projection $\Gal_{p\infty}\twoheadrightarrow \Gal_{\cyc}$
is a finite abelian group and $\cL_{p}$ is merely a collection of $\cL_{p,\chi}$ with $\chi$ running over the characters of that group. 
 \end{remark}

\begin{remark} \label{rem:BH}
When $\widetilde{\pi}$ is non-critical, but has critical slope, then interpolation formula \eqref{eq:interpolation-formula}  
does not suffice to determine $\cL_p(\widetilde{\pi})$ uniquely, and we are indebted to J.~Bella{\"\i}che for having explained to one of us 
how the smoothness of the eigenvariety can be used to palliate this indeterminacy. 
When $\pi$ is Iwahori spherical at all  places above $p$, a  similar approach has also been successfully used by Bergdall and Hansen \cite{bergdall-hansen} who construct $\cL_p$ for regular $\widetilde{\pi}$ which are  either non-critical, or 
such that $\rH^\bullet_c(Y_K, \cD_{k,\sw})_{\gm_{\widetilde{\pi}}}$ is concentrated in degree $d$ and the 
adjoint Bloch-Kato Selmer group $\rH^1_f(F, \mathrm{ad}(V_{\pi}))$  vanishes.
 \end{remark}

It will be essential in \S\ref{exc-padicL} to control $\cL_{p,\chi}$ under simultaneous
variation in $\sw_\lambda$ and the cyclotomic variable. This, however, can only be achieved after a renormalization of the $p$-adic periods, whose variations in $\cU$ are {\it a priori}   well-defined up to an invertible analytic function. This is equivalent to rescaling the basis $\Phi_{\cU, \alpha_{\gu}}^\epsilon$ by an invertible element of  $\mathcal{O}(\cU)$.

\begin{proposition}\label{p:renormalization-padic-periods} For  any finite order character $\chi: \Gal_{p\infty} \to L^\times$, 
the $p$-adic periods $C_\lambda^{\chi_\infty}$ of Def.\ref{d:padic-period} can be renormalized so that for  $z\in \cO_{\C_p}$ and $\lambda \in \cU$ such that
$\lambda(z)=(k_\lambda,\sw_\lambda \langle \cdot\rangle^{2z})\in \cU$,  one has
\begin{align}\label{eq:twisting-bis}
\cL_{p,\chi}(\lambda(z))=\cL_{p,\chi}(\lambda,\langle\cdot \rangle_p^{z}\cdot). 
\end{align}  
\end{proposition}

\begin{proof} By analyticity it  suffices to check \eqref{eq:twisting-bis}  for $z\in \Z$
and for cohomological $\lambda\in\cU$ such that $\widetilde\pi_{\lambda}$ has very non-critical slope, as such pairs $(\lambda,z)$ are very Zariski dense
in $\cU\times  \cO_{\C_p}$.

By Theorem \ref{free-etale} the weight map $\kappa:\cV \to \cU$ is etale at $\lambda$ and so, by the
assumption on $z$, if $\kappa^{-1}(\lambda) =
\widetilde{\pi}_\lambda$ then
$\kappa^{-1}(\lambda(z))=\widetilde{\pi_\lambda\otimes|\cdot|^z}$.
Noting
that an integer $r$ is critical for $\pi$ if
and only if $r-z$ is critical for $\pi\otimes|\cdot|^{z}$, Propositions \ref{p:twisting-periods} and \ref{t:interpolation-one-form}
imply the following equality 
\[\cL_p(\widetilde{\pi_\lambda \otimes|\cdot|^z},
\langle\cdot \rangle_p^{-z}\chi \cdot )=\cL_p(\widetilde{\pi}_\lambda, \chi \cdot) \text{ in }   D(\Gal_{\cyc}, L). \]
As a result, $C_{\lambda(z)}^{\chi_\infty}=C_\lambda^{\chi_\infty} 
\cL_{p,\chi}(\lambda(z), \langle\cdot \rangle_p^{z}\cdot)/\cL_{p,\chi}(\lambda)$ for  $\lambda\in\cU$ cohomological and for 
$z\in \Z$ sufficiently small ($p$-adically). Letting $\sw_\lambda= \langle \cdot\rangle^{2z}\sw$, the 
 function  $\lambda\mapsto \cL_{p,\chi}(\lambda, \langle\cdot \rangle_p^{z}\cdot)/\cL_{p,\chi}((k_\lambda,\sw))$ 
belongs to $\mathcal{O}(\cU)^\times$, for $\cU$ a sufficiently small,  and interpolates 
$C_{\lambda}^{\chi_\infty}/C_{(k_\lambda,\sw)}^{\chi_\infty}$ for $\lambda$ cohomological. 
We may therefore renormalize the periods $C_\lambda^{\chi_\infty}$ to guarantee \eqref{eq:twisting-bis}.
\end{proof}

\begin{remark} 
Exactly as in  Prop. \ref{p:renormalization-padic-periods},  for any finite order characters $\chi$ of $\Gal_{p\infty}$ one has
\begin{align}\label{eq:twisting-Lp}
\cL_{p}(\widetilde{\pi\otimes\chi}, \cdot)= \cL_{p} (\widetilde{\pi}, \chi\cdot ) \text{ in }  D(\Gal_{\cyc}, L).
\end{align}
 
If there exists a finite order character $\chi$ of $\Gal_{p\infty}$ such that 
$\chi_{|\cO_v^\times}=\nu_v^{-1}{}_{|\cO_v^\times}$ for all $v\in S_p$ , then 
$\pi\otimes\chi$ has finite slope. Such a character always exists when $F=\Q$ allowing one to reduce to the finite slope case. For general $F$, such a character may have auxiliary ramification forcing the 
twisted finite slope form to have a different tame level from the original one. 
 \end{remark}

\subsection{Improved $p$-adic $L$-functions} \label{ss:improved}
  When $\cL_p(\widetilde{\pi})$ 
 has a trivial zero at a critical integer $r$ the interpolation formula for $\cL_p(\widetilde{\pi},\chi\chi^{r}_{\cyc})$ 
 `misses' the special $L$-value $L(\pi\otimes\chi,r-\tfrac{1}{2})$.
 An idea due to Greenberg and Stevens \cite{greenberg-stevens} is to then construct a so-called improved $p$-adic $L$-function having only weight variables and interpolating, with non-vanishing extra factors, the critical  $L$-value.
 In order to retrieve the `missed' $L$-value even in the case where several local factors $E(\widetilde{\pi}_v,\chi_v,r)$ simultaneously vanish we will  construct for any  $S\subset S_p$ a rigid analytic function $L_{S}(\lambda, \chi, r)$ over an  $(|\Sigma_S|+1)$-dimensional affinoid 
 $\cU'_{S}= \cX'_S \cap \cU$, where $\cU$ is as in \S\ref{ss:2var}. 

 Consider the subset $ S \subset S_p$ containing all places $v\in S_p$ such that $\nu_v$ is ramified (note that this is always the case if 
 $\pi$ has finite slope). 
For each character $\epsilon: \{\pm 1\}^\Sigma\to \{\pm 1\}$ let $\Phi^\epsilon_{\cU,\alpha_\gu}$ be as in \S \ref{ss:2var} and denote by $\alpha^{\circ}_{v}$ its $U_{\varpi_v}$-eigenvalue. Denote by $\Phi^\epsilon_{\cU'_{S} ,\alpha_\gu }$ the image of $\Phi^\epsilon_{\cU,\alpha_\gu }$ in $\rH^d_c(Y_K, \cD_{\cU'_{S} })^{\leqslant h}$. 
By Def. \ref{defn:XKS} and \eqref{eq:alphas}, for all $v \in S_p\!\setminus S $ we have
\[\alpha_v(\lambda) = \alpha_v^\circ(\lambda)\prod\limits_{\sigma\in
 \Sigma_v}\sigma(\varpi_v)^{\frac{2-\sw-k_\sigma}{2}} \in
\cO(\cU'_{S})^{\times}.\]

 Observe that for any $v \in S_p\!\setminus S $, the rigid 	analytic function $\left(1-\frac{q_v^{r-1}}{\chi_v(\varpi_v)\alpha_v(\cdot)}\right)\in \cO(\cU'_{S})$ specializes at $\lambda$ cohomological to the 
 interpolation factor $ \left(1-\frac{q_v^{r-1}}{(\chi_v\nu_{\lambda,v})(\varpi_v)}\right)$ from Theorem \ref{t:interpolation-one-form}. 
Our aim is to show that the meromorphic  quotient  $\cL_p(\lambda,\chi\chi^{r}_{\cyc}) \prod\limits_{S_p\!\setminus S} 
\left(1-\frac{q_v^{r-1}}{\chi_v(\varpi_v)\alpha_v(\cdot)}\right)$ is in fact analytic and to compute its value at $(k,\sw)$. 
Achieving this requires to take a step back and define the improved $p$-adic $L$-functions using the tools developed in \S \ref{ss:norms}. 

\begin{definition} For $\chi:\Gal_{S\infty}\to L^\times$ a finite order character, we define the  improved $p$-adic $L$-function as
	\[L_{S}(\cdot, \chi, r) =\frac{\alpha_{\gu}\langle \ev_{S}^{r}(\Phi^{\chi_\infty\omega_{p,\infty}^{r-1}}_{\cU'_{S} ,\alpha_\gu }), \chi \rangle-\beta_{\gu}\langle \ev_{S}^{r}(\Phi^{\chi_\infty\omega_{p,\infty}^{r-1}}_{\cU'_{S} ,\beta_\gu }), \chi \rangle}{\alpha_\gu-\beta_\gu}\in \cO(\cU'_{S}).\]
\end{definition}

It follows from \eqref{eq:improved-evK-relation} that $L_{S_p}(\lambda, \chi, r)=\cL_p(\lambda,\chi \chi^{r-1}_{\cyc})$.

\begin{theorem}\label{t:improved-Lp}
	\begin{enumerate}
		\item For $v\in S$, we have the following equality in $\cO(\cU'_{S\!\setminus\{v\}}) $:
	\[\vartheta_{S,v}\left(L_{S}(\cdot, \chi, r)\right) = L_{S\!\setminus\{v\}}(\cdot, \chi, r)\left(1-\frac{q_v^{r-1}}{\alpha_v(\cdot)\chi(\varpi_v)}\right).\]
		\item For any cohomological $\lambda\in \cU'_S(L)$ and any integer $r$ critical for $\lambda$, we have
\[L_{S}(\lambda, \chi , r) = 
	\frac{\mathrm{N}_{F/\Q}^{r-1}(i\gd)\chi(\varpi_\gd^{-1})}{\Omega_{\widetilde{\pi}_{\lambda}}^{\epsilon}} 
	\cdot C_{\lambda}^{\epsilon}\cdot L(\pi_{\lambda}\otimes\chi,r-\tfrac{1}{2}) \prod_{v\in S} E(\widetilde{\pi}_v,\chi_v,r)
	\prod_{\ontop{v\in S_p\!\setminus S}{ \pi_v \text{unram.}}} \left(1-\frac{(\chi_v\omega_\pi\nu_v^{-1})(\varpi_v)}{q_v^{r-1}}\right), 
\]
	where $\epsilon=\chi_\infty\omega_{p,\infty}^{r-1}$ and  $C_{\lambda}^{\epsilon}$ is the $p$-adic period introduced in Def. \ref{d:padic-period}. 	
    \end{enumerate}
\end{theorem}

\begin{proof} (i) It suffices to apply  Corollary \ref{c:ev-evtilde} to $\Phi^{\epsilon}_{\cU'_{S} ,\alpha_\gu}$ and to 
$ \Phi^{\epsilon}_{\cU'_{S} ,\beta_\gu}$. 
 
 (ii) Let $\gf= \prod_{v\in S}v^{n_v}$ be such that $n_v\ge \mathrm{max}(c_v, 1)$. Using the definition of $\ev_{S}^{r}$ we obtain:
\[ \langle \ev_{S}^{r}(\Phi^{\epsilon}_{\cU'_{S} ,\alpha_\gu }), \chi \rangle (\lambda)=  (\alpha_{\gf}^{\circ}(\lambda))^{-1} \sum_{[\eta]\in\Cl(\gf)}\chi(\eta)
\chi^{r-1}_{\cyc} (\eta) 
\langle \ev_{\varpi_{\gf}}^\eta(\Phi^{\epsilon}_{\cU'_{S} ,\alpha_\gu }), 1^{\times}_{S, r} \rangle (\lambda) \]

Letting $j_{\lambda}= (r-1+ \frac{\sw_{\lambda}- 2+ k_{\lambda, \sigma}}{2})_{\sigma \in \Sigma}$, we remark that $\lambda\circ 1^{\times}_{S, r}$ and $z^{j_{\lambda}}$ agree on the support of 
$\ev_{\varpi_{\gf}}^\eta(\Phi_{\widetilde{\pi}_{\lambda},\alpha_\gu}^{\epsilon})$. Together with the definition of $C_\lambda^{\epsilon}$ this implies that:
\[ \langle \ev_{S}^{r}(\Phi^{\epsilon}_{\cU'_{S} ,\alpha_\gu }), \chi \rangle (\lambda)=  C_\lambda^{\epsilon}(\alpha_{\gf}^{\circ}(\lambda))^{-1} \sum_{[\eta]\in\Cl(\gf)}\chi(\eta)
\chi^{r-1}_{\cyc} (\eta) 
\langle	\ev_{\varpi_{\gf}}^\eta(\Phi_{\widetilde{\pi}_{\lambda},\alpha_\gu(\lambda)}^{\epsilon}),z^{j_{\lambda}}\rangle.\]
The remainder of the proof follows from Proposition \ref{p:ev-critical} as in the proof of Theorem \ref{t:interpolation-one-form}.
\end{proof}

\subsection{Partial $p$-adic $L$-functions} \label{ss:partial}

We will now explain how the construction of the previous subsection can be adapted to the partial families constructed in Theorem \ref{free-etale}. 
We are indebted to the referees for their insight and encouragement to present this construction. 

Henceforth we fix $S\subsetneq S_p$ and we suppose we are given a regular non-critical $S$-refinement 
 $\widetilde{\pi}_S= (\pi, \{\nu_v\}_{v\in S})$ of $\pi$ (see Def. \ref{d:hilbert-nearly} and Def. \ref{d:non-critical}). 
 We consider  the neat open compact subgroup $K=K(\widetilde{\pi}_S,\gu)\subset G(\A_f)$ from Def. \ref{d:newline}. By non-criticality of $\widetilde{\pi}_S$, for each character $\epsilon: \{\pm 1\}^\Sigma \to \{\pm 1\}$, 
the basis $\iota_p(b_{\widetilde{\pi}_S,\alpha_\gu}^{\epsilon})$
of $\rH^d_c(Y_K, \cL_{k,\sw}^\vee(L))^\epsilon_{\gm_{\widetilde{\pi}_S}}$ (see Def. \ref{of-the-period-omega-pi})
 lifts canonically to a basis of  $\Phi_{\widetilde{\pi}_S,\alpha_\gu}^{\epsilon}$ of 
$\rH_c^d(Y_K,\cD_{S, (k,\sw)})^\epsilon_{\gm_{\widetilde{\pi}_S}}$. 
 Letting  $h_S= (h_{\widetilde{\pi}_v})_{v \in S}$,  by 
 Theorem \ref{free-etale} there exists an $L$-affinoid neighborhood $\cU_S$ of $(k,\sw)$ in $\cX_S$ and a connected
component $\cV_S$ of $\Sp(\T_{S,\cU_S}^{\leqslant h_S})$ containing
$\widetilde{\pi}_S$ such that the weight map induces an
isomorphism $\kappa:\cV_S\xrightarrow{\sim}\cU_S$. 
We can choose $\cU_S$ sufficiently small so that it satisfies the technical properties stated in \S\ref{ss:2var}.

\begin{definition}
Fix  a basis $\Phi_{\cU_S,\alpha_\gu }^\epsilon$ of the free 
$\cO(\cU_S)$-module $\rH^d_c(Y_K, \cD_{S,\cU_S})^{\epsilon,\leqslant h_S}\otimes_{\T_{S,\cU_S}^{\leqslant h_S}}\cO(\cV_S)$ of rank $1$, 
 such that $(k,\sw)\circ \Phi^\epsilon_{\cU_S,\alpha_\gu }= \Phi_{\widetilde{\pi}_S,\alpha_\gu }^\epsilon$. For $v\in S$ we let  $\alpha_v^\circ \in \cO(\cU_S)^{\times}$ denote the $U_{\varpi_v}$-eigenvalue on $\Phi^\epsilon_{\cU_S,\alpha_\gu}$.
\end{definition}

Given an integer $r$ which is $S$-critical for $(k, \sw)$ (see Def. \ref{def:crit}) and an ideal $\gf\subset \cO_F$ supported at primes in $S$,  using Rem. \ref{rem:crit-range}(ii) yields a map $(\cD_{S,\cU_{S}})_{E(\gf)}\to D(U(\gf)_{S}/E(\gf), \cO(\cU_{S}))$. The construction performed in \S\ref{ss:norms} applies {\it mutatis mutandis} and yields a well-defined map: 
\[\ev_{\varpi_\gf, S}^{r}: \rH^d_c(Y_K, \cD_{S,\cU_S}) \rightarrow D(\mathrm{Gal}_{S\infty}, \cO(\cU_S))\]

For $\chi:\Gal_{S\infty}\to L^\times$ a finite order character, we define the partial $p$-adic $L$-function
	\[L^{S}(\cdot, \chi, r) =\frac{\alpha_{\gu}\langle \ev_{S}^{r}(\Phi^{\chi_\infty\omega_{p,\infty}^{r-1}}_{\cU_{S} ,\alpha_\gu }), 
	\chi \rangle-\beta_{\gu}\langle \ev_{S}^{r}(\Phi^{\chi_\infty\omega_{p,\infty}^{r-1}}_{\cU_{S} ,\beta_\gu }), 
	\chi \rangle}{\alpha_\gu-\beta_\gu}\in \cO(\cU_{S}).\]

 For any cohomological $\lambda\in \cU_S(L)$, Theorem \ref{free-etale}(iii) yields an $S$-refined cuspidal automorphic representation $\widetilde{\pi}_{\lambda,S}$ of weight $\lambda$ whose system of Hecke eigenvalues corresponds to $\kappa^{-1}(\lambda)\in \cV_S$.  
 Given any character $\epsilon: \{\pm 1\}^\Sigma \to \{\pm 1\}$, 
 the specialization $\lambda\circ \Phi^\epsilon_{\cU_S,\alpha_\gu }$ of $\Phi^\epsilon_{\cU_S,\alpha_\gu }$ at $\lambda$ generates the same line 
 in $\rH^d_c(Y_K, \cD_{S, \lambda})^{\epsilon, \leqslant h_S}$ as $\Phi_{\widetilde{\pi}_{\lambda,S},\alpha_\gu(\lambda) }^\epsilon$, hence there exists 
\[C_{\lambda,S}^\epsilon \in L^{\times}\text{ such that } \lambda\circ \Phi^\epsilon_{\cU_S,\alpha_\gu }= C_{\lambda,S}^\epsilon \cdot \Phi_{\widetilde{\pi}_{\lambda,S},\alpha_\gu(\lambda) }^\epsilon.\]

\begin{theorem}\label{t:partial} Given $\chi$ a finite order character of $\Gal_{S \infty}$, and given a
 cohomological $\lambda\in \cU_S(L)$ for which $r$ is critical, we have
						\[L^{S}(\lambda, \chi , r) = 
	\frac{\mathrm{N}_{F/\Q}^{r-1}(i\gd)\chi(\varpi_\gd^{-1})}{\Omega_{\widetilde{\pi}_{\lambda, S}}^{\chi_\infty\omega_{p,\infty}^{r-1}}} 
	C_{\lambda,S}^{\chi_\infty \omega_{p,\infty}^{r-1}}L(\pi_{\lambda}\otimes\chi,r-\tfrac{1}{2}) \prod_{v\in S} E(\widetilde{\pi}_v,\chi_v,r).
	\]
\end{theorem} 
 \begin{proof} The proof is based on Proposition \ref{p:ev-critical}
  in the same way as the proof of Theorem \ref{t:interpolation-one-form}. 
\end{proof}

One may also define partial $p$-adic $L$-functions which are improved at places lying in a  subset of  $S$. We leave the  details of the construction to the interested reader.

\part{The Trivial Zero Conjecture at the central critical point}

Throughout this part $\widetilde{\pi}= (\pi, \{\nu_v\}_{v\in S_p})$ will be a
regular non-critical $p$-refinement of a cuspidal automorphic representation $\pi$ of $G(\A)$ of cohomological weight $(k,\sw)$ 
and tame conductor $\gn$ satisfying the following assumption:
\begin{align}\label{A12} \!
 \pi \text{ has central character } \omega_\pi\!=\!|\cdot|_F^{\sw} \text{ with } \sw \text{ even, and } \pi_v 
\text{ is Iwahori spherical for all } v\!\in\! S_p. 
\end{align}

It follows that $\nu_v$ is unramified for all $v\in S_p$ and we let $\alpha_v= \nu_v(\varpi_v)$. 
The set $S_p$ is then partitioned into $\St_p$ consisting of $v$ such that $\pi_v$ is an unramified
	twist (by $\nu_v$) of the unitary Steinberg representation and its complement $S_p\!\setminus \St_p$ consisting of $v$ such that $\pi_v$ is unramified.

\section{Galois representations and arithmetic $\mathscr{L}$-invariants}\label{s:p-adic-hodge}
The Trivial Zero Conjecture stated in the introduction posits a relationship between the special values of $p$-adic $L$-functions
and arithmetic $\mathscr{L}$-invariants. In this section we turn to the Galois representation side of $p$-adic
families and explain the connection between two types of arithmetic $\mathscr{L}$-invariants 
 associated to Hilbert modular forms. Moreover we will express them in terms of 
derivatives of Hecke eigenvalues, which will allow to relate them to $p$-adic $L$-functions in the last section.

\subsection{Galois representations for Hilbert modular forms}\label{ss:gal HMF}
By Theorem \ref{free-etale} there exists an $L$-affinoid neighborhood $\cU $ of $(k,\sw)$ in $\cX$ and $\kappa:\cV\xrightarrow{\sim}\cU$ a family containing $\widetilde{\pi}$ such that for any  $\lambda\in \cU$ cohomological, $\kappa^{-1}(\lambda)\in \cV$ corresponds to 
$p$-refined cuspidal automorphic representation $\widetilde{\pi}_\lambda$ of weight $\lambda$. The 
cohomological points being Zariski dense in $\cU$, there exists a unique two-dimensional pseudo-character $\mathrm{G}_F\to \cO(\cU)$ interpolating the traces of the corresponding $p$-adic Galois representations $ \rho_{\pi_\lambda}: \mathrm{G}_F\to \Aut_L(V_{\pi_\lambda})$ attached to $\pi_\lambda$. Since $V_{\pi}$ is absolutely irreducible, using a result of Nyssen \cite{nyssen} and Rouquier \cite{rouquier} one shows that, after possibly further shrinking $\cU$, there exists a continuous Galois representation
\begin{align}\label{eq:rhoU}
\rho_{\cU}:\mathrm{G}_F\to
\GL_2(\cO(\cU))
\end{align}
whose specialization at every cohomological $\lambda\in \cU$ is isomorphic to $V_{\pi_\lambda}$. 
Further shrinking $\cU$ one can assume that the map $\lambda\mapsto (k_\lambda,\sw_\lambda)$ defined in 
\eqref{eq:weight-finite-map} is injective on $\cU$, that $\sw_\lambda\circ\omega_p=\omega_p^{\sw}$, and that \eqref{eq:alpha-u}
holds. 
When $\cU$ satisfies all these assumptions and, in addition, the tame conductor of $\pi_\lambda$ equals $\gn$ for every cohomological $\lambda\in \cU$
(see Lem. \ref{l:tame}), we denote:
\begin{align}\label{eq:Xpi}
\cX(\widetilde{\pi})=\cU.
\end{align}

\begin{lemma}\label{l:tame} The tame conductor of 
$\pi_\lambda$ is $\gn$ for all cohomological $\lambda$ sufficiently close to $(k,\sw)$. 
\end{lemma}
\begin{proof}
 By construction, for all cohomological $\lambda\in \cU$ the tame level $\gn_\lambda$ of $\pi_\lambda$ divides $\gn\gu$. 
 We will use the Galois representation $\rho_{\cU}$ to show
 that, after shrinking $\cU$, one has $\gn_\lambda=\gn$. The local-global compatibility at a  finite place $v\notin S_p$, established by Carayol \cite{carayol:hmf-l-adic} and Taylor \cite{taylor},  asserts that the Frobenius semi-simplification of the Weil-Deligne representation $(r_{v,\lambda},N_{v,\lambda})$ attached to 
$\rho_{\pi_\lambda|\mathrm{G}_{F_v}}$ corresponds, via the local Langlands correspondence, to  $\pi_{\lambda, v}\otimes|\cdot|^{-1/2}_{v}$. 
On the other hand, by \cite[Lem.7.8.14]{bellaiche-chenevier}, one can attach to $\rho_{\cU|\mathrm{G}_{F_v}}$ a Weil-Deligne representation $(r_{v,\cU},N_{v,\cU})$.  By \cite[Lem.7.8.17]{bellaiche-chenevier}, after possibly shrinking $\cU$, the restriction of $r_{v,\cU}$ to the inertia subgroup at $v$ has finite image whose  specialization at any cohomological $\lambda\in \cU$ 
 is isomorphic to the restriction of $r_{v,\lambda}$ to the same inertia subgroup. 
Therefore, to conclude that $\gn_\lambda=\gn$ it suffices to show that $N_{v,\lambda} = N_{v,\cU}$. From \cite[Prop.7.8.19]{bellaiche-chenevier} it follows that $N_{v,\lambda}$ lies in the $p$-adic closure of the conjugacy class of $N_{v,\cU}$, which, together with the fact that $\gn_\lambda$ divides $\gn \gu$ implies equality except possibly when $v=\gu$. Finally $\pi_{\lambda,\gu}$ is unramified for all cohomological $\lambda\in \cU$, as being an unramified twist of the Steinberg representation is excluded by \eqref{eq:alpha-u}. 
\end{proof}

To compute derivatives of analytic functions on $\cX(\widetilde{\pi})$ in the following sections, we will consider a subset
of $\cX(\widetilde{\pi})$ which can be parametrized with the  variables $((k_{\lambda,\sigma})_{\sigma\in \Sigma},\sw_\lambda)$ 
corresponding via \eqref{eq:weight-finite-map} to  characters of the form 
\begin{align}\label{eq:weight-parametrization}
\prod_{v\in S_p} \cO_v^\times\times \Z_p^\times \to \C_p^\times,  z=((z_v)_{v\in S_p}, z_0)\mapsto 
(k,\sw)(z)\cdot   \langle z_0\rangle_p^{\sw_\lambda-\sw} \prod_{v\in S_p} \prod_{\sigma\in \Sigma_v} \sigma(\langle
 z_v\rangle_v)^{k_{\lambda,\sigma}-k_\sigma}, 
\end{align}
 where $\langle\cdot\rangle_v:\mathcal{O}_v^\times\to 1+(\varpi_v)$ is the natural projection map. This allows us to parametrize 
 $\cX(\widetilde{\pi})$ by  a neighborhood, denoted  $\cX^{\an}(\widetilde{\pi})$,  of $(k,\sw)$ in the  space 
 $\prod_{\sigma\in \Sigma} (k_\sigma+2p \mathcal{O}_{\mathbb{C}_p})\times (\sw+\mathcal{O}_{\mathbb{C}_p})$ of analytic weights.
 If we impose the weights to vary only  in parallel direction per place above $p$, {\it i.e.}, if $k_{\lambda,\sigma}-k_\sigma=x_v\in \mathcal{O}_{\mathbb{C}_p}$ for all 
$\sigma\in \Sigma_v$, then    \eqref{eq:weight-parametrization} becomes  
\begin{align}\label{eq:weight-parametrization-bis}
\prod_{v\in S_p} \cO_v^\times\times \Z_p^\times \to \C_p^\times, \quad z=((z_v)_{v\in S_p}, z_0)\mapsto 
(k,\sw)(z)\cdot   \langle z_0\rangle_p^{\sw_\lambda-\sw} \prod_{v\in S_p} \langle
 \mathrm{N}_{F_v/\Q_p}(z_v)\rangle_p^{x_v}. 
\end{align}

\subsection{Fontaine-Mazur $\mathscr{L}$-invariants}\label{ss:FM}
Consider the $2$-dimensional $L$-vector space $V=V_{\pi}(\tfrac{2-\sw}{2})$
 endowed with continuous action of $\mathrm{G}_F$. 
The local-global compatibility is also known at places $v\in S_p$ from the 
work of Saito \cite{saito:hmf-hodge}, Blasius-Rogawski \cite{blasius-rogawski} and Skinner \cite{skinner:hmf-hodge}. 
Letting $V_v=V_{| \mathrm{G}_{F_v}}$, $\cD_{\st}(V_v)=(V_v\otimes_{\Q_p}\mathrm{B}_{\st})^{\mathrm{G}_{F_v}}$
 is a free $L\otimes F_{v,0}$-module of rank $2$ ($F_{v,0}$ is the maximal unramified subfield of $F_v$) carrying a semilinear Frobenius  $\varphi_v$ and a nilpotent linear map $N_v$ both inherited from Fontaine's ring $\mathrm{B}_{\st}$ and such that 
 $N_v\circ \varphi_v = p \varphi_v\circ N_v$.  
 To be more precise, the linear map $\varphi_v^{f_v}$, where
$f_v$ is the inertial index at $v$, has eigenvalues
$q_v^{(\sw-2)/2}\alpha_v$ and $q_v^{-\sw/2} \alpha_v^{-1}$, and
the monodromy matrix vanishes if and only if $\pi_v$ is
unramified.
 Moreover the $L\otimes _{\Q_p} F_{v}$-module $\cD_{\st}(V_v)\otimes_{F_{v,0}} F_v$   is endowed with a decreasing  de Rham filtration $\Fil^\bullet (\cD_{\st}(V_v)\otimes_{F_{v,0}} F_v)$   whose jumps, called the labeled Hodge-Tate weights, 
  occur precisely at the integers $\left(\tfrac{k_\sigma-2}{2},\tfrac{-k_\sigma}{2}\right)_{\sigma\in \Sigma_v}$
(we recall that we are using the convention in which  the cyclotomic character has Hodge-Tate weight $-1$).
  In particular  $\Fil^0(\cD_{\st}(V_v)\otimes_{F_{v,0}} F_v)$ is free of rank $1$ over $L \otimes_{\Q_p} F_v$.

In the context when $\pi_v\otimes|\cdot|^{-\sw/2}$ is the Steinberg representation, there is a unique (up to scalar) basis $(e_1,e_2)$ of 
$\mathcal{D}_{\st}(V_v)$ such that  $\varphi_v^{f_v}(e_1)=q_v^{(\sw-2)/2}\alpha_v\cdot e_1$,  
$\varphi_v^{f_v}(e_2)=q_v^{-\sw/2} \alpha_v^{-1}\cdot e_2$ and $N_v(e_2)=e_1$. The Fontaine-Mazur $\mathscr{L}$-invariant is the 
 unique $\mathscr{L}_{\FM}(V_v)\in L\otimes_{\Q_p}F_{v}$ such that 
\[\Fil^0(\cD_{\st}(V_v)\otimes_{F_{v,0}} F_v)=(L\otimes_{\Q_p}F_v)\cdot (e_1+\mathscr{L}_{\FM}(V_v)\cdot  e_2).\]

We next relate $\mathscr{L}_{\FM}(V_v)$ to derivatives of Hecke
eigenvalues. We will consider the
$U_{\varpi_v}^\circ$-eigenvalues $\alpha_v^\circ\in
\cO(\cX(\widetilde\pi))$  (see Def. \ref{d:Phi-alpha}) 
as functions $\alpha_v^\circ((k_{\lambda,\sigma})_{\sigma\in \Sigma},\sw_\lambda)$ by restriction to $\cX^{\an}(\widetilde{\pi})$, and we will denote by $\dlog_u\alpha_v^\circ$ denote the logarithmic derivative 
at $(k,\sw)$ in the direction $u=((u_\sigma)_{\sigma\in \Sigma},u_0)$, {\it i.e.},  
$\dlog_u\alpha_v^\circ= \frac{1}{\alpha_v^\circ(k,\sw)}\frac{d}{dx} \alpha_v^\circ((k,\sw)+ x \cdot u)|_{x=0}$.

 \begin{proposition}\label{p:FM-derivatives}
If $u_\sigma=1=-u_0$ for all $\sigma\in \Sigma_v$, then $e_v^{-1}\cdot \Tr_{F_v/\Q_p}(\mathscr{L}_{\FM}(V_v))=-2\dlog_u \alpha_v^\circ$, where $e_v$ is the ramification index at $v$.
\end{proposition}
\begin{proof}
The line defined by the direction $u$ lies in the space $\cX_{S_p\!\setminus \{v\}}'$, and we denote by $\cU'$ the portion of this line 
inside the ball $\cX^{\an}(\widetilde{\pi})$. We will write $\rho_{\cU'}$ for the analytic Galois representation on
$\cX^{\an}(\widetilde{\pi})$ along $\cU'$, and note that on this line $\alpha_v=\alpha_v^\circ\prod_{\sigma\in
 \Sigma_v}\sigma(\varpi_v)^{(2-k_\sigma-\sw)/2}$ is an analytic function. By \cite{kisin:overconvergent},
$\cD_{\st}(\rho_{\cU'})^{\varphi_v^{f_v}=\alpha_v}$ is an $\cO(\cU') \otimes_{\mathbb{Q}_p}F_{v,0}$-module of rank $1$  and
therefore we may apply \cite[Thm.1.1]{zhang:FM}. Since $\det\rho_{\cU'}=\chi_{\cyc}^{\sw_\lambda-1}=\chi_{\cyc}^{\sw-1-x u_0}$ we have the following equality of differentials at $x=0$:
\[\frac{d \alpha_v((k,\sw)+x u)}{\alpha_v(k,\sw)} =
\frac{1}{2e_v}
\Tr_{F_v/\mathbb{Q}_p}(\mathscr{L}_{\FM}(V_v) d (x u_0)),\]
which immediately implies the formula, as $u_0=-1$. 
\end{proof}

The fact that  $\dlog_u \alpha_v^\circ$ does not
depend on $u$, as long as $u_\sigma=1=-u_0$ for $\sigma\in \Sigma_v$, 
 will be used in the final section.
 
\begin{definition}\label{d:L-inv}
We define $\mathscr{L}(\widetilde{\pi}) = \prod\limits_{v\in E}e_v^{-1} \Tr_{F_v/\Q_p}(\mathscr{L}_{\FM}(V_v))$, where 
$E\subset \St_p$ consists of those places $v$ for which $\pi_v\otimes|\cdot|^{-\sw/2}$ is the Steinberg representation.
\end{definition}

\subsection{Greenberg-Benois $\mathscr{L}$-invariants}\label{ss:greenberg-benois}
The connection between the analytic Galois representation and the $p$-adic family runs deeper than the above description, and in order to state this connection precisely, we introduce the
 category of $(\varphi,\Gamma)$-modules over Robba rings (we refer to \cite[\S1]{benois:L-invariant} for more details). 
For $v\in S_p$ the Robba ring $\mathcal{R}_{F_v}$ is an
$F_v$-algebra endowed with a continuous map $\varphi_v$ and a
continuous action of $\Gamma_v =
\Gal(F_v(\mu_{p^\infty})/F_v)$. Writing $\cU= \cX(\widetilde{\pi})$, a $(\varphi_v,\Gamma_v)$-module over $\mathcal{R}_{F_v,\cU}=\mathcal{R}_{F_v}\widehat{\otimes}_{\Q_p}\mathcal{O}(\cU)$ is a coherent,
locally free sheaf $D$ over $\cR_{F_v,\cU}$ of finite rank endowed with a $\varphi_v$-semilinear map $\varphi_{D}$ which gives an isomorphism $\varphi_D^*D\cong D$,
and a semilinear action of $\Gamma_v$ commuting with $\varphi_{D}$. Furthermore, there exists a functor $\cD_{\rig}^\dagger$ 
associating to a continuous $\cO(\cU)$-linear representation of $\mathrm{G}_{F_v}$ a $(\varphi_v,\Gamma_v)$-module over $\cR_{F_v,\cU}$. 
Since for  $\lambda\in \cU$ cohomological $\rho_{\pi_\lambda|\mathrm{G}_{F_v}}$ has 
 labeled Hodge-Tate weights   $\left(\tfrac{k_{\lambda,\sigma}-\sw_\lambda}{2},\tfrac{2-\sw_\lambda-k_{\lambda,\sigma}}{2}\right)_{\sigma\in \Sigma_v}$, we deduce from 
 \cite[Thm 1.8]{liu:triangulation} (taking into account that \cite[Def 1.7(f)]{liu:triangulation} has a typo, the exponents should be $\kappa_i(z)_\tau$), that there exists a triangulation
\begin{align}\label{eq:drig triangulation}\cD_{\rig}^\dagger(\rho_{\cU|\mathrm{G}_{F_v}})\sim \begin{pmatrix}
\psi_{v,1}&*\\ &\psi_{v,2}\end{pmatrix},
\end{align}
where $\psi_{v,1},\psi_{v,2}:F_v^\times \to
\cO(\cU)^\times$ are continuous characters such that $\psi_{v,1}(\varpi_v)=\alpha_v^\circ$ 
is the analytically varying renormalized Hecke eigenvalue, $\psi_{v,2}(\varpi_v)=\chi_{\cyc}^{\sw_\lambda-1}(\varpi_v)(\alpha_v^\circ)^{-1}$, and for $z_v\in \cO_v^\times$:
\begin{align*}
\psi_{v,1}(z_v)(\lambda)=\prod_{\sigma\in \Sigma_v}\sigma(z_v)^{(\sw_\lambda+k_{\lambda,\sigma}-2)/2},\text{ and }\psi_{v,2}(z_v)(\lambda)=\prod_{\sigma\in \Sigma_v}\sigma(z_v)^{(\sw_\lambda-k_{\lambda,\sigma})/2}.\end{align*}

Suppose the Galois representation $V=V_{\pi}(\tfrac{2-\sw}{2})$ satisfies $\rH^1_f(F, V)=0$, 
as predicted by the Bloch-Kato conjecture when $L(\pi,\tfrac{1-\sw}{2})\neq 0$. For $v\in S_p$, the $(\varphi_v,N_v)$-submodule
$D_v=\mathcal{D}_{\st}(\nu_v(\tfrac{2-\sw}{2})) \subset
\mathcal{D}_{\st}(V_v)$ is  regular  in the sense of Perrin-Riou
\cite[\S3.1.2]{perrin-riou:Lp} (see also \cite{panchishkin:motives}). In this context, the technical conditions of 
 Greenberg \cite{greenberg:trivial-zeros} and Benois \cite{benois:L-invariant} mentioned in the introduction are all satisfied  and there
 is a well-defined arithmetic $\mathscr{L}$-invariant $\mathscr{L}_{\GB}(V, \{D_v\})$
 (this also uses \cite{harron-jorza,rosso:l-invariants-gsp4} extending the construction to 
  an arbitrary $F$). We will not recall its intricate construction, but will instead  show
that it can be computed in this instance by a formula similar to
that of Proposition \ref{p:FM-derivatives}. 

\begin{proposition}\label{p:GB derivatives}
Under the hypothesis $\rH^1_f(F,V)=0$ we have $ \mathscr{L}_{\GB}(V, \{D_v\})=\mathscr{L}(\widetilde{\pi})\cdot \prod\limits_{v\in E}f_v^{-1} $.
\end{proposition}
\begin{proof}
Consider the triangulation 
 of $\rho_\cU(\frac{2-\sw}{2})$ induced from \eqref{eq:drig triangulation} restricted to the $1$-dimensional affinoid $ \cX(\widetilde{\pi}) \cap \cX'_\varnothing$ (see Def. \ref{defn:XKS}). Prop. \ref{p:trivial-zeros-location} implies that $V$ has a trivial zero contribution exactly from
places $v\in E$. Therefore we can apply
\cite[Thm. 4.1, Prop. 4.13]{rosso:l-invariants-gsp4} to compute the $\mathscr{L}$-invariant in terms of  $\dlog$ in the 
parallel direction $u_\sigma=1=-u_0$ for all $\sigma\in \Sigma$ (the theorem in {\it loc. cit.}  assumes parallel weights, but its proof applies verbatim to general weights deforming in a parallel direction). Thus 
\[\mathscr{L}_{\GB}(V, \{D_v\})=\prod_{v\in
 E}\frac{f_v^{-1}\dlog_u ((\psi_{1,v}\psi_{2,v}^{-1})(\varpi_v))}{-\dlog_u
 ((\psi_{1,v}\psi_{2,v}^{-1})(z_v))/\log_p (\mathrm{N}_{F_v/\Q_p}(z_v))}=\prod_{v\in E}-2f_v^{-1}\dlog
 \alpha_v^\circ,\]
for any units $z_v\in \cO_v^\times$, $v\in E$. The desired formula follows from the fact that $\dlog
 (\psi_{1,v}\psi_{2,v}^{-1}(z_v)) =\sum\limits_{\sigma\in \Sigma_v} u_\sigma \log_p\sigma(z_v)=\log_p \mathrm{N}_{F_v/\Q_p}(z_v)$. 
\end{proof}

Finally, we
  remark that $D=\bigoplus_{v\mid
    p}\Ind_{\mathrm{G}_{F_v}}^{G_{\mathbb{Q}_p}}D_v$ is a regular
  submodule of $(\Ind_{\mathrm{G}_F}^{\mathrm{G}_\Q}V)_{|\mathrm{G}_{\mathbb{Q}_p}}$. In
  the Main Theorem we are proving the Trivial Zero Conjecture
  for the pair $(\Ind_{\mathrm{G}_F}^{\mathrm{G}_\Q}V, D)$.
 
\section{Functional equations} \label{s:functional-equation}
In this section we establish functional equations for the $p$-adic $L$-functions constructed in \S\ref{s:p-adic-L} based on their growth, interpolation properties, and the Godement-Jacquet functional equation of the corresponding
complex $L$-functions. These results will be used in the proof of the Trivial Zero Conjecture at the central point in \S \ref{s:ezc}.

\subsection{Functional equations for archimedean $L$-functions}\label{ss:functional-equation-archimedean}
The Jacquet-Langlands global $L$-function and $\varepsilon$-factor are products over all places of $F$
\[L(\pi,s)=\prod_v L(\pi_v,s)\text{ and } \varepsilon(\pi,s)=\prod_v \varepsilon(\pi_v,\psi_v,dx_v,s),\]
where $\psi_v$ and $dx_v$ are as in \S\ref{s:notations} and will henceforth be dropped for the notation. 
Under the assumption \eqref{A12}, $\pi\otimes|\cdot|_F^{-\sw/2}$ is unitary and self-dual, and its root number is given by
\begin{align}\label{eq:eps-pi}
\varepsilon_\pi=\varepsilon(\pi,\tfrac{1-\sw}{2})\in \{\pm 1\}. 
\end{align}
Let $\gc_\pi=\gn \prod\limits_{v\in \St_p}v$
 be the conductor of $\pi$. The functional equation states (see \cite[Thm.11.1]{jacquet-langlands}):
\begin{align}\label{eq:FE-twist}
L(\pi,s)=\varepsilon(\pi,s)L(\pi^\vee, 1-s) \text{, where } \varepsilon(\pi,s) =\varepsilon_\pi\cdot 
(\mathrm{N}_{F/\Q}(\gd^2\gc_\pi))^{\frac{1-\sw}{2}-s}.
\end{align}
\begin{prop}\label{functional equation}
Let $\chi$ be a character of $\Gal_{p\infty}$ of conductor $\gc_\chi= \prod\limits_{v\in S_p} v^{c_v}$. Then
		\begin{align*}	
		\gc_{\pi\otimes\chi}=\gn_{\chi} \gc_\chi^2, \text{ and } \quad
		\varepsilon_{\pi\otimes\chi}=\varepsilon_\pi \mathrm{N}_{F/\Q}(\gc_\chi)	\chi(\varpi_{\gn_{\chi}}) \tau(\overline{\chi})^2
	\prod_{{v\in \St_p}, {c_v>0}}\varepsilon(\pi_v,\tfrac{1-\sw}{2}), 	
	 \end{align*}
 where $\tau(\chi)$ is the Gauss sum defined in \eqref{eq:global-gauss-sum} and $ \gn_{\chi}=\gn \cdot\prod\limits_{{v\in \St_p},{c_v=0}}v$. 
	\end{prop}

\begin{proof} The conductor formula follows from \eqref{A12}. The formula involving the $\varepsilon$-factors 
will be checked by decomposing both sides as products of local terms and using \cite[\S1.3]{jacquet-langlands}. 

 If either $\chi_v$ or $\pi_v$ is unramified then using that 
 $|\omega_\pi|_F^{\sw}=q_v^{-\sw}$ one has
		\[\varepsilon(\pi_v\otimes\chi_v,\tfrac{1-\sw}{2})=\varepsilon(\pi_v,\tfrac{1-\sw}{2}) q_v^{c_v}\chi_v(\varpi_v)^{\gc_{\pi_v}}\tau(\overline{\chi}_v,\psi_v,d_{\chi_v})^2.\]

If $v\in \St_p$ and $\chi_v$ is ramified, then using $\varepsilon(\pi_v,\tfrac{1-\sw}{2}) = -\nu_v(\varpi_v)^{-1}q_v^{-\sw/2}$
 one has	\[\varepsilon(\pi_v\otimes\chi_v,\tfrac{1-\sw}{2})=\varepsilon(\pi_v,\tfrac{1-\sw}{2})
 q_v^{c_v}\tau(\overline{\chi}_v,\psi_v,d_{\chi_v})^2 \varepsilon(\pi_v,\tfrac{1-\sw}{2}) . \qedhere\]
\end{proof}

\subsection{Interpolation formulas} Under the assumption \eqref{A12}, Theorem \ref{t:improved-Lp} and Proposition \ref{p:trivial-zeros-location} take the following more familiar form, involving the global Gauss sum $\tau(\chi)$ (see \eqref{eq:global-gauss-sum}). 

\begin{corollary}\label{mtt-interpolation}
For any $S\subset S_p$, $r\in \Z$ critical for $(k,\sw)$ and $\chi$ character of conductor $\prod\limits_{v\in S} v^{c_v}$
\begin{equation*}
L_S(\widetilde{\pi}, \chi, r) = L_S((k, \sw), \chi, r)=
\frac{\mathrm{N}_{F/\Q}^{r-1}(i\gd)}{\Omega_{\widetilde{\pi}}^{\chi_\infty \omega_{p,\infty}^{r-1}}} 
L(\pi\otimes\chi,r-\tfrac{1}{2}) \tau(\chi) \prod_{v\in S_p}E_S(\widetilde{\pi}_v,\chi_v,r), \text{ where}  
 \end{equation*}
\[\prod_{v\in S_p}E_S(\widetilde{\pi}_v,\chi_v,r)= 
\prod_{c_v>0}\frac{q_v^{r c_v}}{\alpha_v^{c_v}} 
\prod_{\ontop{c_v=0}{v\notin \St_p }}\left(1-\frac{\chi_v(\varpi_v)}{\alpha_v q_v^{r+\sw-1}}\right)
\prod_{\ontop{c_v=0}{ v\in S }} \left(1-\frac{q_v^{r-1}}{\alpha_v\chi_v(\varpi_v)}\right).\]

Moreover $E_S(\widetilde{\pi}_v,\chi_v,r)$ vanishes if and only if $r=\tfrac{2-\sw}{2}$ is the central critical point, 
 $v\in S\cap \St_p$ and $\chi_v(\varpi_v)=\alpha_v^{-1}q_v^{-\sw/2}=-\varepsilon(\pi_v,\tfrac{1-\sw}{2})\in\{\pm 1 \}$. 
\end{corollary}

In \S \ref{s:ezc} we will prove that the order of vanishing of $L_p(\widetilde{\pi},s)$ at $s=\tfrac{2-\sw}{2}$ is at least as large as the number of places $v$ where $E_S(\widetilde{\pi}_v,\chi_v,\tfrac{2-\sw}{2})=0$. We remark that the interpolation formula from Corollary \ref{mtt-interpolation} only gives information about the vanishing of the $p$-adic $L$-function, and leaves completely unanswered the question of higher orders of vanishing.

The next result will be used in \S\ref{ss:functional-equation} to prove the functional equation for $p$-adic $L$-functions.

\begin{corollary}\label{c:FE-formula}
Suppose $\widetilde\pi$ satisfies   \eqref{A12} and let $\widetilde{\varepsilon}_\pi=\varepsilon_\pi \cdot \prod_{v\in \St_p}\varepsilon(\pi_v,\tfrac{1-\sw}{2})\in \{\pm 1\}$. Then 
 \begin{align}\label{eq:functional} 
 \cL_p(\widetilde{\pi},\chi \langle \cdot \rangle_p^{r-1})	=\widetilde{\varepsilon}_\pi\cdot 
	(\chi\omega_{p}^{\sw/2})(-\varpi_{\gn}) \cdot \langle \gn\rangle_p^{r-1+\sw/2}
	\cL_p(\widetilde{\pi},\chi^{-1} \omega_p^{-\sw}\langle \cdot \rangle_p^{1-\sw-r}), 
	\end{align}
for any finite order characters 	 $\chi:\Gal_{p\infty}\to L^\times$ and any integer 
 $r$ critical for $(k,\sw)$.
\end{corollary}

\begin{proof}
Since $\pi^\vee= \pi\otimes |\cdot|_F^{-\sw}$, 
 the archimedean functional equation \eqref{eq:FE-twist} for $\pi\otimes\chi$, Proposition \ref{functional equation} and Corollary \ref{mtt-interpolation} for $S=S_p$ yield 
 \begin{align*}
	&\frac{\Omega_{\widetilde{\pi}}^{\chi_\infty \omega_{p,\infty}^{r-1}} \cL_p(\widetilde{\pi},\chi \chi^{r-1}_{\cyc})}{\mathrm{N}_{F/\Q}^{r-1}(i\gd)\tau(\chi)\prod\limits_{v\in S_p}E(\widetilde{\pi}_v,\chi_v,r)}=
	\varepsilon_\pi \prod_{{v\in \St_p}, {c_v>0}}\varepsilon(\pi_v,\tfrac{1-\sw}{2}) 	
			\mathrm{N}_{F/\Q}(\gc_\chi)\tau(\overline{\chi})^2 \times\\ & \times 
			\mathrm{N}_{F/\Q}(\gd^2\gn_{\chi}
	\gc_\chi^2 )^{1-\sw/2-r}\chi(\varpi_{\gn_{\chi}})		
	\frac{\Omega_{\widetilde{\pi}}^{\chi_\infty \omega_{p,\infty}^{r-1+\sw}} \cL_p(\widetilde{\pi},\chi^{-1}\chi^{1-\sw-r}_{\cyc})}{\mathrm{N}_{F/\Q}^{1-\sw-r}(i\gd)\tau(\overline{\chi})\prod\limits_{v\in S_p}E(\widetilde{\pi}_v,\chi_v^{-1},2-\sw-r)},
	\end{align*}

 Using the global Gauss sum identity $\mathrm{N}_{F/\Q}(\gc_\chi)\tau(\chi)\tau(\overline{\chi})=\chi_f(-1)=\chi_\infty(-1)$  we get:
	\begin{align*}
	\cL_p(\widetilde{\pi},\chi \chi^{r-1}_{\cyc})=\varepsilon_\pi \prod_{{v\in \St_p}, {c_v>0}}\varepsilon(\pi_v,\tfrac{1-\sw}{2}) \cdot 
	\frac{\chi(-\varpi_{\gn_{\chi}})\cL_p(\widetilde{\pi},\chi^{-1}\chi^{1-\sw-r}_{\cyc})}{\mathrm{N}_{F/\Q}(-\gn_{\chi}\gc_\chi^2)^{r+\sw/2-1}}  \prod_{v\in S_p}\frac{E(\widetilde{\pi}_v, \chi_v,r)}{E(\widetilde{\pi}_v,\chi_v^{-1},2-\sw-r)}.
	\end{align*}
 When
$c_v=0$ and $\pi_v$ is unramified then $\alpha_v \beta_v = q_v^{1-\sw}$ and so $E(\widetilde{\pi}_v,
 \chi_v,r)=E(\widetilde{\pi}_v,\chi_v^{-1},2-\sw-r)$. 
 When $c_v>0$ then $\displaystyle \frac{E(\widetilde{\pi}_v,
 \chi_v,r)}{E(\widetilde{\pi}_v,\chi_v^{-1},2-\sw-r)}=q_v^{(2r+\sw-2)c_v}$. 
   Finally, when $c_v=0$ and $v\in \St_p$ then 
\[ \frac{E(\widetilde{\pi}_v,
 \chi_v,r)}{E(\widetilde{\pi}_v,\chi_v^{-1},2-\sw-r)}=-\alpha_v 
 q_v^{r-1+\sw}\chi_v(\varpi_v)^{-1}
= \varepsilon(\pi_v,\tfrac{1-\sw}{2}) q_v^{r-1+\sw/2}\chi_v(\varpi_v)^{-1}.\]
Putting everything together and using that $\gn_{\chi}=\gn \cdot\prod\limits_{v\in \St_p, c_v=0} v$ we obtain:
\[\cL_p(\widetilde{\pi},\chi \chi^{r-1}_{\cyc})
	=\widetilde{\varepsilon}_\pi \chi(-\varpi_{\gn})\mathrm{N}_{F/\Q}(-\varpi_{\gn})^{1-r-\sw/2}\cL_p(\widetilde{\pi},\chi^{-1}\chi^{1-\sw-r}_{\cyc}). \]
 Replacing $\chi$ by $\chi \omega_p^{1-r}$, and noting that $\chi_{\cyc}\omega_p^{-1}=\langle \cdot \rangle_p$ is an even character, 
 yields \eqref{eq:functional}. 
  \end{proof}

When $\widetilde\pi$ has very non-critical slope, Corollary \ref{c:FE-formula} provides enough relations to establish a functional equation for 
$\cL_p(\widetilde{\pi},\chi\bullet)\in D(\Gal_{\cyc},\cO(\cU))$ given its growth. 
When $\widetilde\pi$ has critical slope, but is still is non-critical, then the functional equation will be proven in the next subsection using the unique $p$-adic family containing $\widetilde\pi$ from Theorem \ref{free-etale}. 

\subsection{Functional equations for $p$-adic $L$-functions}\label{ss:functional-equation}

We recall that by construction the multi-variable $p$-adic $L$-function 
$\cL_p\in D(\Gal_{p\infty}, \cO(\cX(\widetilde{\pi})))$ from \eqref{d:2var-L_p} interpolates the $p$-adic $L$-functions 
$\cL_p(\widetilde{\pi}_\lambda, \bullet) \in D(\Gal_{p\infty}, \cO(\cX(\widetilde{\pi})))$ for all cohomological $\lambda\in \cX(\widetilde{\pi})$. 

The cyclotomic (resp. multi-variable) $p$-adic $L$-function attached to $\widetilde{\pi}$ is defined as
\begin{align}\label{eq:padicL}
L_p(\widetilde{\pi},s)=\cL_p(\widetilde{\pi}, \omega_p^{-\sw/2} \langle \cdot \rangle_p^{s-1} ), 
\text{ resp. } L_p(\lambda,s)=\cL_p(\lambda, \omega_p^{-\sw/2} \langle \cdot \rangle_p^{s-1} ), \enspace s\in \cO_{\C_p}, \lambda \in \cX(\widetilde{\pi}). 
\end{align}
By Proposition \ref{p:renormalization-padic-periods} for $z\in \cO_{\C_p}$ and $(k_\lambda,\sw_\lambda)\in
\cX^{\an}(\widetilde{\pi})$ such that $(k_\lambda,\sw_\lambda \langle \cdot\rangle^{2z})\in \cX^{\an}(\widetilde{\pi})$ one has 
\begin{align}\label{eq:twisting}
L_p((k_\lambda,\sw_\lambda+2z), s)=L_p((k_\lambda,\sw_\lambda), s+z).
\end{align}

\begin{theorem}\label{t:functional-equation-2}
The sign $\widetilde{\varepsilon}_{\pi_\lambda}$ of $\widetilde{\pi}_\lambda$ is independent of the cohomological weight $\lambda\in
\cX(\widetilde{\pi})$. For any $\lambda \in \cX(\widetilde{\pi})$,  $f\in A(\Gal_{\cyc},L)$  and for any finite order character  $\chi:\Gal_{p\infty}\to L^\times$, one has
\begin{align}\label{eq:functional-equation-2}
 \cL_p(\lambda,\chi \cdot f)	=\widetilde{\varepsilon}_\pi\cdot 
	(\chi \omega_{p}^{\sw/2} f)(-\varpi_\gn) \langle \gn\rangle_p^{\sw_\lambda/2}
	\cL_p(\lambda,\chi^{-\sw_\lambda}_{\cyc} ( (\chi\cdot f)\circ (\cdot )^{-1})). 
\end{align}

In particular, $L_p(\widetilde{\pi}_\lambda,s) =
\widetilde{\varepsilon}_{\pi} \cdot \langle \gn
\rangle_p^{s-1+\sw_\lambda/2}
L_p(\widetilde{\pi}_\lambda,2-\sw_\lambda-s)$ as analytic
functions in $s$.
\end{theorem}

\begin{proof} 
We remark that the slope is constant in the family and equals $\sum\limits_{v\in S_p}e_{v}h_{\widetilde{\pi}_v}$ (see Def. \ref{d:non-critical-slope}). Consider a cohomological weight $\lambda'\in \cX(\widetilde{\pi})$ having very non-critical slope (see \eqref{eq:very-non-critical-slope}). 
Since the elements $\cL_p(\lambda',\chi \cdot)$ and $\cL_p(\lambda',\chi_{\mathrm{cyc}}^{-\sw_{\lambda'}} \chi^{-1} \cdot)$ of $D(\Gal_{\cyc},\cO(\cU))$ both have growth at most $\sum\limits_{v\in S_p}e_{v}h_{\widetilde{\pi}_v}$, by \cite[Thm.2.3, Lem.2.10]{vishik}  it suffices to check 
\eqref{eq:functional-equation-2} for all $f=\chi' \langle \cdot \rangle_p^{r-1}$, where $r$ is a 
critical integer for $(k,\sw)$ and $\chi' $ is a finite order character of $\Gal_{\cyc}$ . This is precisely formula \eqref{eq:functional}
applied to $\chi\chi'$, except that there is $\widetilde{\varepsilon}_{\pi_{\lambda'}}$ instead of $\widetilde{\varepsilon}_{\pi}$. 

We now turn to the problem of showing that the sign is generically constant in the family. Choose $\lambda'\in \cX(\widetilde{\pi})$ 
as above for which $2-\tfrac{\sw_{\lambda'}}{2}$ is a critical integer. 
The absolute convergence of 
 $L(\pi_{\lambda'}\otimes\chi', s)$ for $\Re(s)>1-\frac{\sw_{\lambda'}}{2}$ implies that 
 $L(\pi_{\lambda'} \otimes \chi\omega_{p}^{-1},(3-\sw_{\lambda'})/2)\neq 0$. Moreover Proposition \ref{p:trivial-zeros-location} implies that 
 $E(\widetilde{\pi}_{\lambda',v}, \chi_{v}\omega_{p}^{-\sw/2}, 2-\frac{\sw_{\lambda'}}{2}) \neq 0$ for any $v\in S_p$, and therefore 
 $\cL_p(\lambda',\chi'\chi_{\cyc}^{1-\frac{\sw_{\lambda'}}{2}})\neq 0$ by Corollary \ref{mtt-interpolation}.
 Letting $f=\omega_p^{-\sw/2} \langle\cdot\rangle_p^{s-1}$, the quotient
 \[\varepsilon(\lambda,s)=
 \frac{\langle \gn\rangle_p^{1-s-\sw_\lambda/2}\cL_p(\lambda,\chi\omega_p^{-\sw/2}\langle\cdot\rangle_p^{s-1})}{\chi(-\varpi_{\gn}) \cL_p(\lambda,\chi^{-1}\omega_p^{-\sw/2} \langle\cdot\rangle_p^{1-\sw_{\lambda}-s})},\]
 is a well-defined, non-identically zero meromorphic function in the variables $(\lambda,s)\in \cX(\widetilde{\pi})\times \cO_{\C_p}$. 
 Indeed we have shown that its numerator does not vanish at $(\lambda',2-\frac{\sw_{\lambda'}}{2})$, and $
 \varepsilon(\lambda',2-\frac{\sw_{\lambda'}}{2})= \widetilde{\varepsilon}_{\pi_{\lambda'}}$ by the  functional equation. 
 Similarly $\varepsilon(\lambda,s)=\widetilde{\varepsilon }_{\pi_\lambda}\in \{\pm 1\}$ 
for all cohomological $\lambda\in \cX(\widetilde{\pi})$ having very non-critical slope and such that $\varepsilon(\lambda,s)$ is well-defined, 
and the Zariski density of such points implies  that $\varepsilon(\lambda,s)$ is constant with value $\widetilde{\varepsilon }\in \{\pm 1\}$, 
independent of $\chi$.

To finish the proof of \eqref{eq:functional-equation-2} it suffices to check that
$\widetilde{\varepsilon}=\widetilde{\varepsilon}_{\pi_\lambda}$ for all cohomological $\lambda\in \cX(\widetilde{\pi})$. By a theorem of Rohrlich \cite{rohrlich:nonvanishing} applied at the central critical point, there exists a finite order character $\chi'$ of $\Gal_{p\infty}$ such that $L(\pi_\lambda \otimes \chi',\tfrac{1-\sw_\lambda}{2})\neq 0$ and $E(\widetilde{\pi}_{\lambda,v}, \chi'_{v}\omega_{p}^{-\sw/2}, 1-\tfrac{\sw_\lambda}{2})\neq 0$ for all $v\in S_p$ . Then Corollary \ref{mtt-interpolation} implies that
$\cL_p(\lambda, \chi'\chi_{\mathrm{cyc}}^{-\sw_{\lambda}/2})\ne 0$, 
and it follows then from  Corollary \ref{c:FE-formula} that 
$\widetilde{\varepsilon}_{\pi_{\lambda}}=\widetilde{\varepsilon}$.
\end{proof}

\section{The Trivial Zero Conjecture}\label{s:ezc}

By Corollary \ref{mtt-interpolation}, the set $E\subset S_p$ of places at which the local interpolation factor of $L_p(\widetilde{\pi},s)$
vanishes at the central point $\tfrac{2-\sw}{2}$ consists precisely of $v\in \St_p$ such that 
$\varepsilon(\pi_v,\tfrac{1-\sw}{2})=-1$.

\begin{theorem}[Trivial zero conjecture at the central critical point]\label{t:ezc-hmf} 
The $p$-adic $L$-function $L_p(\widetilde{\pi},s)$ has order of vanishing at
least $e=|E|$ at $\tfrac{2-\sw}{2}$ and
\[\frac{L_p^{(e)}(\widetilde{\pi},\tfrac{2-\sw}{2})}{e!} =
\mathscr{L}(\widetilde{\pi})
 \frac{L(\pi,\tfrac{1-\sw}{2})}{\mathrm{N}_{F/\Q}^{\sw/2}(i\gd) \Omega_{\widetilde{\pi}}^{\omega_{p,\infty}^{\sw/2}}}
  \cdot 2^{|\St_p\!\setminus E|}
 \prod_{v\in S_p\!\setminus \St_p}
 (1-\alpha_v^{-1}q_v^{-\sw/2})^2.\]
Moreover, if the Greenberg-Benois arithmetic $\mathscr{L}$-invariant is defined, the Trivial Zero Conjecture of the introduction holds for the Galois representation $\Ind_{\mathrm{G}_F}^{\mathrm{G}_\Q}V_\pi(\tfrac{2-\sw}{2})$ with the choice of regular submodule as in \S\ref{ss:greenberg-benois}.
\end{theorem}
The proof of Theorem \ref{t:ezc-hmf} will occupy the remainder of this section.

\subsection{Local behavior in partial families}
Crucial for the computation of Taylor coefficients of our $p$-adic $L$-functions is the following technical lemma.

\begin{lemma}\label{l:steinberg}
Let $S=S_p\bs \{v\}$ for some $v\in \St_p$. After possibly shrinking $\cX(\widetilde{\pi})$, for any 
cohomological $\lambda\in \cX_S\cap \cX(\widetilde{\pi})$ the local representation $\pi_{\lambda,v}$ is an unramified twist of the Steinberg representation (see Def. \ref{defn:XKS} and \eqref{eq:Xpi}).
\end{lemma}
\begin{proof} Let us first present  an argument using $p$-adic Hodge theory and the existence of the analytic Galois representation
$\rho_\cU$ from \eqref{eq:rhoU}.
Let $\mathcal{Z}\subset \cX(\widetilde{\pi})\cap \cX_S$ be the Zariski closure of the subset of cohomological  points $\lambda$
such that  $\pi_{\lambda,v}$ is unramified, {\it i.e.}, $\rho_{\pi_\lambda,v}$ is crystalline. 
Since the  labeled Hodge-Tate weights of $\rho_{\pi_\lambda,v}$ are constant for all cohomological $\lambda\in \cX_S$, 
a Theorem of Berger and Colmez \cite{berger-colmez} (see also \cite[Prop.3.17]{chenevier:application}) implies that 
$(k,\sw)\notin \mathcal{Z}$, hence the claim. 

We now give a proof  based on the theory of partially
finite slope families developed in \S\ref{subsection-controls},  which we believe is of  independent interest.
Let $K=K(\widetilde{\pi},\gu)$ (which, by assumption
\eqref{A12}, is the same as $K(\widetilde{\pi}_S,\gu)$),
$h=(h_{\widetilde{\pi}_w})_{w\in S_p}$ and
$h_S=(h_{\widetilde{\pi}_w})_{w\in S}$. Writing $\cU_S=\cX_S\cap
\cX(\widetilde{\pi})$, the natural restriction map
$\mathcal{D}_{\cU_S}\to \mathcal{D}_{S,\cU_S}$ yields a
$\widetilde{\mathbb{T}}_S$-equivariant map
$H_c^d(Y_K,\mathcal{D}_{\cU_S})^{\epsilon,\leqslant h}\to
H_c^d(Y_K,\mathcal{D}_{S,\cU_S})^{\epsilon,\leqslant h_S}$ such
that the action of $U_{\varpi_v}$ on the left agrees with that
of $U_{\varpi_v}^\circ$ on the right. Under this map, the localization $H_c^d(Y_K,\mathcal{D}_{\cU_S})^{\epsilon,\leqslant h}_{\widetilde{\pi}}$ maps to $H_c^d(Y_K,\mathcal{D}_{S,\cU_S})^{\epsilon,\leqslant h_S}_{\widetilde{\pi}_S}$ further localized at the ideal generated by $U_{\varpi_v}^\circ-\alpha_v^\circ$ and $U_{\delta}-\nu_v(\delta)$ for $\delta\in \mathcal{O}_v^\times$, which does not vanish.
Therefore, Theorem \ref{free-etale}(ii) applied to both $S_p$ and $S$
implies that, after possibly shrinking
$\cX(\widetilde{\pi})$, there exist components $\mathcal{V}$ of $\mathbb{T}_{\cU_S}^{\leqslant h}$ and $\mathcal{V}_S$ of $\mathbb{T}_{S,\cU_S}^{\leqslant h_S}$ such that the weight maps $\kappa:\cV\xrightarrow{\sim} \cU_S$ and $\kappa_S: \cV_S\xrightarrow{\sim} \cU_S$ are isomorphisms and there exists an $\cO(\cU_S)$-linear natural map between free rank one $\cO(\cU_S)$-modules
\begin{align}\label{eq:compatibility-S}
\rH_c^\bullet(Y_K, \cD_{\cU_S})^{\epsilon, \leqslant h}\otimes_{\T_{\cU_S}^{\leqslant h}}
\cO(\cV)\xrightarrow{\sim} \rH_c^\bullet(Y_K,
\cD_{S,\cU_S})^{\epsilon, \leqslant
 h_S}\otimes_{\T_{S,\cU_S}^{\leqslant h_S}}
\cO(\cV_S).
\end{align}
After fixing bases, the above isomorphism is necessarily given by multiplication by some 
$g\in \cO(\cU_S)$. Specializing at $(k,\sw)$, the classicality Theorem
\ref{t:classicality} implies that $g(k,\sw)\neq 0$, since multiplication by $g(k,\sw)$ takes
the class corresponding to $\widetilde{\pi}_f^K$ to the class
corresponding to $\widetilde{\pi}_{S,f}^K$.
We then shrink $\cX(\widetilde{\pi})$ so as to ensure that $g$ is
nowhere vanishing on $\cU_S$. 

By the proof of Theorem \ref{free-etale}(iii), specializing
\eqref{eq:compatibility-S} at $\lambda \in \cU_S(L)$ cohomological yields \[\rH_c^\bullet(Y_K,\cL_{\lambda}^\vee(L))^{\epsilon,\leqslant h}\otimes_{\T^{\leqslant h}_{\kappa^{-1}(\lambda)}}L\xrightarrow{\sim}\rH_c^\bullet(Y_K,\cL_{\lambda}^\vee(L))^{\epsilon,\leqslant h_S}\otimes_{\T^{\leqslant h_S}_{S,\kappa_S^{-1}(\lambda)}}L,\]
which is an isomorphism since multiplication by $g(\lambda)$ is a non-zero map between $L$-lines. We conclude that $\widetilde{\pi}_{\lambda}$ and $\widetilde{\pi}_{\lambda,S}$ have the same Hecke eigenvalues away from a bad set of primes, and therefore $\pi_\lambda\simeq\pi_{\lambda,S}$ by Strong Multiplicity One for $\GL_2$.

It therefore suffices to show that $\pi_{\lambda,S,v}$ is
Steinberg for all $\lambda\in \cU_S$. However, Theorem
\ref{free-etale}(iii) and its proof imply that $\widetilde{\pi}_{\lambda,S}$
is a non-critical $S$-refinement, hence  we have:
\[1=\dim \rH_c^\bullet(Y_K,\cD_{S,\lambda})^\epsilon_{\gm_{\widetilde{\pi}_{\lambda,S}}}=\dim(\pi_{\lambda,S,f}^K)_{\gm_{\widetilde{\pi}_{\lambda,S}}}.\]
Decomposing as tensor product we deduce that $\dim(\pi_{\lambda,S,v}^{I_v})=1$, as we are not localizing using 
Hecke operators at $v\notin S$ (see Def. \ref{d:mpi-tilde}). Hence $\pi_{\lambda,v}\simeq \pi_{\lambda,S,v}$ 
is an unramified twist of the Steinberg representations, as unramified representations have a 2-dimensional $I_v$-invariants.
\end{proof}

\subsection{Taylor coefficients}\label{exc-padicL}

In this section we use the interplay between properties of partially improved $p$-adic $L$-functions and 
the variation of the root number in partial finite slope families to 
establish the vanishing of many Taylor coefficients of the $p$-adic $L$-function of the family. 

\begin{definition}
Let $u\in 4 \mathcal{O}_{\mathbb{C}_p}$ and $x=(x_v)_{v\in E}\in (2p \mathcal{O}_{\mathbb{C}_p})^E$. For any subset $S\subset E$ we denote $x_S= (x_v)_{v\in S}$ and define $\lambda_{x,u}^S=(k_\lambda, \sw_\lambda) \in \cX^{\an}(\widetilde{\pi})$ by  
\[  \sw_\lambda=\sw-u \text{ and } 
k_{\lambda,\sigma}=\begin{cases} k_\sigma &,  \text{ for }\sigma\in \Sigma_{S_p\bs E},\\
k_\sigma+u  &,  \text{ for } \sigma\in \Sigma_{E\bs S},\\
k_\sigma+x_v &, \text{ for } \sigma\in \Sigma_S. \end{cases} \]
\end{definition}

Writing $\displaystyle \mathbb{L}_p(x, u) = \langle
\gn\rangle_p^{u/4}L_p(\lambda_{x,u}^E, \tfrac{2-\sw}{2})$, 
Theorem \ref{t:functional-equation-2} implies that 
\begin{align}\label{eq:odd-u} 
\mathbb{L}_p(x,-u) = \widetilde{\varepsilon} \cdot \mathbb{L}_p(x, u), \text{ with } \widetilde{\varepsilon}=(-1)^e\varepsilon_\pi.
\end{align}
In particular we may write $\mathbb{L}_p(x,u) = \sum\limits_{i\geqslant 0}A_i(x)u^{i}$, 
where $A_i(x)$ is $p$-adic analytic in $(x_v)_{v\in E}$ and the sum runs over 
$i$ even (resp. odd) when $\widetilde{\varepsilon}=1$ (resp. $\widetilde{\varepsilon}=-1$). 
By \eqref{eq:twisting} we see that
\[L_p(\widetilde{\pi}, s) =\langle
\gn\rangle_p^{(2s+\sw-2)/4}\mathbb{L}_p((0)_{v\in E}, 2-\sw-2s).\]

By definition, 
$\lambda_{x,u}^S\in \cX^{\an}(\widetilde{\pi})\cap \cX_{S\sqcup(S_p\bs E)}'$ and we can define: 
\begin{align}\label{eq:LL-FE}
\mathbb{L}_S(x_S, u) = \langle
\gn\rangle_p^{u/4}L_{S\sqcup(S_p\bs E)}(\lambda_{x,u}^S,\mathbbm{1}, \tfrac{2-\sw}{2}),
\end{align}
where $L_{S\sqcup(S_p\bs E)}$ is the the improved $p$-adic $L$-function from \S\ref{ss:improved}. 

By definition we have $\mathbb{L}_p(x, u)= \mathbb{L}_E(x, u)$ and by Theorem \ref{t:improved-Lp}
for all $S\subset E$ we have 
\begin{align}\label{eq:improved-LL}
\mathbb{L}_p((x_{S}, (u)_{v \in E\bs S}), u)=\mathbb{L}_{S}(x_{S}, u)\prod_{v\in E\bs S}
\left(1-\alpha_v^\circ(\lambda_{x,u}^{S})^{-1}\prod\limits_{\sigma\in \Sigma_v}\sigma(\varpi_v)^{(k_\sigma-2)/2} \right). 
\end{align}

We now turn to the Taylor expansions of the functions $A_i(x)$. 
For a multi-index $n=(n_v)_{v\in E}$ of non-negative integers we denote
$x^{n}=\prod\limits_{v\in E}x_v^{n_v}$. For each $i$ write
the power series expansion
\[A_i(x)=\sum_{n\in \Z_{\geqslant 0}^E} a_{i}(n)x^{n}.\]
We will prove that a large number of such coefficients vanish, a fact which is not implied by the Trivial Zero Conjecture. 
 More precisely all coefficients of total degree $<e$ and as well as most coefficients in degree $e$ vanish. 
 For $S\subset E$ we will write $n_S = (n_v)_{v\in S}$ and
$n=(n_S,n_{E\bs S})$.

For convenience we denote $|n|=\sum\limits_{v\in E}n_v$ and
$||n||=|\{v\in E\mid n_v\neq 0\}|$, so that $||n||\leqslant
|n|$.

Our first technical result concerns the vanishing of certain
linear combinations of the Taylor coefficients $a_i(n)$. Recall that the function $\mathbb{L}_p(x,u)$ is even in $u$ if $\widetilde{\varepsilon}=1$ and odd in $u$ if $\widetilde{\varepsilon }=-1$.
\begin{lemma}\label{l:taylor1}
Let $S\subset E$ be such that $(-1)^{|E\bs S|}=-\widetilde{\varepsilon}$. For $n_S\in \mathbb{Z}_{\geqslant 0}^S$ and  $\ell\leqslant e-|S|$ we have
\[\sum_{n_{E\bs S}} a_{\ell-|n_{E\bs S}|}(n_S,n_{E\bs S})=0.\]
\end{lemma}
\begin{proof} For $x_S\in \mathbb{Z}_{>0}^S$  $p$-adically close to $0$, 
$\lambda=\lambda_{x_S,0}^S\in \cX^{\an}(\widetilde{\pi})\cap \cX_{S\sqcup(S_p\bs E)}'$ is a
cohomological weight such  that
$k_{\lambda,\sigma}=k_\sigma$ for $\sigma\in \Sigma_{S_p\!\setminus S}$. By Lemma \ref{l:steinberg} and the analyticity of the eigenvalue $\alpha_v^\circ(\lambda)$
we conclude that $\pi_{\lambda,v}\otimes|\cdot|^{-\sw/2}$ is the Steinberg representation for all $v\in E\bs S$, 
in particular  $\varepsilon(\pi_{\lambda,v}, \tfrac{1-\sw}{2})=-1$. From here and from Theorem \ref{t:functional-equation-2} we deduce that
\[\varepsilon(\pi_\lambda,\tfrac{1-\sw}{2}) = (-1)^{|E\bs S|}\cdot\widetilde{\varepsilon}=-1,\]
which implies that $L(\pi_\lambda,\tfrac{1-\sw}{2})=0$. Corollary
 \ref{mtt-interpolation} then implies that $\mathbb{L}_{S}(x_{S}, 0)= 0$. Moreover, the fact that $\varepsilon(\pi_{\lambda,v}, \tfrac{1-\sw}{2})=-1$ for $v\in E\bs S$ implies that $\alpha_v^\circ(\lambda) = \prod\limits_{\sigma\in \Sigma_v}\sigma(\varpi_v)^{(k_{\sigma}-2)/2}$ for $v\in E\bs S$. By Zariski density these equalities are also true for all $x_S$.
Thus each factor  in \eqref{eq:improved-LL} vanishes at $u=0$
and so  $u^{|E\bs S|+1}$ divides the  analytic function $\mathbb{L}_p((x_{S}, (u)_{v \in E\bs S}), u)$. 
Expanding, we deduce that for all $n\in \Z_{\geqslant 0}^E$:
\[u^{|E\bs S|+1}\textrm{ divides }\sum_{n_S}x_S^{n_S}\sum_i\sum_{n_{E\bs
  S}}a_i(n)u^{|n_{E\bs S}|+i}.\]
Collecting terms of the form $x_S^{n_S}u^\ell$ in the above divisibility  yields the desired equality.
\end{proof}

The following proposition proves that the Taylor expansion of $\mathbb{L}_p(x,u)$ contains only terms of degree $\geqslant e$.
\begin{proposition}\label{p:taylor-coeffs}
\begin{enumerate}
\item \label{enum:taylor-i} If $||n||<e-i$, then $a_i(n)=0$. 
\item \label{enum:taylor-ii} For any given $i<e$, we have  $ \sum\limits_{|n|=e-i}a_i(n)=0$.
\end{enumerate}
\end{proposition}
\begin{proof}
If $(-1)^{|E\bs S| }=(-1)^{i+1} \neq -\widetilde{\varepsilon}$ then $A_i(x)= 0$ by (\ref{eq:odd-u})  and both claims are clear. 

(\ref{enum:taylor-i}) We will prove this fact by induction on $(||n||+i,i)$ ordered
lexicographically. The base case $i=||n||=0$ follows from Lemma
 \ref{l:taylor1} applied to $n_S=(0)_{v\in S}$,  $\ell=0$  and to any $S\subset E$ satisfying its hypothesis. Suppose now
that $a_j(m)=0$ for all $j$ and $m$ such that either
$||m||+j<||n||+i$ or $||m||+j=||n||+i$ and $j<i$.
Since $||n||<e-i$ there exists  $S\subset E$  such that $|S|=e-i-1$, $n_v=0$ for all $v\in E\bs S$, and $||n||=||n_S||$. 
Then  $(-1)^{|E\bs S| }=(-1)^{i+1} = -\widetilde{\varepsilon}$  and Lemma \ref{l:taylor1} applied to 
$\ell=i<|E\bs S|=i+1$ yields:
\[\sum_{m_{E\bs S}}a_{i-|m_{E\bs S}|}(n_S,m_{E\bs S})=0.\]
Consider a term $a_j(n_S, m_{E\bs S})$ in the above sum and
write $m=(n_S,m_{E\bs S})$. Then $||m||+j=||n_S||+||m_{E\bs
 S}||+j\leqslant ||n||+|m_{E\bs S}|+j=||n||+i$. The inductive
hypothesis then implies that $a_j(m)=0$ whenever the previous
inequality is strict, or when $j<i$. The sole surviving term in
the sum is then $a_i(n_S,(0)_{v\in E\bs S})=a_i(n)=0$, as
desired.

(\ref{enum:taylor-ii}) Let $S\subset E$ be any subset with cardinality $e-i-1\geqslant 0$ and let
$n_S=(1)_{v\in S}$. Since $(-1)^{|E\bs S| }= -\widetilde{\varepsilon}$,   Lemma \ref{l:taylor1} applied to $\ell=i+1=e-|S|$ yields
\begin{align}\label{eq:some-sums} 
\sum\limits_{n_{E\bs S}}a_{i+1-|n_{E\bs S}|}(n_S,n_{E\bs S})=0.
\end{align}
Letting $n=(n_S,n_{E\bs S})$ and noting that $|n|=|S|+|n_{E\bs S}|$,
we deduce from (\ref{enum:taylor-i}) that 
$a_{i+1-|n_{E\bs S}|}(n)=a_{e-|n|}(n)$ vanishes unless  $||n||=|n|$. 
Summing \eqref{eq:some-sums} over
all such  subsets $S\subset E$ yields:
\begin{align}\label{eq:super-hard}
0=\sum_{|S|=e-i-1}\sum_{|n_{E\bs S}|=||n_{E\bs S}||}a_{e-|n|}(n_S,n_{E\bs S} )=
\sum_{j=0}^{i+1}\binom{e-j}{e-i-1}\sum_{|n|=||n||=e-j}a_j(n).
\end{align}
Since  $\sum\limits_{|n|=e-i}a_i(n)=\sum\limits_{||n||=|n|=e-i}a_i(n)$ by (\ref{enum:taylor-i}),  it suffices to show the 
vanishing of the latter, which is deduced  from \eqref{eq:super-hard} by an  induction on $i$, 
using that  either $A_0(x)=0$ or $A_1(x)=0$. 
\end{proof}

\begin{remark}
While Proposition \ref{p:taylor-coeffs} implies that
$\mathbb{L}_p(x,u)$ only contains monomials of total degree
$\geqslant e$, our methods do not imply that these monomials are multiples of $u^e$. For example, when $e=4$ and $\varepsilon_\pi=1$, one can only show that
\[\mathbb{L}_p(x_1,x_2,x_3,x_4,u)= au^2(x_1-x_2)(x_3-x_4)+bu^2(x_1-x_3)(x_2-x_4)+cu^4+(\textrm{degree $\geqslant 5$ terms}).\]
Similarly, for $\varepsilon_\pi=1$ and $e=2$ we have $\mathbb{L}_p(x_1,x_2,u) \in (x_1^2x_2^2,u^2)$, while for $e=3$ we have:
\begin{align*}
 \mathbb{L}_p(x_1,x_2,x_3,u) \in (ux_1^2(x_2-x_3),ux_2^2(x_1-x_3), ux_3^2(x_1-x_2),ux_1x_2x_3,u^3).
\end{align*}

\end{remark}

\subsection{Proof of the Trivial Zero Conjecture}
In this section we prove Theorem \ref{t:ezc-hmf}.

\begin{lemma}\label{l:ezc-pf-1}
Keeping the hypotheses and notations of Theorem \ref{t:ezc-hmf}, the analytic function $\mathbb{L}_p((u)_{v\in E},u)$ vanishes at $u=0$ to order at least $e$ and
\[ \frac{(-2)^{e}}{e!} \cdot \frac{d^e}{du^e}\mathbb{L}_p((u)_{v\in E},u)|_{u=0}= \mathscr{L}(\widetilde{\pi}) \cdot
\frac{L(\pi,\tfrac{1-\sw}{2}) }{\mathrm{N}_{F/\Q}^{\sw/2}(i\gd)\Omega_{\widetilde{\pi}}^{\omega_{p,\infty}^{\sw/2}}}\cdot
  2^{|\St_p\!\setminus E|}
 \prod_{v\in S_p\!\setminus \St_p} \left(1-\frac{q_v^{-\sw/2}}{\alpha_v}\right)^2. \]
\end{lemma}
\begin{proof}
By \eqref{eq:improved-LL} we see that
$\mathbb{L}_p((u)_{v\in E}, u)=\mathbb{L}_{\varnothing}(u)\prod\limits_{v\in E}
\Big(1- \alpha_v^\circ(\lambda_{(u),u}^{\varnothing})^{-1}\prod\limits_{\sigma\in
  \Sigma_v}\sigma(\varpi_v)^{(k_\sigma-2)/2} \Big)$.

Since each interpolation factor vanishes at $u=0$ it follows that the order of vanishing of $\mathbb{L}_p((u)_{v\in E},u)$ is at least $e$ at $u=0$.
Differentiating $e$ times at $u=0$ we deduce from \eqref{eq:LL-FE} that
\[\frac{d^e}{du^e}\mathbb{L}_p((u)_{v\in E},u)|_{u=0}=e!L_{S_p\bs E}(\widetilde\pi ,\mathbbm{1} ,\tfrac{2-\sw}{2})\prod_{v\in
 E}\dlog \alpha_v^\circ .\]
 Moreover by Corollary \ref{mtt-interpolation} one has: 
 $E_{S_p\bs E}(\widetilde{\pi}_v,\mathbbm{1},\tfrac{2-\sw}{2})=\begin{cases} \left(1-\alpha_v^{-1}q_v^{-\sw/2}\right)^2 &, \textrm{ if } v\in S_p\bs \St_p, \\ 1+\varepsilon(\pi_v, \tfrac{1-\sw}{2}) & , \textrm{ if } v\in \St_p. \end{cases}$ 

The desired formula then follows from Corollary \ref{mtt-interpolation}
and Proposition \ref{p:FM-derivatives} noting that
$\alpha_v^\circ(\lambda_{(0),0}^{\varnothing})=\prod\limits_{\sigma\in \Sigma_v}\sigma(\varpi_v)^{(k_\sigma-2)/2}$ for $v\in E$ and  
$\varepsilon(\pi_v, \tfrac{1-\sw}{2})=1$ for $v\in \St_p\bs E$.
\end{proof}

\begin{proof}[\bf{Proof of Theorem \ref{t:ezc-hmf}}]
Recall that $\displaystyle L_p(\widetilde{\pi},s) = \langle
\gn\rangle_p^{(2s+\sw-2)/4}\mathbb{L}_p((0)_{v\in E},
2-\sw-2s)$
and so
\begin{align}\label{eq:diff-u}
L^{(m)}_p(\widetilde{\pi},s)|_{s=\tfrac{2-\sw}{2}} =
\sum_{k=0}^m \binom{m}{k} \left(\tfrac{1}{2}\log_p\langle \gn\rangle\right)^{m-k}
(-2)^k \frac{d^k}{d u^k}\mathbb{L}_p((0)_{v\in E}, u)|_{u=0}.
\end{align}

Differentiating $\mathbb{L}_p((0)_{v\in E}, u)$ we see that
$\frac{d^k}{d u^k}\mathbb{L}_p((0)_{v\in E},
u)|_{u=0}=k!A_k((0)_{v\in E})$. 
By Proposition \ref{p:taylor-coeffs} these derivatives vanish for $k<e$, which implies that the order of vanishing of $L_p(\widetilde{\pi},s)$ at $s=\tfrac{2-\sw}{2}$ is at least $e$.
Differentiating the power series expansion of $\mathbb{L}_p((u)_{v\in E},u)$ we see that
\begin{align*}
\frac{d^e}{du^e}\mathbb{L}_p((u)_{v\in E},u)|_{u=0}&=e!\sum_i\sum_{|n|=e-i}a_{i}(n).
\end{align*}
By Proposition \ref{p:taylor-coeffs} we see that the
interior sum above vanishes when $i<e$, hence 
\begin{align*}
\frac{d^e}{du^e}\mathbb{L}_p((u)_{v\in E},u)|_{u=0}&=e!a_{e}((0)_{v\in E})=e!A_{e}((0)_{v\in E})= \frac{d^e}{du^e}\mathbb{L}_p((0)_{v\in E}, u)|_{u=0}.
\end{align*}
Then, \eqref{eq:diff-u} implies that
\[L^{(e)}_p(\widetilde{\pi},s)|_{s=\tfrac{2-\sw}{2}}=(-2)^e\frac{d^e}{d
 u^e}\mathbb{L}_p((0)_{v\in E},
u)|_{u=0}=(-2)^e\frac{d^e}{du^e}\mathbb{L}_p((u)_{v\in E},u)|_{u=0}\]
and Theorem \ref{t:ezc-hmf} then follows from Lemma
\ref{l:ezc-pf-1}.

Finally, it remains to explain why the Trivial Zero Conjecture of the introduction
holds for $\Ind_F^\mathbb{Q} V_\pi(\tfrac{2-\sw}{2})$. Let $V=V_\pi(\tfrac{2-\sw}{2})$ and $D_v\subset \mathcal{D}_{\st}(V_v)$ regular submodules as in \S\ref{ss:greenberg-benois}. Then $D=\bigoplus\limits_{v\in S_p}
\Ind_{F_v}^{\mathbb{Q}_p}D_v\subset \mathcal{D}_{\st}(\Ind_F^\mathbb{Q} V|_{G_{\mathbb{Q}_p}})$ is a regular submodule. By
Proposition \ref{p:GB derivatives} and
\cite[Cor.3.9]{rosso:l-invariants-gsp4}, we have
\[\mathscr{L}_{\GB}(\Ind_F^\mathbb{Q} V_\pi(\tfrac{2-\sw}{2}),
D)=\mathscr{L}_{\GB}(\widetilde{\pi})=\mathscr{L}(\widetilde{\pi}) \prod_{v\in E}f_v^{-1} .\]

The conjecture now follows from the main formula of this theorem and \cite[p.1398]{hida:tate-curves} which explains that the analytic $\mathscr{L}$-invariants of $\Ind_F^\mathbb{Q} V$ and $V$ also differ by $\prod_{v\in E} f_v^{-1}$.
\end{proof}


\end{document}